\documentclass[11pt,reqno]{amsart}
\usepackage{amssymb}
\usepackage{amsfonts}

\numberwithin{equation}{section}
     
\theoremstyle{plain}

\newtheorem{theorem}[subsection]{Theorem}
\newtheorem{proposition}[subsection]{Proposition}
\newtheorem{lemma}[subsection]{Lemma}
\newtheorem{corollary}[subsection]{Corollary}
\newtheorem{claim}[subsection]{Claim}
\theoremstyle{definition}
\newtheorem{definition}[subsection]{Definition}

\newtheorem*{mpet-rpt}{Theorem \ref{mpet}}

\renewcommand{\leq}{\leqslant}
\renewcommand{\geq}{\geqslant}

\newcommand{\md}[1]{\ensuremath{(\operatorname{mod}\, #1)}}

\renewcommand{\mod}{{\ \operatorname{mod}\ }}
\newcommand\E{{\mathbb{E}}}
\newcommand\Z{\mathbb{Z}}
\newcommand\R{\mathbb{R}}

\newcommand\C{\mathbb{C}}
\newcommand\N{\mathbb{N}}
\newcommand\F{\mathbb{F}}

\newcommand\B{\mathcal{B}}

\newcommand\Q{\mathbb{Q}}
\newcommand\gw{\operatorname{gw}}

\newcommand\ab{\operatorname{ab}}

\newcommand\eps{\varepsilon}
\newcommand\GI{\operatorname{GI}}
\newcommand\HP{\operatorname{HP}}

\newcommand\poly{\operatorname{poly}}
\renewcommand\deg{\operatorname{deg}}
\newcommand\id{\operatorname{id}}

\newcommand\str{{\operatorname{str}}}
\newcommand\nil{{\operatorname{nil}}}
\newcommand\sml{{\operatorname{sml}}}
\newcommand\unf{{\operatorname{unf}}}
\newcommand\dist{{\operatorname{dist}}}
\newcommand\Grow{{\mathcal F}}

\begin{document}
\title[Arithmetic regularity and counting lemmas]{An arithmetic regularity lemma, an associated counting lemma, and applications}
\author{Ben Green}
\address{Mathematical Institute\\
Woodstock Road\\
Oxford OX2 6GG\\
England}
\email{ben.green@maths.ox.ac.uk}
\author{Terence Tao}
\address{Department of Mathematics\\
UCLA\\
Los Angeles, CA 90095\\
USA}
\email{tao@math.ucla.edu}

\subjclass{}

\begin{abstract} Szemer\'edi's regularity lemma can be viewed as a rough structure theorem for arbitrary dense graphs, decomposing such graphs into a structured piece (a partition into cells with edge densities), a small error (corresponding to irregular cells), and a uniform piece (the pseudorandom deviations from the edge densities). We establish an \emph{arithmetic regularity lemma} that similarly decomposes bounded functions $f : [N] \rightarrow \C$, into a (well-equidistributed, virtual) $s$-step nilsequence, an error which is small in $L^2$ and a further error which is minuscule in the Gowers $U^{s+1}$-norm, where $s \geq 1$ is a parameter. We then establish a complementary \emph{arithmetic counting lemma} that counts arithmetic patterns in the nilsequence component of $f$.

We provide a number of applications of these lemmas: a proof of Szemer\'edi's theorem on arithmetic progressions, a proof of a conjecture of Bergelson, Host and Kra, and a generalisation of certain results of Gowers and Wolf.

Our result is dependent on the inverse conjecture for the Gowers $U^{s+1}$ norm, recently established for general $s$ by the authors and T.~Ziegler.
\end{abstract}

\maketitle

\begin{center}
\emph{To Endre Szemer\'edi on the occasion of his 70th birthday.}
\end{center}

\tableofcontents

\setcounter{tocdepth}{1}

\section{Introduction}

\emph{Important note added October 2020.} This paper has been revised so that only systems of linear forms satisfying a condition called the \emph{flag property} (see \eqref{flag-property}) are covered by the counting lemma (Lemma \ref{count-lem}). Translation-invariant systems, as well as systems of Cauchy-Schwarz complexity 1, have this property. We thank Daniel Altman for drawing our attention to what appeared to be a minor technical issue in one of our proofs but which ultimately led us to realise that the the  counting lemma fails quite badly (with rather simple examples, which we shall describe) without some assumption of this type.\vspace*{11pt}

Szemer\'edi's celebrated \emph{regularity lemma} \cite{szemeredi-aps, szemeredi-reg} is a fundamental tool in graph theory; see for instance \cite{komlos} for a survey of some of its many applications.  It is often described as a structure theorem for graphs $G = (V,E)$, but one may also view it as a decomposition for arbitrary functions $f: V \times V \to [0,1]$. For instance, one can recast the regularity lemma in the following ``analytic'' form.  Define a \emph{growth function} to be any monotone increasing function $\Grow: \R^+ \to \R^+$ with $\Grow(M) \geq M$ for all $M$.

\begin{lemma}[Szemer\'edi regularity lemma, analytic form]\label{szrl}  Let $V$ be a finite vertex set, let $f: V \times V \to [0,1]$ be a function, let $\eps > 0$, and let $\Grow: \R^+ \to \R^+$ be a growth function.  Then there exists an positive integer\footnote{As usual, we use $O(X)$ to denote a quantity bounded in magnitude by $CX$ for some absolute constant $X$; if we need $C$ to depend on various parameters, we will indicate this by subscripts.  Thus for instance $O_{\eps,\Grow}(1)$ is a quantity bounded in magnitude by some expression $C_{\eps,\Grow}$ depending on $\eps,\Grow$.} $M = O_{\eps,\Grow}(1)$ and a decomposition
\begin{equation}\label{fss}
 f = f_\str + f_\sml + f_\unf
\end{equation}
of $f$ into functions $f_\str, f_\sml, f_\unf: V \times V \to [-1,1]$ such that:

\textup{(}$f_\str$ structured\textup{)} $V$ can be partitioned into $M$ cells $V_1,\ldots,V_M$, such that $f_\str$ is constant on $V_i \times V_j$ for all $i,j$ with $1 \leq i,j \leq M$;

\textup{(}$f_\sml$ small\textup{)} The quantity\footnote{We use here the expectation notation $\E_{a \in A} f(a) := \frac{1}{|A|} \sum_{a \in A} f(a)$ for any finite non-empty set $A$, where $|A|$ denotes the cardinality of $A$.} $\|f_\sml\|_{L^2(V \times V)} := (\E_{v,w \in V} |f_\sml(v,w)|^2)^{1/2}$ is at most $\eps$.

\textup{(}$f_\unf$ very uniform\textup{)}  The box norm  $\| f_\unf\|_{\Box^2(V \times V)} $, defined to be the quantity
$$ (\E_{v_1,v_2,w_1,w_2 \in V} f_\unf(v_1,w_1) f_\unf(v_1,w_2) f_\unf(v_2,w_1) f_\unf(v_2,w_2))^{1/4},$$
is at most $1/\Grow(M)$.

\textup{(}Nonnegativity\textup{)} $f_{\str}$ and $f_\str + f_\sml$ take values in $[0,1]$.
\end{lemma}

Informally, this regularity lemma decomposes any bounded function into a structured part, a small error, and an extremely uniform error.  While this formulation does not, at first sight, look much like the usual regularity lemma, it easily implies that result: see \cite{tao-revisit}.  The idea of formulating the regularity lemma with an arbitrary growth function $\Grow$ first appears in \cite{afks}, and is also very useful for generalisations of the regularity lemma to hypergraphs. See, for example,  \cite{tao-hypergraph}.  The bound on $M$ turns out to essentially be an iterated version of the growth function $\Grow$, with the number of iterations being polynomial in $1/\eps$.  In applications, one usually selects the growth function to be exponential in nature, which then makes $M$ essentially tower-exponential in $1/\eps$.  See \cite{tao-icm,tao-focs} for a general discussion of these sorts of structure theorems and their applications in combinatorics.  See also \cite{lovasz} for a related analytical perspective on the regularity lemma.

In applications the regularity lemma is often paired with a \emph{counting lemma} that allows one to control various expressions involving the function $f$. For example, one might consider the expression
\begin{equation}\label{fff}
\E_{u,v,w \in V} f(u,v) f(v,w) f(w,u),
\end{equation}
which counts triangles in $V$ weighted by $f$.  Applying the decomposition \eqref{fss} splits expressions such as \eqref{fff} into multiple terms (in this instance, $27$ of them).  The key fact, which is a slightly non-trivial application of the Cauchy-Schwarz inequality, is that  the terms involving the box-norm-uniform error $f_\unf$ are negligible if the growth function $\Grow$ is chosen rapidly enough.  The terms involving the small error $f_\sml$ are somewhat small, but one often has to carefully compare those errors against the main term (which only involves $f_\str$) in order to get a non-trivial bound on the final expression \eqref{fff}.  In particular, one often needs to exploit the positivity of $f_\str$ and $f_\str+f_\sml$ to first \emph{localise} expressions such as \eqref{fff} to a small region (such as the portion of a graph between a ``good'' triple $V_i,V_j,V_k$ of cells in the partition of $V$ associated to $f_\str$) before one can obtain a useful estimate.

The graph regularity and counting lemmas can be viewed as the first non-trivial member of a hierarchy of \emph{hypergraph} regularity and counting lemmas, see e.g. \cite{chung,gowers-hypergraph,gowers-hyper-4,nagle-rodl-schacht1,nagle-rodl-schacht,tao-hypergraph}.  The formulation in \cite{tao-hypergraph} is particularly close to the formulation given in Theorem \ref{szrl}.  These lemmas are suitable for controlling higher order expressions such as
$$
\E_{u,v,w,x \in V} f(u,v,w) f(v,w,x) f(w,x,u) f(x,u,v).
$$

Our objective in this paper is to introduce an analogous hierarchy of such regularity and counting lemmas (one for each integer $s \geq 1$), in \emph{arithmetic} situations. Here, the aim is to decompose a function $f: [N] \to [0,1]$ defined on an arithmetic progression $[N] := \{1,\ldots,N\}$ instead of a graph. One is interested in counting averages such as
$$
\E_{n,r \in [N]} f(n) f(n+r) f(n+2r),
$$
which counts $3$-term arithmetic progressions weighted by $f$, as well as higher order expressions such as
$$
\E_{n,r \in [N]} f(n) f(n+r) f(n+2r) f(n+3r).
$$
As it turns out, the former average will be best controlled using the $s=1$ regularity and counting lemmas, while the latter requires the $s=2$ versions of these lemmas.  In this paper we shall see several examples of these types of applications of the two lemmas.\vspace{11pt}

\textsc{The arithmetic regularity lemma.} We begin with by formulating our regularity lemma. Following the statement we explain the terms used here.

\begin{theorem}[Arithmetic regularity lemma]\label{strong-reg}
Let $f: [N] \to [0,1]$ be a function, let $s \geq 1$ be an integer, let $\eps > 0$, and let $\Grow: \R^+ \to \R^+$ be a growth function.  Then there exists a quantity $M = O_{s,\eps,\Grow}(1)$ and a decomposition 
$$ f = f_{\nil} + f_{\sml} + f_{\unf}$$
of $f$ into functions $f_\nil, f_\sml, f_\unf: [N] \to [-1,1]$ of the following form:

\textup{(}$f_\nil$ structured\textup{)} $f_\nil$ is a $(\Grow(M),N)$-irrational virtual nilsequence of degree $\leq s$, complexity $\leq M$, and scale $N$;

\textup{(}$f_\sml$ small\textup{)} $f_\sml$ has an $L^2[N]$ norm of at most $\eps$;

\textup{(}$f_\unf$ very uniform\textup{)} $f_\unf$ has a $U^{s+1}[N]$ norm of at most $1/\Grow(M)$;

\textup{(}Nonnegativity\textup{)} $f_{\nil}$ and $f_\nil + f_\sml$ take values in $[0,1]$.
\end{theorem}

\emph{Remark.} This result easily implies the recently proven \emph{inverse conjecture for the Gowers norms} (Theorem \ref{gis-conj}).  Conversely, this inverse conjecture, together with the equidistribution theory of nilsequences, will be the main ingredient used to prove Theorem \ref{strong-reg}.

We prove this theorem in \S  \ref{regularity-sec}.  We turn now to a discussion of the various concepts used in the above statement.  Readers who are interested in applications may skip ahead to the end of the section.

The $L^2[N]$ norm, used to control $f_\sml$, is simply
$$ \|f\|_{L^2[N]} := (\E_{n \in [N]} |f(n)|^2)^{1/2}.$$

We turn next to the Gowers uniformity norm $U^{s+1}[N]$, used to control $f_\unf$.
If $f : G \rightarrow \C$ is a function on a finite additive group $G$, and $k \geq 1$ is an integer, then the \emph{Gowers uniformity norm} $\|f\|_{U^k(G)}$ is defined by the formula
\[ \Vert f \Vert_{U^{k}(G)} := \big(  \E_{x,h_1,\dots,h_k \in G} \Delta_{h_1} \dots \Delta_{h_k}  f(x) \big)^{1/2^k},\]
where $\Delta_h f: G \to \C$ is the multiplicative derivative of $f$ in the direction $h$, defined by the formula
\[ \Delta_h f(x) := f(x+h) \overline{f(x)}.\] 
In this paper we will be concerned with functions on $[N]$, which is not quite a group. To define the Gowers norms of a function $f : [N] \rightarrow \C$, set $G := \Z/\tilde N\Z$ for some integer $\tilde N \geq 2^k N$, define a function $\tilde f : G \rightarrow \C$ by $\tilde f(x) = f(x)$ for $x = 1,\dots,N$ and $\tilde f(x) = 0$ otherwise, and set $\Vert f \Vert_{U^k[N]} := \Vert \tilde f \Vert_{U^k(G)}/\Vert 1_{[N]}\Vert_{U^k(G)}$, where $1_{[N]}$ is the indicator function of $[N]$. It is easy to see that this definition is independent of the choice of $\tilde N$, and so for definiteness one could take $\tilde N := 2^k N$. 
Henceforth we shall write simply $\Vert f \Vert_{U^k}$, rather than $\Vert f \Vert_{U^k[N]}$, since all Gowers norms will be on $[N]$. One can show that $\Vert \cdot \Vert_{U^k}$ is indeed a norm for any $k \geq 2$, though we shall not need this here; see \cite{gowers-longaps}.  For further discussion of the Gowers norms and their relevance to counting additive patterns see \cite{gowers-longaps}, \cite[\S 5]{green-tao-longprimeaps} or \cite[\S 11]{tao-vu}.


Finally, we turn to the notion of a irrational virtual nilsequence, which is the concept that defines the structural component $f_\nil$.  This is the most complicated concept, and requires a certain number of preliminary definitions.  We first need the notion of a \emph{filtered nilmanifold}. The first two sections of \cite{green-tao-nilratner} may be consulted for a more detailed discussion.

\begin{definition}[Filtered nilmanifold] Let $s \geq 1$ be an integer.  A \emph{filtered nilmanifold} $G/\Gamma = (G/\Gamma, G_\bullet)$ of degree $\leq s$ consists of the following data:

A connected, simply-connected nilpotent Lie group $G$;

A discrete, cocompact subgroup $\Gamma$ of $G$ (thus the quotient space $G/\Gamma$ is a compact manifold, known as a \emph{nilmanifold});

A \emph{filtration} $G_\bullet = (G_{(i)})_{i=0}^\infty$ of closed connected subgroups
$$ G = G_{(0)} = G_{(1)} \geq G_{(2)} \geq \ldots$$
of $G$, which are \emph{rational} in the sense that the subgroups $\Gamma_{(i)} := \Gamma \cap G_{(i)}$ are cocompact in $G_{(i)}$, such that $[G_{(i)},G_{(j)}] \subseteq G_{(i+j)}$ for all $i,j \geq 0$, and such that $G_{(i)}=\{\id\}$ whenever $i>s$;

A \emph{Mal'cev basis}\footnote{A Mal'cev basis is a basis $X_1,\ldots,X_{\dim(G)}$ of the Lie algebra of $G$ that exponentiates to elements of $\Gamma$, such that $X_j,\ldots,X_{\dim(G)}$ span a Lie algebra ideal for all $j \leq i \leq \dim(G)$, and $X_{\dim(G)-\dim(G_{(i)})+1},\ldots,X_{\dim(G)}$ spans the Lie algebra of $G_{(i)}$ for all $1 \leq i \leq s$. For a detailed discussion of this concept, see \cite[\S 2]{green-tao-nilratner}.} $\mathcal{X} = (X_1,\ldots,X_{\dim(G)})$ adapted to $G_{\bullet}$.
\end{definition}

Once a Mal'cev basis has been specified, notions such as the rationality of subgroups may be quantified in terms of it. Furthermore one may use a Mal'cev basis to define a metric $d_{G/\Gamma}$ on the nilmanifold $G/\Gamma$. The results of this paper are rather insensitive to the precise metric that one takes, but one may proceed for example as in \cite[Definition 2.2]{green-tao-nilratner}. We encourage the reader not to think too carefully about the precise definition (or about Mal'cev bases in general), but it is certainly important to have some definite metric in mind so that one can make sense of notions such as that of a \emph{Lipschitz function} on $G/\Gamma$.

Observe that every filtered nilmanifold $G/\Gamma$ comes with a canonical \emph{probability Haar measure} $\mu_{G/\Gamma}$, defined as the unique Borel probability measure on $G/\Gamma$ that is invariant under the left action of $G$.  We abbreviate $\int_{G/\Gamma} F(x)\ d\mu_{G/\Gamma}(x)$ as $\int_{G/\Gamma} F$.

We will need a quantitative notion of \emph{complexity} for filtered nilmanifolds, though once again, the precise definition is somewhat unimportant. 

\begin{definition}[Complexity] Let $M \geq 1$.  We say that a filtered nilmanifold $G/\Gamma = (G/\Gamma,G_\bullet)$ has \emph{complexity $\leq M$} if the dimension of $G$, the degree of $G_\bullet$, and the rationality of the Mal'cev basis $\mathcal{X}$ (cf. \cite[Definition 2.4]{green-tao-nilratner}) are bounded by $M$.
\end{definition}

\emph{Heisenberg example.} The model example of a degree $\leq 2$ filtered nilmanifold is the \emph{Heisenberg nilmanifold}
$$ G/\Gamma := \left( \begin{smallmatrix} 1 & \R & \R \\ 0 & 1 & \R \\ 0 & 0 & 1 \end{smallmatrix} \right)/
\left(\begin{smallmatrix} 1 & \Z & \Z \\ 0 & 1 & \Z \\ 0 & 0 & 1 \end{smallmatrix}\right)$$
with the lower central series $G_{(0)}=G_{(1)} = G$ and
$$ G_{(2)} = [G,G] = \left(\begin{smallmatrix} 1 & 0 & \R \\ 0 & 1 & 0 \\ 0 & 0 & 1 \end{smallmatrix} \right)$$
with Mal'cev basis $\mathcal{X} = \{X_1,X_2,X_3\}$ consisting of the matrices
$$ X_1 = \left(\begin{smallmatrix} 0 & 1 & 0 \\ 0 & 0 & 0 \\ 0 & 0 & 0 \end{smallmatrix}\right), X_2 = \left(\begin{smallmatrix} 0 & 0 & 0 \\ 0 & 0 & 1 \\ 0 & 0 & 0 \end{smallmatrix}\right), X_3 = \left(\begin{smallmatrix} 0 & 0 & 1 \\ 0 & 0 & 0 \\ 0 & 0 & 0 \end{smallmatrix}\right).$$

With the definition of filtered nilmanifold in place, the next thing we need is the idea of a \emph{polynomial sequence}. The basic theory of such sequences was laid out in Leibman \cite{leibman-group-1}, and was extended slightly to general filtrations in \cite{green-tao-nilratner}. An extensive discussion may be found in Section 6 of that paper.

\begin{definition}[Polynomial sequence]  Let $(G/\Gamma,G_\bullet)$ be a filtered nilmanifold, with filtration $G_\bullet = (G_{(i)})_{i=0}^\infty$.  A \emph{\textup{(}multidimensional\textup{)} polynomial sequence} adapted to this filtered nilmanifold is a sequence $g: \Z^D \to G$ for some $D \geq 1$ with the property that 
$$ \partial_{h_1} \ldots \partial_{h_i} g(n) \in G_{(i)}$$
for all $i \geq 0$ and $h_1,\ldots,h_i,n \in \Z^D$, where $\partial_h g(n) := g(n+h) g(n)^{-1}$ is the derivative of $g$ with respect to the shift $h$.  The space of all such polynomial sequences will be denoted $\poly(\Z^D, G_\bullet)$.  The space of polynomial sequences taking values in $\Gamma$ will be denoted $\poly(\Z^D, \Gamma_\bullet)$.  When $D=1$, we refer to multidimensional polynomial sequences simply as \emph{polynomial sequences}.
\end{definition}

\emph{Remark.} We will be primarily interested in the one-dimensional case $D=1$, but will need the higher $D$ case in order to establish the counting lemma, Theorem \ref{count-lem}.

One of the main reasons why we work with polynomial sequences, instead of just linear sequences such as $n \mapsto g_0 g_1^n$, is that the former forms a group.

\begin{theorem}[Lazard-Leibman]\label{ll-thm}  If $(G/\Gamma,G_\bullet)$ is a filtered nilmanifold and $D \geq 1$ is an integer, then $\poly(\Z^D,G_\bullet)$ is a group \textup{(}and $\poly(\Z^D,\Gamma_\bullet)$\textup{)} is a subgroup.
\end{theorem}

\begin{proof} See \cite{leibman-group-2} or \cite[Proposition 6.2]{green-tao-nilratner}.  
\end{proof}

With the concept of a polynomial sequence in hand, it is easy to define a \emph{polynomial orbit}.

\begin{definition}[Orbits]  Let $D, s \geq 1$ be integers, and $M, A > 0$ be parameters.  A  \emph{\textup{(}multidimensional\textup{)} polynomial orbit} of degree $\leq s$ and complexity $\leq M$ is any function\footnote{Strictly speaking, the orbit is the tuple of data $(G, \Gamma, G/\Gamma, G_\bullet, n \mapsto g(n)\Gamma)$, rather than just the sequence $n \mapsto g(n) \Gamma$, but we shall abuse notation and use the sequence as a metonym for the whole orbit.} $n \mapsto g(n) \Gamma$ from $\Z^D \to G/\Gamma$, where $(G/\Gamma,G_\bullet)$ is a filtered nilmanifold of complexity $\leq M$, and $g \in \poly(\Z^D, G_\bullet)$ is a \textup{(}multidimensional\textup{)} polynomial sequence.  
\end{definition}

Using the concept of polynomial orbit, we can define the notion of a (polynomial) nilsequence, as well as a generalisation which we call a \emph{virtual} nilsequence, in analogy with virtually nilpotent groups (groups with a finite index nilpotent subgroup).

\begin{definition}[Nilsequences]
A \emph{\textup{(}multidimensional, polynomial\textup{)} nilsequence} of degree $\leq s$ and complexity $\leq M$ is any function $f: \Z^D \to \C$ of the form $f(n) = F(g(n)\Gamma)$, where $n \mapsto g(n)\Gamma$ is a polynomial orbit of degree $\leq s$ and complexity $\leq M$, and $F: G/\Gamma \to \C$ is a function of Lipschitz norm\footnote{The (inhomogeneous) Lipschitz norm $\|F\|_{\operatorname{Lip}}$ of a function $F: X \to \C$ on a metric space $X = (X,d)$ is defined as
$$\|F\|_{\operatorname{Lip}} := \sup_{x \in X} |F(x)| + \sup_{x,y \in X: x \neq y} \frac{|F(x)-F(y)|}{|x-y|}.$$
} at most $M$.
\end{definition}

\begin{definition}[Virtual nilsequences]
Let $N \geq 1$.  A \emph{virtual nilsequence} of degree $\leq s$ and complexity $\leq M$ at scale $N$ is any function $f: [N] \to \C$ of the form $f(n) = F(g(n)\Gamma, n \md{q}, n/N)$, where $1 \leq q \leq M$ is an integer, $n \mapsto g(n)\Gamma$ is a polynomial orbit of degree $\leq s$ and complexity $\leq M$, and $F: G/\Gamma \times \Z/q\Z \times \R \to \C$ is a function of Lipschitz norm at most $M$.  (Here we place a metric on $G/\Gamma \times \Z/q\Z \times \R$ in some arbitrary fashion, e.g. by embedding $\Z/q\Z$ in $\R/\Z$ and taking the direct sum of the metrics on the three factors.)
\end{definition}

One concept that featured in Theorem \ref{strong-reg} remains to be defined: that of an \emph{irrational} orbit. The definition is a little technical and takes some setting up, and so we defer it and the discussion of some motivating examples to Appendix \ref{appendix-a}. Very roughly speaking, an irrational orbit is one that is equidistributed and for which the filtration $G_{\bullet}$ is as small as possible.  

This concludes our attempt to discuss all the concepts involved in the arithmetic regularity lemma, Theorem \ref{strong-reg}; we turn now to a statement and discussion of the counting lemma.\vspace{11pt}

\textsc{Counting lemma.} In applications of the arithmetic regularity lemma, we will be interested in counting additive patterns such as arithmetic progressions or parallelepipeds.
To understand the phenomena properly it is advantageous to work in a somewhat general setting similar to that taken in \cite{gowers-wolf-1,gowers-wolf-2,gowers-wolf-3,green-tao-linearprimes}. In the latter paper one works with a system $\Psi = (\psi_1,\ldots,\psi_t)$ of integer-coefficient linear forms (or equivalently, group homomorphisms) $\psi_1,\ldots,\psi_t: \Z^D \to \Z$, and consider expressions such as
\begin{equation}\label{enz}
\E_{\mathbf{n} \in \Z^D \cap P} f( \psi_1(\mathbf{n}) ) \ldots f( \psi_t(\mathbf{n}) )
\end{equation}
where $P$ is a convex subset of $\R^D$. 

From now on, when we refer to a system $\Psi$ this is what we mean, and we will assume throughout the paper that none of the forms $\psi_j$ is identically zero.

Thus, for instance, if counting arithmetic progressions, one might use the linear forms
\begin{equation}\label{hp-lattice} 
\psi_i(n_1,n_2) := n_1 + (i-1) n_2; \quad i=1,\ldots,k
\end{equation}
whilst for counting parallelepipeds one might instead use the linear forms
\begin{equation}\label{hk-lattice} 
\psi_{\omega_1,\ldots,\omega_k}(n_0,n_1,\ldots,n_k) := n_0+\omega_1n_1+\ldots+\omega_kn_k; \quad \omega_1,\ldots,\omega_k \in \{0,1\}.\end{equation}

In order to understand the contribution to \eqref{enz} coming from the structured part $f_\nil$ of $f$, one is soon faced with the question of understanding the equidistribution of the orbit
\begin{equation}\label{v-orbit} (g(\psi_1(\mathbf{n})) \Gamma,\dots, g(\psi_t(\mathbf{n})) \Gamma)\end{equation} 
inside $(G/\Gamma)^t$, where $\mathbf{n} = (n_1,\ldots,n_D)$ ranges over $\Z^D \cap P$. We abbreviate this orbit as $g^\Psi({\mathbf n}) \Gamma^t$, where $g^\Psi: \Z^D \to G^t$ is the polynomial sequence
\begin{equation}\label{gpsi-def}
g^\Psi({\mathbf n}) := (g(\psi_1(\mathbf{n})),\dots, g(\psi_t(\mathbf{n}))).
\end{equation}
A very useful model for this question, in which \emph{infinite} orbits were considered in the ``linear'' case $g(n) = g^nx$, was studied by Leibman \cite{leibman-orb-diag}.  His work leads one to the following definition.

\begin{definition}[The Leibman group]\label{leib-gp}  Let $\Psi = (\psi_1,\ldots,\psi_t)$ be a collection of linear forms $\psi_1,\ldots,\psi_t: \Z^D \to \Z$.  For any $i \geq 1$, define $\Psi^{[i]}$ to be the linear subspace of $\R^k$ spanned by the $(\psi_1^i({\bf n}),\ldots,\psi_t^i({\bf n}))$, as ${\bf n}$ ranges over $\Z^D$. Given a filtered nilmanifold $(G/\Gamma, G_\bullet)$, we define the \emph{Leibman group} $G^{\Psi} \lhd G^t$ to be the Lie subgroup of $G^t$ generated by the elements $g_i^{\vec v_i}$ for $i \geq 1$, $g_i \in G_{(i)}$, and $\vec v_i \in {\Psi}^{[i]}$, with the convention that\footnote{We define $g^v$ for real $v$ by the formula $g^v := \exp( v \log(g) )$, where $\exp: {\mathfrak g} \to G$ is the usual exponential map from the Lie algebra ${\mathfrak g}$ to $G$ (this is a homeomorphism since $G$ is nilpotent, connected, and simply connected).}
$$ g^{(v_1,\ldots,v_t)} := (g^{v_1},\ldots,g^{v_t})$$ 
for each $g \in G$.  Note that $G^{\Psi}$ is normal in $G^t$ because $G_{(i)}$ is normal in $G$.
We will show in \S \ref{counting-sec} that $G^{\Psi}$ is also a rational subgroup of $G^t$, thus $\Gamma^{\Psi} := \Gamma^t \cap G^{\Psi}$ is a discrete cocompact subgroup of $G^{\Psi}$.
\end{definition}

We will only be able to work with this definition effectively when $\Psi$ has the \emph{flag property}, by which we mean that 
\begin{equation}\label{flag-property} \Psi^{[1]} \leq \Psi^{[2]} \leq \Psi^{[3]} \leq \cdots \end{equation} This condition is satisfied in many important cases, as follows.
\begin{enumerate}
\item \emph{Translation-invariant systems}, where $\Psi(\R^D)$ contains the vector $\vec{1} = (1,1,\dots, 1)$. These satisfy the flag condition by Corollary \ref{filtr} (ii). This includes both \eqref{hp-lattice} (arithmetic progressions) and \eqref{hk-lattice} (parallelepipeds).
\item Systems of \emph{Gowers-Wolf complexity $1$}, that is to say systems with $\Psi^{[2]} = \R^t$. These satisfy the flag condition by Corollary \ref{filtr} (iii). This includes all systems of Cauchy-Schwarz complexity 1; see Appendix \ref{complexity-app} for basic definitions and proofs concerning these concepts.
\item Certain other systems, for instance rooted parallelepipeds minus $0$, where $\Psi$ consists of the $2^k - 1$ forms $\psi_{\omega}(n_1,\ldots,n_k) := \omega_1n_1+\ldots+\omega_kn_k$ for $\omega \in \{0,1\}^k \setminus \{0\}$, where here $\omega = (\omega_1,\dots, \omega_k)$. The reason for this is that the linear relations between the powers $\psi_{\omega}^i$ are generated by alternating $\pm 1$ sums over $(i+1)$-dimensional faces in $\{0,1\}^k \setminus \{0\}$. Any $(i+2)$-dimensional face is a union of two $(i+1)$-dimensional faces, and so $\Psi^{[i]} \leq \Psi^{[i+1]}$.\vspace*{7pt}
\end{enumerate}

An example\footnote{There is nothing particularly special about this example, but it is the one shown to us by Daniel Altman in order to explain the error in the original version of our paper.}
 of a system $\Psi$ not satisfying the flag property is the following, where $D = 2$, $t = 4$ and ${\bf n} = (n_1, n_2)$:
\begin{equation}\label{non-trans} \psi_1(\mathbf{n}) = n_2, \; \psi_2(\mathbf{n}) = 2n_1 + 2n_2, \; \psi_3(\mathbf{n}) = n_1 + 3n_2, \; \psi_4(\mathbf{n}) = n_1.\end{equation}
Here, $\Psi^{[2]}$ is the hyperplane $\{ (x_1, x_2, x_3, x_4) \in \R^4: 24 x_1 + 3 x_2 - 4x_3 - 8 x_4 = 0\}$, but there are elements in $\Psi^{[1]}$, for instance $(1,2,3,0)$, which do not lie in this hyperplane.\vspace*{10pt}

\emph{Examples. } Let us look more closely at the Leibman group construction corresponding to the two systems \eqref{hp-lattice} and \eqref{hk-lattice} above. In the case of arithmetic progressions, where $\Psi$ is as in \eqref{hp-lattice}, the Leibman group $G^{\Psi}$ is sometimes referred to as the \emph{Hall-Petresco group} $\HP^k(G_{\bullet})$ and has the particularly simple alternative description
\[ \HP^k(G_{\bullet}) = G^{\Psi} = \{ (g(0),\dots, g(k-1)) : g \in \poly(G_{\bullet})\},\] 
We will prove this fact in \S \ref{counting-sec}. 
In the case of parallelepipeds, where $\Psi$ is as in \eqref{hk-lattice}, the Leibman group $G^{\Psi}$ has been referred to as the \emph{Host-Kra cube group} \cite{green-tao-linearprimes} and it too has an alternative description. See \cite[Appendix E]{green-tao-linearprimes} for more information: we will not be making use of this particular group here.

Let $g \in \poly(\Z, G_\bullet)$ be a polynomial sequence, and let $\Psi = (\psi_1,\ldots,\psi_t)$ be a system of linear forms $\psi_1,\ldots,\psi_t: \Z^d \to \Z$ satisfying the flag property.  It turns out (see Lemma \ref{motiv-lemma}) that the sequence $g^\Psi$ takes values in $G^\Psi$.  More remarkably,  the orbit \eqref{v-orbit} is in fact \emph{totally equidistributed} on $G^{\Psi}/\Gamma^{\Psi}$ if $g$ is sufficiently irrational. It is this result that we refer to as our counting lemma.

\begin{theorem}[Counting lemma]\label{count-lem}  Let $M,D,t,s$ be integers with $1 \leq D, t$, $s \leq M$, let $(G/\Gamma,G_\bullet)$ be a degree $\leq s$ filtered nilmanifold of complexity $\leq M$, let $g: \Z \to G$ be an $(A,N)$-irrational polynomial sequence adapted to $G_\bullet$, let $\Psi = (\psi_1,\ldots,\psi_t)$ be a system of linear forms $\psi_1,\ldots,\psi_t: \Z^D \to \Z$ satisfying the flag property \eqref{flag-property}, with coefficients of magnitude at most $M$, and let $P$ be a convex subset of $[-N,N]^{D}$.  Then for any Lipschitz function $F: (G/\Gamma)^t \to \C$ of Lipschitz norm at most $M$, one has\footnote{We use $o_{A \to \infty;M}(X)$ to denote a quantity bounded in magnitude by $c_M(A) X$, where $c_M(A) \to 0$ as $A \to \infty$ for fixed $M$.  Similarly for other choices of subscripts.} 
\begin{align*}
 \sum_{\mathbf{n} \in \Z^D \cap P} F(g^\Psi(\mathbf{n})\Gamma^t) 
 &= \operatorname{vol}(P) \int_{g(0)^\Delta G^{\Psi}/\Gamma^{\Psi}} F \\
 &\quad + o_{A \to \infty;M}( N^D ) + o_{N \to \infty; M}(N^D),
\end{align*}
where $g(0)^\Delta := (g(0),\ldots,g(0)) \in G^t$ and the integral is with respect to the probability Haar measure on the coset \[ g(0)^\Delta G^{\Psi}/\Gamma^{\Psi},\] viewed as a subnilmanifold of $(G/\Gamma)^t$, and $\operatorname{vol}(P)$ is the Lebesgue measure of $P$ in $\R^D$.

More generally, whenever $\Lambda \leq \Z^D$ is a sublattice of index $[\Z^D:\Lambda] \leq M$, and $\mathbf{n_0} \in \Z^D$, one has
\begin{align*}
\sum_{\mathbf{n} \in (\mathbf{n_0}+\Lambda) \cap P} F(g^\Psi(\mathbf{n})\Gamma^t) &= \frac{\operatorname{vol}(P) }{[\Z^D:\Lambda]} \int_{g(0)^\Delta G^{\Psi}/\Gamma^{\Psi}} F\\
&\quad  + o_{A \to \infty;M}( N^D ) + o_{N \to \infty; M}(N^D).
\end{align*}
\end{theorem}

The counting lemma is, of course, best understood by seeing it in action as we shall do several times later on.  The errors $o_{A \to \infty;M}(N^D)$ and $o_{N \to \infty;M}(N^D)$ are negligible in most applications, as $A$ will typically be a huge function $\Grow(M)$ of $M$, and $N$ can also be taken to be arbitrarily large compared to $M$.

We remark that one could easily extend the above lemma to control averages of virtual irrational nilsequences, rather than just irrational sequences, by introducing some additional integrations over the local factors $\Z/q\Z$ and $\R$, but this would require even more notation than is currently being used and so we do not describe such an extension here.

Before turning to applications, let us explain why Theorem \ref{count-lem} can fail when the flag property \eqref{flag-property} does not hold, even in rather simple examples. Take $G = \R^2$, with the filtration $G_{\bullet}$ given by $G_{(0)} = G_{(1)} = \R^2$, $G_{(2)} = \{0\} \times \R$ and $G_{(3)} = G_{(4)} = \cdots = \{(0,0)\}$. Set $\Gamma = \Z^2$, and let $\Psi$ be the collection of forms in \eqref{non-trans}. Let $\alpha, \beta,\gamma$ be irrational numbers, with no low height linear relation between them and $1$. Consider the sequences $g, \tilde g : \Z \rightarrow G$ defined by
\[ g(n) = (\alpha n, \beta n^2), \quad \tilde g(n) = (\alpha n ,\beta n^2 + \theta n).\] Then both sequences are adapted to $G_{\bullet}$ and are $(A, N)$-irrational with $A \rightarrow \infty$ as $N \rightarrow \infty$ (we leave it as an exercise to check this using the definition, Definition \ref{irrat-def}). However, the sequences $g^{\Psi}(\mathbf{n}), \tilde g^{\Psi}(\mathbf{n})$ have significantly different distribution properties in $G^4/\Gamma^4$ as $\mathbf{n}$ ranges over $[N]^2$. 

The former sequence $(g^{\Psi}(\mathbf{n}))_{\mathbf{n} \in [N]^2}$ takes values in the $5$-dimensional torus consisting of points $((x_j, y_j))_{j = 1}^4 \in G^4$ satisfying the relations
\[  x_4 - x_3 = 3x_1, \quad x_2 - 2x_1 = 2x_4, \quad 24 y_1 + 3y_2 - 4y_3 - 8y_4 = 0,\] modulo $\Gamma^4$. This is due to the relation
\[ 24 \psi_1(\mathbf{n})^2 + 3 \psi_2(\mathbf{n})^2 - 4 \psi_3(\mathbf{n})^2 - 8 \psi_4(\mathbf{n})^2 = 0.\]
However, the latter sequence $(\tilde g^{\Psi}(\mathbf{n}))_{\mathbf{n} \in [N]^2}$ does not take values in this space because
\[ 24 \psi_1(\mathbf{n}) + 3 \psi_2(\mathbf{n}) - 4 \psi_3(\mathbf{n}) - 8 \psi_4(\mathbf{n}) \neq 0\] in general. In fact, one may verify using standard results about distribution of polynomial sequences in tori that $(\tilde g^{\Psi}(\mathbf{n}))_{\mathbf{n} \in [N]^2}$ becomes, as $N \rightarrow \infty$, equidistributed in the $6$-dimensional torus
\[  \{ ((x_j, y_j))_{j = 1}^4 \in G^4 : x_4 - x_3 = 3x_1, x_2 - 2x_1 = 2x_4\}/\Gamma^4.\] 
This example shows that any variant of Theorem \ref{count-lem} applying to non translation-invariant systems $\Psi$ would have to take account of more subtle information about the sequence $g : \Z \rightarrow G$ than irrationality. This would (if it is possible) require a wholesale reworking of the proof. \vspace{11pt}

\textsc{Applications.} The proofs of the regularity and counting lemmas occupy about half the paper. In the remaining half, we give a number of applications of these results to problems in additive combinatorics.  The scheme of the arguments in all of these cases is similar.  First, one applies the arithmetic regularity lemma to decompose the relevant function $f$ into structured, small, and (very) uniform components $f = f_\nil + f_\sml + f_\unf$. Very roughly speaking, these are analysed as follows:

$f_\nil$ is studied using algebraic properties of nilsequences, particularly the counting lemma;

$f_\sml$ is shown to be negligible, though often (unfortunately) some additional algebraic input is required to ensure that this error does not conspire to destroy the contribution from $f_\nil$;

$f_\unf$ is easily shown to be negligible using results of ``generalised von Neumann'' type as discussed in \S \ref{gvn-sec}.

As we shall see, dealing with the error $f_\sml$ can cause a certain amount of pain. To show that this error is truly negligible, one often has to prove that patterns guaranteed by $f_\nil$ (such as arithmetic progressions) do not concentrate on some small set which might be contained in the support of $f_\sml$.

We now give specific examples of this paradigm.  In \S \ref{szem-sec} we give a ``new'' proof of Szemer\'edi's famous theorem on arithmetic progressions. This is hardly exciting nowadays, with at least 14 proofs already in the literature \cite{austin-multdimsz, austin-dhj,furstenberg-book,fko, gowers-longaps,gowers-hypergraph,nagle-rodl-schacht,polymath,szemeredi-aps,tao-quant-ergodic,tao-hypergraph} as well as (slightly implicitly) in \cite{bergelson-host-kra,host-kra,ziegler}. However this proof makes the point that for a certain class of problems it suffices to ``check the result for nilsequences'', and in so doing one really sees the structure of the problem. Just as random and structured graphs are two obvious classes to test conjectures against in graph theory, we would like to raise awareness of nilsequences as potential (and, in certain cases such as this one, the \emph{only}) sources of counterexamples.

The second application, proven in \S \ref{bhk-sec}, is to establish a conjecture of Bergelson, Host and Kra \cite{bergelson-host-kra}. Here and in the sequel we use the notation $X \ll_{\alpha,\eps} Y$ or $Y \gg_{\alpha,\eps} X$ synonymously with $X = O_{\alpha,\eps}(Y)$, and similarly for other choice of subscripts.

\begin{theorem}[Bergelson-Host-Kra conjecture]\label{bhk-thm}  Let $k = 1, 2, 3$ or $4$, and suppose that $0 < \alpha < 1$ and $\eps > 0$.  Then for any $N \geq 1$ and any subset $A \subseteq [N]$ of density $|A| \geq \alpha N$, one can find $\gg_{\alpha,\eps} N$ values of $d \in [-N,N]$ such that there are at least $(\alpha^k - \eps)N$ $k$-term arithmetic progressions in $A$ with common difference $d$.
\end{theorem}
\emph{Remarks.} The claim is trivial for $k=1$, and follows from an easy averaging argument when $k=2$.  This theorem was established in the case $k = 3$ by the first author in \cite{green-regularity}: we give a new proof of this result which may be of independent interest.  The case $k=4$ is new, although a finite field analogue of this result previously appeared in lecture notes of the first author \cite{green-montreal} (reporting on joint work).  A counterexample example of Ruzsa in the appendix to \cite{bergelson-host-kra} shows that Theorem \ref{bhk-thm} fails when $k \geq 5$.

Finally, in \S \ref{gowers-wolf-sec}, we establish a generalisation of a recent result of Gowers and Wolf \cite{gowers-wolf-1,gowers-wolf-2,gowers-wolf-3} regarding the ``true'' complexity of a system of linear forms.

\begin{theorem}\label{gwolf}  Let $\Psi = (\psi_1,\ldots,\psi_t)$ be a collection of linear forms from $\Z^D$ to $\Z$ satisfying the flag condition, and let $s \geq 1$ be an integer such that the polynomials $\psi_1^{s+1},\ldots,\psi_t^{s+1}$ are linearly independent.  Then for any function $f: [N] \to \C$ bounded in magnitude by $1$ \textup{(}and defined to be zero outside of $[N]$\textup{)} obeying the bound $\|f\|_{U^{s+1}[N]} \leq \delta$ for some $\delta > 0$, one has
$$ \E_{\mathbf{n} \in [N]^D} \prod_{i=1}^t f( \psi_i(\mathbf{n}) ) = o_{\delta \to 0; s, D,t,\Psi}(1).$$
\end{theorem}

\emph{Remarks.}  This result was conjectured in \cite{gowers-wolf-1} (without the flag property assumption), where it was shown that the linear independence hypothesis was necessary.  The programme in \cite{gowers-wolf-1,gowers-wolf-2,gowers-wolf-3} gives an alternate approach to this result that avoids explicit mention of nilsequences, and in particular establishes the counterpart to Theorem \ref{gwolf} in finite characteristic; their work also gives a proof of this theorem in the case when the Cauchy-Schwarz complexity of the system (see Theorem \ref{cs-lemma}) is at most two\footnote{Note that, although not all systems of Cauchy-Schwarz complexity two have the flag property, the results of Gowers and Wolf are only nontrivial when the Gowers-Wolf complexity of the system is one, that is to say $\Psi^{[2]} = \R^t$. In this case the flag property does hold. Therefore our result includes the result of Gowers and Wolf, albeit with worse bounds.}, with better bounds than our result, which is all but ineffective.  It is worth mentioning that the arguments in \cite{gowers-wolf-1,gowers-wolf-2,gowers-wolf-3} also develop several structural decomposition theorems along the lines of Theorem \ref{strong-reg}, but using the language of locally polynomial phases rather than nilsequences. 

The general case of \cite[Conjecture 2.5]{gowers-wolf-1}, in which the system $\Psi$ does not satisfy the flag property, remains open. \vspace{11pt}

\textsc{Relation to previous work.} A result closely related to Theorem \ref{strong-reg} in the case $s = 1$ was proved by Bourgain as long ago as 1989 \cite{bourgain}. In that paper, the decomposition was applied to give a different proof of \emph{Roth's theorem}, that is to say Szemer\'edi's theorem for $3$-term progressions. A different take on this result was supplied by the first author in \cite{green-regularity}, where the application to the case $k = 3$ of the Bergelson-Host-Kra conjecture was noted. In that same paper a construction of Gowers \cite{gowers-lower} was modified to show that any application of the arithmetic regularity lemma must lead to awful (tower-type) bounds; the same kind of construction would show that the cases $s \geq 2$ of Theorem \ref{strong-reg} lead to tower-type bounds as well. In\footnote{The relevant part of these lecture notes by the first author reported on joint work of the two of us.} \cite{green-montreal} the analogue of the case $s = 2$ of Theorem \ref{strong-reg} in a finite field setting was stated, proved, and used to deduce the finite field analogue of the Bergelson-Host-Kra conjecture in the case $k = 4$. In that same paper the present work was promised (as reference [22]) at ``some future juncture''. Four years later we have reached that juncture and we apologise for the delay. We note, however, that until the very recent resolution of the inverse conjectures for the Gowers norms \cite{green-tao-ziegler-u4inverse,uk-inverse} many of our results would have been conditional; furthermore, we are heavily dependent on our work \cite{green-tao-nilratner}, which had not been envisaged when the earlier promise was made.

In the meantime a greater general understanding of decomposition theorems of this type has developed through the work of Gowers \cite{gowers-regularity}, Reingold-Trevisan-Tulsiani-Vadhan \cite{rttv}, and Gowers-Wolf \cite{gowers-wolf-1,gowers-wolf-2,gowers-wolf-3}; see also the survey \cite{tao-focs} of the second author. While Theorem \ref{strong-reg} is related to several of these general decomposition theorems, it also relies upon specific structure of nilmanifolds. In any case it seems appropriate, in this volume, to give a proof using the ``energy increment argument'' pioneered by Szemer\'edi.

\textsc{Acknowledgments.}  BG was, while this work was being carried out, a fellow at the Radcliffe Institute at Harvard. He is very happy to thank the Institute for proving excellent working conditions.
TT is supported by a grant from the MacArthur Foundation, by NSF grant DMS-0649473, and by the NSF Waterman award.

\section{Proof of the arithmetic regularity lemma}\label{regularity-sec}

We now prove Theorem \ref{strong-reg}.  The proof proceeds in two main stages.  Firstly, we establish a ``non-irrational regularity lemma'', which establishes a weaker version of Theorem \ref{strong-reg} in which the structured component $f_\nil$ is a polynomial nilsequence, but one which is not assumed to be irrational.  The main tool here is the \emph{inverse conjecture $\GI(s)$ for the Gowers norms} \cite{uk-inverse}, combined with the energy incrementation argument that appears in proofs of the graph regularity lemma.  In the second stage, we upgrade this weaker regularity lemma to the full regularity lemma by converting the nilsequence to a irrational nilsequence.  The main tool here is a dimension reduction argument and a factorisation of nilsequences similar to that appearing in \cite{green-tao-nilratner}.

\textsc{The non-irrational regularity lemma.} We begin the first stage of the argument.  As mentioned above, the key ingredient is the following result.  

\begin{theorem}[$\GI(s)$]\label{gis-conj}
Let $s \geq 1$, and suppose that $f : [N] \rightarrow \C$ is a function bounded in magnitude by $1$ such that $\Vert f \Vert_{U^{s+1}[N]} \geq \delta$ for some $\delta>0$. Then there is a degree $\leq s$ polynomial nilsequence $\psi: \Z \to \C$ of complexity $O_{s,\delta}(1)$ such that $|\langle f, \psi \rangle_{L^2[N]}| \gg_{s,\delta} 1$, where 
$$ \langle f, \psi\rangle_{L^2[N]} := \E_{n \in [N]} f(n) \overline{\psi(n)}$$
is the usual inner product.
\end{theorem}

\emph{Remark.} The difficulty of this conjecture increases with $s$.  The case $s=1$ easily follows from classical harmonic analysis.  The case $s=2$ was established by the authors in \cite{green-tao-u3inverse}, building upon the breakthrough paper of Gowers \cite{gowers-4aps}.  The case $s=3$ was recently established by the authors and Ziegler in \cite{green-tao-ziegler-u4inverse}, and the general case will appear in the forthcoming paper \cite{uk-inverse} by the authors and Ziegler.

For technical reasons, it is convenient to replace the notion of a degree $\leq s$ polynomial nilsequence by a slightly different concept.  The following definition is not required beyond the end of the proof of Proposition \ref{semi-reg}.

\begin{definition}[$s$-measurability]\label{s-meas-def}  Let $\Phi: \R^+ \to \R^+$ be a growth function and $s \geq 1$.  A subset $E \subseteq [N]$ is said to be \emph{$s$-measurable} with growth function $\Phi$ if for every $M \geq 1$, there exists a degree $\leq s$ polynomial nilsequence $\psi: \Z \to [0,1]$ of complexity $\leq \Phi(M)$ such that
$$ \| \psi - 1_E \|_{L^2[N]} \leq 1/M.$$
\end{definition}

An example of a $1$-measurable function would be a regular Bohr set, as introduced in \cite{bourgain-triples} and discussed further in \cite[\S 2]{green-tao-u3inverse}. We will not need Bohr sets elsewhere in this paper, so we shall not dwell any longer on this example. However the reader will see ideas related to the basic theory of those sets in the proof of Corollary \ref{corda} below.

We make the simple but crucial observation that if $E, F$ are $s$-measurable with some growth functions $\Phi, \Phi'$ respectively, then boolean combinations of $E, F$ such as $E \cap F$, $E \cup F$, or $[N] \backslash E$ are also $s$-measurable with some growth function depending on $\Phi,\Phi'$. Underlying this, of course, is that fact that the product and sum of two nilsequences is also a nilsequence, and hence the set of nilsequences form a kind of algebra (graded by complexity). The role of algebraic structure of this kind was brought to the fore in the work of Gowers \cite{gowers-regularity} cited above.

Theorem \ref{gis-conj} then implies

\begin{corollary}[Alternate formulation of $\GI(s)$]\label{corda} Let $s \geq 1$, and suppose that $f : [N] \rightarrow [-1,1]$ is such that $\Vert f \Vert_{U^{s+1}[N]} \geq \delta$ for some $\delta>0$.   Then there exists a growth function $\Phi_{s,\delta}$ depending only on $s,\delta$, and an $s$-measurable set $E \subset N$ with growth function $\Phi_{s,\delta}$, such that
$$ |\E_{n \in [N]} f(n) 1_E(n)| \gg_{s,\delta} 1.$$
\end{corollary}

\begin{proof}  We allow implied constants to depend on $s,\delta$.  By Theorem \ref{gis-conj}, there exists a degree $\leq s$ polynomial nilsequence $\psi$ of complexity $O(1)$ such that
$$ |\E_{n \in [N]} f(n) \overline{\psi(n)}| \gg 1.$$
By taking real and imaginary parts of $\psi$, and then positive and negative parts, and rescaling, we may assume without loss of generality that $\psi$ takes values in $[0,1]$.  By Fubini's theorem, we then have
$$ |\int_0^1 \E_{n \in [N]} f(n) 1_{E_t}(n)\ dt| \gg 1$$
where $E_t := \{ n \in [N]: \psi(n) \geq t \}$.  We thus see that there is a subset $\Omega \subset [0,1]$ of Lebesgue measure $|\Omega| \gg 1$ such that 
$$ |\E_{n \in [N]} f(n) 1_{E_t}(n)| \gg 1$$
uniformly for all $t \in \Omega$.

It remains to show that at least one\footnote{Here we are, in some sense, finding a ``regular'' nil-Bohr set $\{n \in [N] : \psi(n) \geq t\}$, that is to say one rather insensitive to small changes in the value of $t$.  A similar idea also appears in \cite[Claim 2.2]{rttv}.} of the $E_t$ is $s$-measurable with respect to a suitable growth function.  For any $t \in \R$, we consider the maximal function
$$ M(t) := \sup_{r>0} \frac{1}{2r} \frac{1}{N} |\{ n \in [N]: |\psi(n)-t| \leq r \}|.$$
From the Hardy-Littlewood maximal inequality or the Besicovitch covering lemma we have that the set $\{ t \in \R: M(t) \ge \lambda \}$ has Lebesgue measure $O(1/\lambda)$ for any $\lambda > 0$.  Thus, we can find $t \in \Omega$ such that $M(t)= O(1)$.  Fixing such a $t$, we then see that 
$$ |\{ n \in [N]: |\psi(n)-t| \leq r \}| \ll rN$$
for all $r>0$.  As a consequence, for any $r>0$, one can then approximate $1_{E_t}$ to within $O( \sqrt{r} )$ in $L^2[N]$ norm by a Lipschitz function of $\psi$ with Lipschitz norm $O(1/r)$.  This implies that $1_{E_t}$ is $s$-measurable with some growth function $\Phi$ depending only on $s,\delta$, and the claim follows.
\end{proof}

We rephrase this fact in terms of conditional expectations. The following definition, like Definition \ref{s-meas-def}, will only be needed until the end of the proof of Proposition \ref{semi-reg}.

\begin{definition}[$s$-factors] An \emph{$s$-factor} $\B$ of complexity $\leq M$ and growth function $\Phi$ is a partition of $[N]$ into at most $M$ sets (or \emph{cells}) $E_1,\ldots,E_m$ which are $s$-measurable of growth function $\Phi$.  Given an $s$-factor $\B$ and a function $f: [N] \to \C$, we define the \emph{conditional expectation} $\E(f|\B): [N] \to \C$ of $f$ with respect to the $s$-factor to be the function which equals $\E_{n \in E_j} f(n)$ on each cell of the partition.  We define the \emph{index} or \emph{energy} ${\mathcal E}(\B)$ of the $s$-factor $\B$ relative to $f$ to be the quantity $\|\E(f|\B)\|_{L^2[N]}^2$.  

An $s$-factor $\B'$ is said to \emph{refine} another $\B$ if every cell of $\B'$ is contained in a cell of $\B$.  
\end{definition}

\begin{corollary}[Lack of uniformity implies energy increment]\label{energy-inc}  
Let $s \geq 1$, let $\B$ be an $s$-factor of complexity $\leq M$ and some growth function $\Phi$, and suppose that $f : [N] \rightarrow [0,1]$ is such that $\Vert f - \E(f|\B) \Vert_{U^{s+1}[N]} \geq \delta$ for some $\delta>0$.   Then there exists a refinement $\B'$ of $\B$ of complexity $\leq 2M$ and some growth function depending on $s,\delta,M,\Phi$, such that
$$ {\mathcal E}(\B') - {\mathcal E}(\B) \gg_{s,\delta} 1.$$
\end{corollary}

\begin{proof}  By Corollary \ref{corda}, we can find an $s$-measurable set $E$ with a growth function depending on $s,\delta$ such that
\begin{equation}\label{fee}
|\langle f - \E(f|\B), 1_E \rangle_{L^2[N]}| \gg_{s,\delta} 1
\end{equation}
Now let $\B'$ be the partition generated by $\B$ and $E$; then $\B'$ clearly has complexity $\leq 2M$ and a growth function depending on $s,\delta,M,\Phi$.  Since $1_E$ is measurable with respect to the partition $\B'$ (that is to say it is constant on each cell of this partition), we can rewrite the left-hand side of \eqref{fee} as
$$ |\langle \E(f|\B') - \E(f|\B), 1_E \rangle_{L^2[N]}|$$
and hence by the Cauchy-Schwarz inequality
$$ \| \E(f|\B') - \E(f|\B) \|_{L^2[N]} \gg_{s,\delta} 1.$$
The claim then follows from Pythagoras' theorem.
\end{proof}

We can iterate this to obtain a weak regularity lemma, analogous to the weak graph regularity lemma of Frieze and Kannan \cite{frieze}.

\begin{corollary}\label{weak-reg}  
Let $s \geq 1$, let $\B$ be an $s$-factor of complexity $\leq M$ and some growth function $\Phi$, let $f : [N] \rightarrow [0,1]$, and let $\eps > 0$.  Then there exists a refinement $\B'$ of $\B$ of complexity $O_{s,M,\eps}(1)$ and some growth function depending on $s,\eps,M,\Phi$, such that
\begin{equation}\label{fF}
\| f - \E(f|\B') \|_{U^{s+1}[N]} \leq \eps.
\end{equation}
\end{corollary}
\begin{proof} We define a sequence of successively more refined factors $\B'$, starting with $\B' := \B$.  If \eqref{fF} already holds then we are done, so suppose that this is not the case.  Then by Corollary \ref{energy-inc}, we can find a refinement $\B''$ of complexity $O_{s,M,\eps}(1)$ and some growth function depending on $s,\eps,M,\Phi$ whose energy is larger than that of $\B'$ by a factor $\gg_{s,\eps} 1$.  On the other hand, the energy clearly ranges between $0$ and $1$.  Thus after replacing $\B'$ with $\B''$ and iterating this algorithm at most $O_{s,\eps}(1)$ times we obtain the claim.
\end{proof}

One final iteration then gives the full non-irrational regularity lemma.

\begin{proposition}\label{semi-reg}  Let $f: [N] \to [0,1]$, let $s \geq 1$, let $\eps > 0$, and let $\Grow: \R^+ \to \R^+$ be a growth function.  Then there exists a quantity $M = O_{s,\eps,\Grow}(1)$ and a decomposition 
$$ f = f_\nil + f_\sml + f_\unf$$
of $f$ into functions $f_\nil, f_\unf: [N] \to [-1,1]$ such that:

\textup{(}$f_\nil$ structured\textup{)} $f_\nil$ equals a degree $\leq s$ polynomial nilsequence of complexity $\leq M$.

\textup{(}$f_\sml$ small\textup{)} $\| f_\sml \|_{L^2[N]} \leq \eps$.

\textup{(}$f_\unf$ very uniform\textup{)} $\|f_\nil\|_{U^{s+1}[N]} \leq 1/\Grow(M)$.

\textup{(}Nonnegativity\textup{)} $f_\nil$ and $f_\nil + f_\sml$ take values in $[0,1]$.
\end{proposition}
\begin{proof}  We need a growth function $\tilde \Grow: \R^+ \to \R^+$, somewhat more rapidly growing than $\Grow$ in manner that depends on $\Grow$, $s$, $\eps$. We will specify the exact requirements we have of it later.  We then define a sequence $1 = M_0 \leq M_1 \leq \ldots$ by setting $M_0 := 1$ and $M_{i+1} := \tilde \Grow(M_i)$.

Applying Corollary \ref{weak-reg} repeatedly, we may find for each $i \geq 0$ an $s$-factor $\B_i$ of complexity $O_{s,M_i}(1)$ and a growth function depending on $s,M_i$, such that each $\B_i$ refines $\B_{i-1}$, and such that
$$ \| f - \E(f|\B_i) \|_{U^{s+1}[N]} \leq 1/M_i$$
for all $i \geq 0$.

By Pythagoras' theorem, the energies ${\mathcal E}(\B_i)$ are non-decreasing, and also range between $0$ and $1$.  Thus by the pigeonhole principle, one can find $i = O_\eps(1)$ such that
$$ {\mathcal E}(\B_{i+1}) - {\mathcal E}(\B_{i}) \leq \eps^2/4,$$
which by Pythagoras' theorem again is equivalent to
$$ \| \E(f|\B_{i+1}) - \E(f|\B_{i}) \|_{L^2[N]} \leq \eps/2.$$
Meanwhile, as $\B_i$ is an $s$-factor and $f$ is bounded, we can find a degree $\leq s$ polynomial nilsequence $f_\nil: [N] \to \R$ of complexity $O_{s,M_i}(1)$ such that
$$ \| \E(f|\B_{i}) - f_\nil \|_{L^2[N]} \leq \eps/2.$$
Since $\E(f|\B_{i})$ ranges in $[0,1]$, we may retract $f_\nil$ to $[0,1]$ also (note that this does not increase the complexity of $f_\nil$).  If we then set $f_\unf := f -  \E(f|\B_{i+1})$ and $f_\sml :=  \E(f|\B_{i+1}) - f_\nil$, we obtain the claim.
\end{proof}

\emph{Remark.} The application of the Hardy-Littlewood maximal inequality in the proof of Corollary \ref{corda} makes for a reasonably tidy argument. A more direct approach would be to carve up $[N]$ into approximate level sets of nilsequences, and then to approximate the projections onto the factors thus defined by nilsequences using the Weierstrass approximation theorem. There are a number of technicalities involved in this approach, chiefly involving the need to choose the approximate level sets randomly. This kind of argument was employed, in a closely related context, in \cite[Chapter 7]{green-tao-longprimeaps}.  One can also use utilise arguments based on the Hahn-Banach theorem instead; see \cite{gowers-regularity}, \cite{rttv}, and \cite{gowers-wolf-1,gowers-wolf-2,gowers-wolf-3}.\vspace{11pt}

\textsc{Obtaining irrationality.} Our task now is to replace the nilsequence $f_{\nil}$ appearing in Proposition \ref{semi-reg} with a highly ``irrational'' nilsequence as advertised in the statement of our main theorem, Theorem \ref{strong-reg}. It turns out to be sufficient to establish the following claim.

\begin{proposition}\label{totalis}  Let $s, M_0 \geq 1$, let $\Grow$ be a growth function, and let $f: \Z \to [0,1]$ be a degree $\leq s$ nilsequence of complexity $\leq M_0$.  Then there exists an $M = O_{s,M_0, \Grow}(1)$, such that $f$ \textup{(}when restricted to $[N]$\textup{)} is also a $(\Grow(M),N)$-irrational degree $\leq s$ virtual nilsequence of complexity $\leq M$ at scale $N$.\end{proposition}

To establish Theorem \ref{strong-reg} from this and Proposition \ref{semi-reg} one first applies the latter result with $\Grow$ replaced by a much more rapid growth function $\Grow'$, and then one applies Proposition \ref{totalis} to the structured component $f_\nil$ obtained in Theorem \ref{weak-reg}.  

It remains to prove Proposition \ref{totalis}.  Let $s, M_0,\Grow,\psi$ be as in that proposition.  By definition, we have $\psi = F_0( g_0(n) \Gamma )$ for some degree $\leq s$ filtered nilmanifold $(G/\Gamma,G_\bullet)$ of complexity $\leq M_0$, a polynomial sequence $g_0 \in \poly(\Z,G_\bullet)$, and a function $F_0: G/\Gamma \to \C$ which has a Lipschitz norm of at most $M_0$.   Since $\psi$ takes values in $[0,1]$, we may assume without loss of generality that $F_0$ is real, and by replacing $F_0$ with the retraction $\max(\min(F_0,1),0)$ to $[0,1]$ if necessary, we may assume that $F_0$ also takes values in $[0,1]$.  Henceforth $(G/\Gamma,G_\bullet)$, $g_0$, and $F_0$ are fixed.\vspace{11pt}

\textsc{Factorisation results.} One of the main results of our paper \cite{green-tao-nilratner} was a decomposition of an arbitrary polynomial nilsequence $g$ on $G/\Gamma$ into a product\footnote{In our paper \cite{green-tao-nilratner} the letter $\eps$ was used for a smooth nilsequence, but we use $\beta$ here to avoid conflict with various uses of $\eps$ to denote a small positive real number.} $\beta g' \gamma$, where $\beta$ is ``smooth'', $\gamma$ is ``rational'', and $g'(n)\Gamma$ is equidistributed inside some possibly smaller nilmanifold $G'/\Gamma'$. We need a similar result here, but with $g'$ having the somewhat stronger property of being \emph{irrational} that we mentioned in the introduction. The notion of irrationality is discussed in more detail in Appendix \ref{appendix-a}.

We will be also using the notions of \emph{smooth} and \emph{rational} polynomial sequences from \cite{green-tao-nilratner}. Again, the basic definitions and properties of these concepts are recalled in Appendix \ref{appendix-a}.  

Define a \emph{complexity $\leq M$ subnilmanifold} of $(G/\Gamma,G_\bullet)$ to be a degree $\leq s$ filtered nilmanifold $(G'/\Gamma',G'_\bullet)$ of complexity $\leq M$, where each subgroup $G'_{(i)}$ in the filtration $G'_\bullet$ is a rational subgroup of the associated subgroup $G_{(i)}$ of complexity $\leq M$, $\Gamma' = G' \cap \Gamma$, and each element of the Mal'cev basis of $(G'/\Gamma',G'_\bullet)$ is a rational linear combination of the Mal'cev basis of $(G/\Gamma,G_\bullet)$, where the coefficients all have height $\leq M$.  We define the \emph{total dimension} of such a nilmanifold to be the quantity $\sum_{i=0}^s \dim(G'_{(i)})$; this is also the dimension of $\poly(\Z,G_\bullet)$ (thanks to the Taylor series expansion, Lemma \ref{taylor-lem}).  

We make the easy remark that if $(G'/\Gamma',G'_\bullet)$ is a complexity $\leq M$ subnilmanfiold of $(G/\Gamma,G_\bullet)$ for some $M \geq M_0$, and $(G''/\Gamma'',G''_\bullet)$ is a complexity $\leq M$ subnilmanifold of $(G'/\Gamma',G'_\bullet)$, then $(G''/\Gamma'',G''_\bullet)$ is a complexity $O_M(1)$ subnilmanifold of $(G/\Gamma,G_\bullet)$.

Our first lemma is very similar in form to \cite[Lemma 7.9]{green-tao-nilratner}.

\begin{lemma}[Initial factorisation]\label{init}  Let $(G'/\Gamma',G'_\bullet)$ be a complexity $\leq M$ subnilmanifold of $(G/\Gamma,G_\bullet)$ for some $M \geq M_0$, let $g' \in \poly(\Z, G'_\bullet)$, and let $A > 0$ and $N \geq 1$.  Then at least one of the following statements hold:

\textup{(}Irrationality\textup{)} $g'$ is $(A,N)$-irrational in $(G'/\Gamma',G'_\bullet)$.

\textup{(}Dimension reduction\textup{)}  There exists a factorisation
$$ g' = \beta g'' \gamma$$
where $\beta \in \poly(\Z,G'_\bullet)$ is $(O_{M,A}(1),N)$-smooth, $g'' \in \poly(\Z,G''_\bullet)$ takes values in a subnilmanifold $(G''/\Gamma'',G''_\bullet)$ of $(G'/\Gamma',G'_\bullet)$ of strictly smaller total dimension and of complexity $O_{M,A}(1)$, and $\gamma \in \poly(\Z,G'_\bullet)$ is $O_{M,A}(1)$-rational.
\end{lemma}
\begin{proof} To make this proof a little more readable, we drop one dash from every expression. Thus $g'$ becomes $g$, $G''$ becomes $G'$, and so on. Suppose that $g$ is not $(A,N)$-irrational.  Recall (see Lemma \ref{taylor-lem}) that $g$ has a Taylor expansion that we may write in the form
\[ g(n) = g_0 g_1^{ \binom{n}{1}} g_2^{\binom{n}{2}} \dots g_s^{\binom{n}{s}},\] 
where $g_i \in G_{(i)}$ for each $i$. It follows from Lemma \ref{short-factorisation} that for some $i$, $1 \leq i \leq s$, we can factorise
$$ g_i = \beta_i g'_i \gamma_i,$$
where $g'_i \in G_{(i)}$ lies in the kernel of some horizontal character $\xi_i : G_{(i)} \rightarrow \R$ of complexity $O_{A,M}(1)$, $\gamma_i \in G_{(i)}$ is $O_{A,M}(1)$-rational in the sense that $\gamma_i^m \in \Gamma_{(i)}$ for some $m = O_{A,M}(1)$, and $\beta_i \in G_{(i)}$ has distance $O_{A,M}(1/N^i)$ from the origin.

We now divide into two cases, depending on whether $i>1$ or $i=1$.  First suppose that $i>1$.  Then the Taylor expansion of $g$ reads, with an obvious notation,
$$ g(n) = g_{< i}(n)(\beta_{i} g'_{i} \gamma_{i})^{\binom{n}{i}} g_{>i}(n).$$ By commutating all the $\beta_i$s to the left and all the $\gamma_i$s to the right, and using the group properties of polynomial sequences (Theorem \ref{ll-thm}), one can rewrite this as
$$ g(n) = \beta_{i}^{\binom{n}{i}} g'(n) \gamma_{i}^{\binom{n}{i}}$$
where
$$ g'(n) := g_{< i}(n) g_i^{\prime \binom{n}{i}} \tilde g_{>i}(n)  $$
and $\tilde g_{>i}(n)$ is another polynomial sequence taking values in $G_{(i+1)}$.  Observe that $g'$ is then a polynomial sequence adapted to the subnilmanifold $(G'/\Gamma',G'_\bullet)$, where $G'/\Gamma' = G/\Gamma$ and $G'_{(j)} = G_{(j)}$ for $j \neq i$, but $G'_{(i)} = \ker (\xi'_i)$. This is indeed a subnilmanifold, with complexity $O_{A,M}(1)$; note that $(G'_{(l)})_{l = 0}^{\infty}$ is a filtration, thanks to our insistence in the definition of $i$-horizontal character (cf. Definition \ref{irrat-def}) that $[G_{(j)}, G_{(i-j)}] \subseteq \ker(\xi'_i)$ for all $0 \leq j \leq i$.
Meanwhile, $\beta_{i}^{\binom{n}{i}}$ is a $(O_{A,M}(1),N)$-smooth sequence and $\gamma_{i}^{\binom{n}{i}}$ is a $O_{A,M}(1)$-rational sequence, so we have the desired factorisation in the $i>1$ case.

When $i=1$, the above argument does not quite work, because $G'_{(1)}$ would be distinct from $G'_{(0)}$ and would thus not qualify as a filtration.  But this can be easily remedied by performing an additional factorisation
$$ g_0 = \beta_0 g'_0$$
where $\beta_0 \in G'$ is a distance $O_{A,M}(1)$ from the identity, and $g'_0$ lies in the kernel of $\xi'_1$.  This leads to a factorisation of the form
$$ g(n) = \beta_0 \beta_1^n g'(n) \gamma_1^n$$
where
$$ g'(n) = g'_0 g_1^{\prime n} g'_{>1}(n)$$
and $g'_{>1}$ is a polynomial sequence taking values in $G'_{(2)}$.  One then argues as before, but now one sets both $G''_{(0)}$ and $G''_{(1)}$ equal to the kernel of $\xi'_1$. 
\end{proof}

We can iterate the above lemma to obtain the following result, which is analogous to \cite[Theorem 1.19]{green-tao-nilratner}. Apart from dealing with irrationality rather than equidistribution, the following result is somewhat different to that just cited in that one requires an arbitrary (rather than polynomial) growth function, but one does not (of course) need polynomial complexity bounds. A variant of \cite[Theorem 1.19]{green-tao-nilratner} was also given in \cite[Theorem 4.2]{green-tao-ziegler-u4inverse}.

\begin{lemma}[Complete factorisation]\label{factor}   Let $(G/\Gamma,G_\bullet)$ be a degree $\leq s$ filtered nilmanifold of complexity $\leq M_0$, and let $g \in \poly(\Z, G_\bullet)$.  For any growth function $\Grow'$, we can find a quantity $M_0 \leq M \leq O_{M,\Grow'}(1)$ and a factorisation $g = \beta g' \gamma$ where:

$\beta \in \poly(\Z,G_\bullet)$ is $(O_{M}(1),N)$-smooth;

$g' \in \poly(\Z,G_\bullet)$ is $(\Grow'(M),N)$-irrational in a subnilmanifold $(G'/\Gamma',G'_\bullet)$ of $(G/\Gamma,G_\bullet)$ of complexity $O_M(1)$, and

$\gamma \in \poly(\Z,G_\bullet)$ is $O_M(1)$-periodic.
\end{lemma}

\begin{proof}  We use an iterative argument, setting $\beta =  \gamma = \id$, $g'=g$, $M= M_0$, and $(G'/\Gamma',G'_\bullet) =(G/\Gamma,G_\bullet)$ to begin with.  In particular, $(G',\Gamma',G'_\bullet)$ is initially a subnilmanifold of $(G/\Gamma,G_\bullet)$ of complexity $O_{M}(1)$.
If $g'$ is $\Grow'(M)$-equidistributed in $(G'/\Gamma',G'_\bullet)$ then we are done; otherwise, by Lemma \ref{init} we may factorise $g' = \beta' g'' \gamma'$ where $c'$ is $(O_{\Grow'(M)}(1),N)$-smooth, $\gamma$ is $O_{\Grow'(M)}(1)$-periodic, and $g''$ now takes values in a subnilmanifold $(G''/\Gamma'',G''_\bullet)$ of $(G'/\Gamma',$ $G'_\bullet)$ of complexity $O_{\Grow'(M)}(1)$ and smaller total dimension than $(G'/\Gamma',G'_\bullet)$.  We then replace $\beta$ by $\beta\beta'$, $\gamma$ by $\gamma' \gamma$, $g'$ by $g''$, $(G'/\Gamma',G'_\bullet)$ by $(G''/\Gamma'',G''_\bullet)$, and increase $M$ to a quantity of the form $O_{\Grow'(M)}(1)$, using Lemma \ref{prod-smooth} to conclude that the new $\beta$ is smooth and the new $\gamma$ is rational.  We then iterate this process.  Since the total dimension of $(G/\Gamma,G_\bullet)$ is initially $O_{M_0}(1)$, this process can iterate at most $O_{M_0}(1)$ times, and the claim follows.
\end{proof}

With this lemma we can now establish Proposition \ref{totalis} and hence Theorem \ref{strong-reg}.  Let $\Grow'$ be a rapid growth function (depending on $\eps,M_0,\Grow$) to be chosen later.  We apply Lemma \ref{factor}, obtaining some $M$ with $M_0 \leq M \leq O_{M_0,\Grow'}(1)$ and a factorisation
$$ \psi(n) = F( \beta(n) g'(n) \gamma(n) \Gamma )$$
with $\beta,g'$ and $\gamma$ having the properties described in that lemma.

The sequence $\gamma$ is $O_{M}(1)$-rational and so, by Lemma \ref{prod-smooth}, the orbit $n \mapsto \gamma(n) \Gamma$ is periodic with some period $q = O_{M}(1)$, and thus $\gamma(n) \Gamma$ depends only on $n \mod q$.  

For each $n$, the rationality of $\gamma(n)$ ensures that $\gamma(n) \Gamma$ intersects $\Gamma$ in a subgroup of $\Gamma$ of index $O_M(1)$.  Since there are only $O_M(1)$ different possible values of $\gamma(n) \Gamma$, we may thus find a subgroup $\Gamma'$ of $\Gamma$ of index $O_M(1)$ such that $\Gamma' \subseteq \gamma(n) \Gamma$ for all $n$.

We can thus express $\psi$ as a virtual nilsequence
$$ \psi(n) = \tilde F( g'(n) \Gamma', n \mod q, n/N )$$
where $\tilde F: G/\Gamma' \times \Z/q\Z \times \R$ is defined by the formula
$$ \tilde F( x, a, y ) := F( \beta(Ny) \tilde x \gamma(\tilde a) \Gamma )$$ whenever $y \in \frac{1}{N}\Z$ and by Lipschitz extension to all $y \in \R$.
where $\tilde a$ is any integer with $\tilde a =a \mod q$, and $\tilde x$ is any element of $G$ such that $\tilde x\Gamma' = x$.  One easily verifies that $\tilde F$ is well-defined and has a Lipschitz norm of $O_M(1)$.  Also, since $g'$ was already $(\Grow(M),N)$-irrational in $G/\Gamma$, and $\Gamma'$ has index $O_M(1)$ in $\Gamma$, we see that $g'$ is $(\gg_M \Grow(M),N)$-irrational in $G/\Gamma'$.  Proposition \ref{totalis} now follows by replacing $M$ by a suitable quantity of the form $O_{M}(1)$, and choosing $\Grow'$ sufficiently rapidly growing depending on $\Grow$.

\section{Proof of the counting lemma}\label{counting-sec}

The purpose of this section is to prove the counting lemma, Theorem \ref{count-lem}. We begin by recalling from the introduction the definition of the Leibman group $G^{\Psi}$.

\begin{definition}[The Leibman group]  Let $\Psi = (\psi_1,\ldots,\psi_t)$ be a system of linear forms $\psi_1,\ldots,\psi_t: \Z^D \to \Z$.  For any $i \geq 1$, define $\Psi^{[i]}$ to be the linear subspace of $\R^t$ spanned by the  $(\psi_1^i(\textbf{n}),\ldots,\psi_t^i(\textbf{n}))$, as $\textbf{n}$ ranges over $\Z^D$. Given a filtered nilmanifold $(G/\Gamma, G_\bullet)$, we define the \emph{Leibman group} $G^{\Psi} \lhd G^t$ to be the Lie subgroup of $G^t$ generated by the elements $g_i^{\vec{v}_i}$ for $i \geq 1$, $g_i \in G_{(i)}$, and $\vec{v}_i \in {\Psi}^{[i]}$, with the convention that if $\vec{v} = (v_1,\dots,v_t)$ then
$$ g^{\vec{v}} := (g^{v_1},\ldots,g^{v_t}).$$ 
\end{definition}

Now might be a good time to remark explicitly that we have introduced a slightly vulgar convention that we hope will help the reader follow this section and other parts of the paper. Bold font letters such as $\mathbf{n} \in \R^D$ denote $D$-dimensional vectors, whilst arrows such as $\vec{v} \in \R^t$ denote $t$-vectors. Occasionally we shall write $m_i := \dim (\Psi^{[i]})$.

When reading this section, it might be found helpful to have a running example in mind. We will take as an illustrative example the case $D = 2$, $t = 4$ and $\Psi = (\psi_1,\dots,\psi_4)$, where $\psi_i(\mathbf{n}) = n_1 + in_2$ for $i = 0,1,2,3$. The system $\Psi$, of course, defines a $4$-term arithmetic progression. As we remarked in the introduction the corresponding Leibman group $G^{\Psi}$ is also known as the \emph{Hall-Petresco group} $\HP^4(G)$. The reader will easily confirm that in this case we have
\[ \Psi^{[1]} = \R (1,1,1,1) \oplus \R (0,1,2,3)\] and
\[ \Psi^{[2]} = \R(1,1,1,1) \oplus \R (0,1,2,3) \oplus \R(0,0,1,3)\]
and
\[ \Psi^{[3]} = \R(1,1,1,1) \oplus \R (0,1,2,3) \oplus \R(0,0,1,3) + \R(0,0,0,1) = \R^4.\]
Note that the flag property \eqref{flag-property} is satisfied; this also follows from the fact that $\Psi$ is translation-invariant, as we shall explain below.

Some work must be done before we can describe $G^{\Psi} = \HP^4(G)$ in a pleasant way. However we can already establish the following lemma, whose statement and proof go some way towards explaining the introduction of the Leibman group.

\begin{lemma}\label{motiv-lemma}
Let $\Psi = (\psi_1,\dots,\psi_t)$ be a system of linear forms $\psi_1,\ldots,\psi_t: \Z^D \to \Z$ satisfying the flag property. Suppose that $(G/\Gamma, G_{\bullet})$ is a filtered nilmanifold and that $g \in \poly(\Z,G_{\bullet})$ is a polynomial sequence. Then the sequence $g^{\Psi} : \Z^D \rightarrow G^t$ defined by $g^{\Psi}(\mathbf{n}) := (g(\psi_1(\mathbf{n})),\dots, g(\psi_t(\mathbf{n})))$ takes values in $G^{\Psi}$.
\end{lemma}
\begin{proof} The sequence $g(n)$ has a (unique) Taylor expansion 
\[ g(n) = g_0 g_1^{\binom{n}{1}} \dots g_s^{\binom{n}{s}}\] with $g_i \in G_{(i)}$ for all $i$ (see Lemma \ref{taylor-lem}). Substituting in, it follows that 
\[ g^{\Psi}(\mathbf{n}) = \prod_{i=0}^s g_i^{(\binom{\psi_1(\mathbf{n})}{i},\dots,\binom{\psi_t(\mathbf{n})}{i})},\] and it is immediate from the definition and the flag property that each element in this product lies in $G^{\Psi}$.\end{proof}

The counting lemma, whose proof is the main objective of this section, was stated as Theorem \ref{count-lem}. Essentially, it states that $g^{\Psi}(\mathbf{n})\Gamma^{\Psi}$ is equidistributed in $G^{\Psi}/\Gamma^{\Psi}$ as $\mathbf{n}$ ranges over ``nice'' subsets of ``big'' lattices, provided that the original sequence $g$ is suitably irrational. We will recall what that means in due course, but our first task is to develop the basic theory of the Leibman group $G^{\Psi}$. At the moment, for example, we have not established that $G^{\Psi}$ is a connected Lie subgroup of $G^t$ or that $G^{\Psi}/\Gamma^{\Psi}$ has the structure of a filtered nilmanifold. Nor have we developed tools for calculating inside this group.\vspace{11pt}

\textsc{Basic facts about the Leibman group and nilmanifold.} We can endow $\R^t$ with the structure of a commutative algebra over $\R$ by using the pointwise product
$$ \vec{x}.\vec{y} = (x_1 y_1,\ldots,x_t y_t)$$
and setting $\vec{1} = (1,\ldots,1)$ to be the multiplicative identity.
With this algebra structure, one can view the spaces ${\Psi}^{[i]}$ defined in Definition \ref{leib-gp} as the span of the powers $\Psi(\mathbf{n})^i$ for $\mathbf{n} \in \Z^D$, where we view $\Psi$ as a homomorphism from $\Z^D$ to $\Z^t$. 

By a standard depolarisation argument we have the following important fact.
\begin{lemma}\label{depol}  ${\Psi}^{[i]}$ contains all products $\Psi(\mathbf{n}_1) \ldots \Psi(\mathbf{n}_i)$, $\mathbf{n}_1,\ldots,\mathbf{n}_i \in \Z^D$.
\end{lemma}
\begin{proof}  Observe the elementary depolarisation identity
$$ \Psi(\mathbf{n}_1) \ldots \Psi(\mathbf{n}_i) = \frac{(-1)^i}{i!} \sum_{\omega \in \{0,1\}^i} (-1)^{|\omega|} \Psi(\omega_1 \mathbf{n}_1 + \ldots + \omega_i \mathbf{n}_i)^i$$
where $\omega = (\omega_1,\ldots,\omega_i)$ and $|\omega| := \omega_1+\ldots+\omega_i$. The claim follows.
\end{proof}

This has several useful consequences. 

\begin{corollary}\label{filtr} Let $\Psi = (\psi_1,\dots, \psi_t)$ be a system of linear forms from $\Z^D$ to $\Z$. Then
\begin{enumerate}
\item for any $i,j \geq 0$, we have $\Psi^{[i]} \cdot \Psi^{[j]} \subseteq \Psi^{[i+j]}$;
\item If $\Psi$ is translation-invariant, then it has the flag property \eqref{flag-property};
\item If $\Psi^{[i]} = \R^t$ for some $i$, then $\Psi^{[i+1]} = \Psi^{[i+2]} = \dots = \R^t$.
\end{enumerate}
\end{corollary}
\begin{proof}
(i) is immediate from Lemma \ref{depol}. (ii) is immediate from (i): if $\Psi$ is translation invariant then, by definition, $\Psi^{[1]}$ contains $\vec{1} = (1,1,\dots, 1)$, and therefore $\Psi^{[i]} = \Psi^{[i]} \cdot \vec{1} \subseteq \Psi^{[i]} \cdot \Psi^{[1]} \subseteq \Psi^{[i+1]}$, for all $i$. Finally, we turn to (iii). Suppose that $\Psi^{[i+1]} \neq \R^t$. Then there is some nonzero vector $(c_1,\dots, c_t)$ which annihilates $\Psi^{[i+1]}$, so by Lemma \ref{depol} and (i) we have
\begin{equation}\label{lin-rel} \sum_{j = 1}^t c_j \psi_j^i(\mathbf{n}) \psi_j(\mathbf{m}) = 0\end{equation} for all $\mathbf{m}, \mathbf{n} \in \Z^D$. Since none of the forms $\psi_j$ is zero\footnote{this is part of the definition of a system $\Psi$, but in any case item (iii) of the corollary is vacuous without this assumption.} and a finite collection of hyperplanes cannot cover $\Z^D$, we may choose $\mathbf{m}$ such that none of the $\psi_j(\mathbf{m})$ is zero. But then \eqref{lin-rel} gives a linear relation between the $\psi_j^i$, and so $\Psi^{[i]} \neq \R^t$.
\end{proof}

Let $(G/\Gamma, G_\bullet)$ be a degree $\leq s$ filtered nilmanifold.
From Definition \ref{leib-gp}, the Leibman group $G^{\Psi}$ is the subgroup of $G^t$ generated by the group elements $g_i^{v_i}$ for $i \geq 1$, $v_i \in {\Psi}^{[i]}$, and $g_i \in G_{(i)}$.  For any $i_0 \geq 1$, let $G_{(i_0)}^{\Psi}$ be the subgroup of $G^{\Psi}$ generated by those $g_i^{\vec v_i}$ with $i \geq i_0$, $\vec v_i \in {\Psi}^{[i]}$, $g_i \in G_{(i)}$, with the convention that $G_{(0)}^{\Psi} := G^{\Psi}$.  

\begin{lemma}[Filtration property for $G_\bullet^\Psi$]\label{gij-lemma}
$G_\bullet^{\Psi} := (G_{(i)}^{\Psi})_{i=0}^\infty$ is a filtration on $G^{\Psi}$.   In other words, the $G_{(i)}^\Psi$ are nested with $[G_{(i)}^\Psi,G_{(j)}^\Psi] \subset G_{(i+j)}^\Psi$ for all $i, j\geq 0$.
\end{lemma}

\begin{proof} It suffices to check that if $g_i \in G_{(i)}$, $g_j \in G_{(j)}$, $\vec{v}_i = (v_{i1},\dots, v_{it}) \in \Psi^{[i]}$ and $\vec{v}_j = (v_{j1},\dots,v_{jt}) \in \Psi^{[j]}$ then $[g_i^{\vec{v}_i}, g_j^{\vec{v}_j}] \in G_{(i+j)}^{\Psi}$. But this follows from the Baker-Campbell-Hausdorff formula (see \eqref{comm-form}), the filtration property of $G_{(i)}$ and Corollary \ref{filtr}.\end{proof}

Suppose we have a system $\Psi$ satisfying the flag property. The subspaces $\Psi^{[i]}$ are rational (i.e. they can be defined over $\Q$).  From a greedy algorithm (and clearing denominators) we may thus find a basis $\vec{v}_1,\ldots,\vec{v}_{m_s} \in \Psi^{[s]}$ with the following properties:

(Integrality) $\vec{v}_1,\ldots,\vec{v}_{m_s}$ all lie in $\Z^t$;

(Partial span) For every $1 \leq i \leq s$, $\vec{v}_1,\ldots,\vec{v}_{m_i}$ span $\Psi^{[i]}$;

(Row echelon form)  For each $1 \leq j \leq m_s$, there exists $l_j$, $1 \leq l_j \leq t$, such that $\vec{v}_j$ has a non-zero $l_j$ coordinate, but such that $\vec{v}_{j'}$ has a zero $l_j$ coordinate for all $j < j' \leq m_s$.

For instance, the basis 
$$ \vec v_1 := (1,1,1,1); \quad \vec v_2 := (0,1,2,3); \quad \vec v_3 := (0,0,1,3); \quad \vec v_4 := (0,0,0,1)$$
we implicitly gave above for our running example is already in this form.

Fix such a basis.  For each basis element $\vec{v}_j$, we can define the \emph{degree} $\deg(\vec{v}_j)$ of that element to be the first $i$ for which $j \leq m_i$, thus $\deg(\vec{v}_j)$ is an integer between $1$ and $s$, and $\vec{v}_j \in \Psi^{[\deg(\vec{v}_j)]}$.

Observe that an arbitrary element of $G^\Psi$ can be expressed as a product of finitely many elements of the form $g_j^{\vec v_j}$ for $0 \leq j \leq m_s$ and $g_j \in G_{(\deg(\vec{v}_j))}$.
By many applications\footnote{Indeed, one uses \eqref{bch-2} and Lemma \ref{gij-lemma} to extract out and collects all terms with degree $\deg(\vec{v}_j)= 1$, leaving only terms with base $g_j$ in $G_{(2)}$.  Then one extracts out those terms with degree $2$ (merging them with the $i=1$ terms as necessary), leaving only terms with base in $G_{(3)}$.  Continuing this process gives the desired factorisation.}
of the Baker-Campbell-Hausdorff formula (see \eqref{bch-2}) and Lemma \ref{gij-lemma}, we can now express any element of $G^{\Psi}$ in the form
\begin{equation}\label{go}
\prod_{j=1}^{m_s} g_j^{\vec{v}_j}
\end{equation}
where $g_j \in G_{(\deg(\vec{v}_j))}$ for all $1 \leq j \leq m_s$.

Thus, in our running example, we have the explicit description of $G^{\Psi} = \HP^4(G)$ as
\[ \{ (g_0, g_0 g_1, g_0 g_1^2 g_2, g_0 g_1^3 g_2^3 g_3) : g_0 \in G_{(0)}, g_1 \in G_{(1)}, g_2 \in G_{(2)}, g_3 \in G_{(3)}\}.\]
Note that from results on the Taylor expansion (see Lemma \ref{taylor-lem}) this group may also be identified as
\[ \{ (g(0), g(1), g(2), g(3)) : g \in \poly(\Z,G_{\bullet})\}.\]
The group nature of $\HP^4(G)$ is then easily deduced from Theorem \ref{ll-thm}, but this presentation is somewhat specific to the Hall-Petresco case and we shall not require it further.

From the row-echelon form one can verify inductively that the representation \eqref{go} is unique (this can be seen clearly by working with the Hall-Petresco example presented above). This gives $G^{\Psi}$ the structure of a connected, simply connected Lie group, with dimension
\begin{equation}\label{dimgo}
\dim(G^\Psi) = \sum_{i=1}^s \dim(G^{(i)}) (\dim(\Psi^{[i]}) - \dim(\Psi^{[i-1]})) 
\end{equation}
(with the convention that $\Psi^{[0]}$ is trivial).  A similar argument also shows that every element of $G^{\Psi}_{(i_0)}$ can be expressed uniquely in the form \eqref{go}, where now $g_j$ is constrained to lie in $G_{(\max(\deg(v_j),i_0))}$ rather than $G_{(\deg(v_j))}$.   In particular, by reading off the coefficients $g_{j}$ one at a time, this implies the pleasant identity
\begin{equation}\label{pleasant}
G^{\Psi}_{(i)} = G^{\Psi} \cap (G_{(i)})^k.
\end{equation}

\emph{Remark.} From Taylor expansion (see Lemma \ref{taylor-lem}) we see that the sequence $g^\Psi$ in \eqref{gpsi-def} lies in $\poly(\Z,G^\Psi_\bullet)$.  While we do not directly use this fact here, it may help explain why the filtration $G^\Psi_\bullet$ will plays a prominent role in the proof of the counting lemma that we will shortly come to.

Recall that we normalised the basis vectors $\vec{v}_{j} \in \Z^t$ to have integer coefficients.  As a consequence, we see that if the $g_{j}$ are in $\Gamma$, then the expression \eqref{go} lies in $\Gamma^k$. From this (and many applications of Lemma \ref{gij-lemma}) we see that $\Gamma^{\Psi}_{(i)} := \Gamma^k \cap G^{\Psi}_{(i)}$ is cocompact in $G^{\Psi}_{(i)}$ for each $i$, and so $(G^{\Psi}/\Gamma^{\Psi}, G^{\Psi}_\bullet)$ is a filtered nilmanifold.  Furthermore, the same argument shows that the $G^{\Psi}_{(i)}$ are rational subgroups of $G^k$ and so $(G^{\Psi}/\Gamma^{\Psi}, G^{\Psi}_\bullet)$ is a subnilmanifold of $(G^k/\Gamma^k, G^k_\bullet)$.\vspace{11pt}

\textsc{The counting lemma: preliminary man{\oe}uvres.} Now that we have verified that $G^{\Psi}/\Gamma^{\Psi}$ is indeed a nilmanifold, we can begin the proof of Theorem \ref{count-lem}. 

We begin with some easy reductions.  
First, observe that for fixed $M$, there are only finitely many possibilities for $s,D,t,\Psi$, and (up to isomorphism) there are only finitely many possibilities for $(G/\Gamma,G_\bullet)$ and $\Gamma$.  Thus it will suffice to establish the result for a single choice of $s,D,t,\Psi,(G/\Gamma, G_\bullet)$, with the bounds depending on these quantities.  Hence, we fix these quantities and allow all implicit constants to depend on these quantities (thus, in this section, we will not explicitly subscript out $O(1)$ quantities).

Similarly, because the space of Lipschitz functions with Lipschitz norm $O(1)$ is precompact in the uniform topology (by the Arzel\'a-Ascoli theorem), it suffices to prove the desired bound for each fixed $F$, as the uniformity in $F$ then follows from an easy approximation argument.  Thus we fix $F$ and allow all quantities to depend on $F$.

Next, we observe that we may normalise $g(0) = \id$.  Indeed, we may factorise $g(0) = c_0 \gamma_0$ where $d_G(c_0,\id) = O(1)$ and $\gamma_0 \in \Gamma$.  Factorising, we obtain
$$ g(n) = c_0 g'(n) \gamma_0$$
where $g'(n) := c_0\gamma_0 (\gamma_0^{-1}g(n)\gamma_0)$. Note that $g'(0) = \id$ and that Taylor coefficients of $g'$ are given by $g'_i = \gamma_0^{-1} g_i \gamma_0$, and so $g'$ is also $(A,N)$-irrational.  It is then an easy matter to see that Theorem \ref{count-lem} for $g$ and $F$ follows from Theorem \ref{count-lem} for $g'$ and for the shifted function $F'(x) := F(c_0 x)$, which is still Lipschitz with norm $O(1)$.  

Note that we may assume that $A$ and $N$ are large, as the claim is trivial otherwise.\vspace{11pt}

\textsc{Equidistribution in the Leibman group.} Let us recall what we are trying to prove. In the counting lemma, Theorem \ref{count-lem}, our aim is to show that if $g(n)$ is suitably irrational then the orbit $(g^{\psi}(\mathbf{n}))_{\mathbf{n} \in (\mathbf{n}_0 + \Lambda) \cap P}$ is equidistributed on the Leibman nilmanifold $G^{\Psi}/\Gamma^{\Psi}$. We shall proceed by contradiction, supposing this orbit is not equidistributed and deducing that $g(n)$ could not have been irrational. The reader should recall the definition of \emph{irrational} in this context: it is given in Definition \ref{irrat-def}.

Our main tool will be a mild generalisation of the ``multiparameter Leibman criterion'', which is \cite[Theorem 8.6]{green-tao-nilratner}. Here is the statement we shall use.

\begin{theorem}\label{mpet}
Suppose that $(G/\Gamma, G_{\bullet})$ is a filtered nilmanifold of complexity $\leq M$ and that $g \in \poly(\Z^D, G_{\bullet})$ is a polynomial sequence for some $D \leq M$. Suppose that $\Lambda \subseteq \Z^D$ is a lattice of index $\leq M$, that $\mathbf{n_0} \in \Z^D$ has magnitude $\leq M$, and that $P \subseteq [-N,N]^D$ is a convex body. Suppose that $\delta > 0$, and that 
$$ \big| \sum_{{\mathbf n} \in (\mathbf{n_0}+\Lambda) \cap P} F(g({\mathbf{n}})\Gamma) - \frac{\operatorname{vol}(P)}{[\Z^D:\Lambda]} \int_{G/\Gamma} F \big| > \delta N^D \|F\|_{\operatorname{Lip}}$$
for some Lipschitz function $F: G/\Gamma \to \C$.
Then there is a nontrivial homomorphism $\eta : G \rightarrow \R$ which maps $\Gamma$ to $\Z$, has complexity $O_M(1)$ and such that 
\[ \Vert \eta \circ g \Vert_{C^{\infty}([N]^D)} = O_{\delta,M}(1).\]
\end{theorem}

\emph{Remarks.} This differs from \cite[Theorem 8.6]{green-tao-nilratner} in several insubstantial ways. On the one hand we have no concern here with the polynomial bounds that were important in that setting. However, we are dealing here with a sublattice $\Lambda \subseteq \Z^D$ rather than $\Z^D$ itself, and with an arbitrary convex body $P$ rather than the box $[N]^D$. This more general result can be deduced from \cite[Theorem 8.6]{green-tao-nilratner} in a somewhat routine, though slightly tedious, manner. We sketch the details in Appendix \ref{appendix-b}. The notation $C^{\infty}([N]^D)$ is recalled both in the appendix and later in this section.

Later on, the notation will get a little complicated. Let us, then, first apply Theorem \ref{mpet} to establish the following very simple special case of the counting lemma (it is, of course, the special case in which $\Psi$ consists of the single form $\psi_1(\textbf{n}) = n_1$).

\begin{lemma}[Irrational implies equidistributed]\label{irrat-equi}
Suppose that $(G/\Gamma, G_{\bullet})$ is a filtered nilmanifold of complexity at most $M$ and that $g : \Z \rightarrow G$ is an $(A,N)$-irrational polynomial sequence. Then we have the equidistribution property
$$ \E_{n \in [N]} F(g(n)\Gamma) = \int_{G/\Gamma} F + O_M( A^{-c_M} \|F\|_{\operatorname{Lip}} )$$
for all Lipschitz $F: G/\Gamma \to \C$ and some $c_M > 0$.
\end{lemma}

\begin{proof} Suppose the conclusion is false. Then by\footnote{In fact here we only need the rather simpler $1$-parameter version, which is \cite[Theorem 1.16]{green-tao-nilratner}.} Theorem \ref{mpet} there is some continuous homomorphism $\eta : G \rightarrow \R$ which vanishes on $[G,G]$ and maps $\Gamma$ to $\Z$, has complexity $O_{\delta}(1)$, and for which $\Vert \eta \circ g \Vert_{C^{\infty}[N]} \leq \delta^{-O(1)}$. Recall (cf. \cite[Definition 2.7]{green-tao-nilratner}) what this means: in the Taylor expansion 
\[ \textstyle\eta \circ g(n) = \alpha_0 + \alpha_1 \binom{n}{1} + \dots + \alpha_s \binom{n}{s},\] the $j$th coefficient $\alpha_j$ satisfies $\Vert \alpha_j\Vert_{\R/\Z} \leq \delta^{-O(1)}/N^j$ for $j = 1,\dots, s$. If the sequence $g$ is developed as a Taylor expansion
\[ g(n) = g_0 g_1^{\binom{n}{1}} \dots g_s^{\binom{n}{s}}\] then we of course have $\alpha_j = \eta(g_j)$.
Choose $i$ maximal so that the restriction $\eta|_{G_{(i)}}$ is nontrivial. Then certainly $\Vert \eta(g_i) \Vert_{\R/\Z} \leq \delta^{-O(1)}/N^i$. We claim that $\eta$ is an $i$-horizontal character in the sense of Definition \ref{def-a4}, a statement which will clearly contradict the supposed $(A,N)$-irrationality of $g$ if $\delta$ is a sufficiently small power of $1/A$. To this end all we need do is confirm that $\eta$ vanishes on $G_{(i+1)}$ and on $[G_{(j)}, G_{(i-j)}]$ for $0 \leq j \leq i$, and maps $\Gamma_{(i)}$ to $\Z$. The first of these follows from the maximality of $i$, whilst the second and third follow immediately from the properties of $\eta$ stated at the beginning of the proof.\end{proof}

Let us turn now to the more notationally intensive general case, that is to say the proof of Theorem \ref{count-lem} itself. We will assume henceforth that $\Psi$ has the flag property. Now, we apply Theorem \ref{mpet} to $G^{\Psi}/\Gamma^{\Psi}$ to conclude that there is a non-trivial continuous homomorphism $\eta : G^{\Psi} \rightarrow \R$ which maps $\Gamma^{\Psi}$ to $\Z$, has complexity $O_\delta(1)$, and satisfies 
\begin{equation}\label{smoothness-d} \Vert \eta \circ g^{\Psi} \Vert_{C^{\infty}([N]^D)} = O_\delta(1).\end{equation} Much as in the proof of Lemma \ref{irrat-equi}, what this means is that if $\eta \circ g^{\Psi}(\mathbf{n})$ is developed as a Taylor series in multi-binomial coefficients $\binom{\mathbf{n}}{\mathbf{j}} = \binom{n_1}{j_1} \dots \binom{n_D}{j_D}$ (see Lemma \ref{taylor-lem}), the coefficient $\alpha_{\mathbf{j}}$ satisfies $\Vert \alpha_{\mathbf{j}} \Vert_{\R/\Z} \ll_\delta N^{-|\mathbf{j}|}$. Our aim is to use this information to contradict the assumption that $g(n)$ is $(A,N)$-irrational.

Let us once again take $i$ maximal such that $\eta |_{G^{\Psi}_{(i)}}$ is nontrivial.  
Considering again the Taylor expansion of $g(n)$, we have
\begin{equation}\label{eta-g-taylor} (\eta \circ g^{\Psi})(\mathbf{n}) = \sum_{j=1}^i \eta (g_j^{\binom{\psi_1(\mathbf{n})}{j}},\dots, g_j^{\binom{\psi_t(\mathbf{n})}{j}}).\end{equation}
Take the basis $\vec{v}_1,\vec{v}_2,\dots$ for $\Psi^{[i]}$ described earlier. Then, since the vector \[ \textstyle (\binom{\psi_1(\mathbf{n})}{j},\dots,\binom{\psi_t(\mathbf{n})}{j})\] lies in $\Psi^{[j]}$, there is an expansion
\begin{equation}\label{a13} \textstyle(\binom{\psi_1(\mathbf{n})}{j},\dots,\binom{\psi_t(\mathbf{n})}{j}) \displaystyle= P_{j,1}(\mathbf{n}) \vec{v}_1 + \dots + P_{j,m_j} (\mathbf{n})\vec{v}_{m_j}\end{equation} for $j = 1,\dots, i$, where the $P_{j,k}: \Z^D \to \R$ are polynomials of degree at most $j$, recalling that $m_j := \dim (\Psi^{[j]})$.
Comparing with \eqref{eta-g-taylor}, we obtain
\begin{equation}\label{a14}(\eta \circ g^{\Psi})(\mathbf{n}) = \sum_{j=1}^i \sum_{k = 1}^{m_j}P_{j, k}(\mathbf{n}) \eta(g_j^{\vec{v}_k}).\end{equation}
We are going to look at the coefficients $\alpha_{\mathbf{i}}$ of \eqref{a14} for the monomial $\mathbf{n}^{\mathbf{i}} := n_1^{i_1} \dots n_D^{i_D}$, where $\mathbf{i} = (i_1,\dots, i_D)$ and  $|\mathbf{i}| := |i_1| + \dots + |i_d| = i$.  We are assuming that every such coefficient satisfies $\Vert \alpha_{\mathbf{i}} \Vert_{\R/\Z} \ll_\delta N^{-i}$. Note also that
\begin{equation}\label{a15} \alpha_{\mathbf{i}} = \sum_{k=1}^{m_i} (P_{i,k})_{\mathbf{i}} \eta(g_i^{\vec{v}_k}),\end{equation}
where $(P_{i,k})_{\mathbf{i}}$ is the ${\mathbf{n}}^{\mathbf{i}}$ coefficient of $P_{i,k}({\mathbf n})$; this is because terms of total degree $i$ cannot arise from the terms $j = 1,\dots, i-1$ in the sum on the right hand side of \eqref{a14}.

On the other hand by taking $j = i$ in \eqref{a13} we have
\begin{align}\nonumber  (P_{i,1} & (\mathbf{n}))_{\mathbf{i}} \vec{v}_1 + \dots + (P_{i,m_i}  (\mathbf{n}))_{\mathbf{i}}\vec{v}_{m_i} \\  \nonumber & = \frac{1}{i_1! \dots i_D!}(\psi_1(\mathbf{e}_1)^{i_1}\cdots \psi_1(\mathbf{e}_D)^{i_D}, \dots, \psi_t(\mathbf{e}_1)^{i_1} \cdots \psi_t(\mathbf{e}_D)^{i_D})\\ & = \frac{1}{i_1! \dots i_D!} \Psi(\mathbf{e}_1)^{i_1} \cdots \Psi(\mathbf{e}_D)^{i_D},\label{a16}\end{align} where $\textbf{e}_j = (0,\dots, 1, \dots,0) \in \Z^D$, the $1$ being in the $j$th position, and $\Psi({\mathbf e}_j) := (\psi_1({\mathbf e}_j), \ldots, \psi_t({\mathbf e}_j)) \in \R^t$.  

Comparing \eqref{a15} and \eqref{a16} and using the fact that $\eta$ is a homomorphism on $G^{\Psi}$, we obtain
\[ \alpha_{\mathbf{i}} = \frac{1}{i_1! \dots i_D!} \eta(g_i^{\Psi(\mathbf{e}_1)^{i_1} \cdots \Psi(\mathbf{e}_D)^{i_D}}).\]
Thus, for each $\mathbf{i}$ with $|\mathbf{i}| = |i_1| + \dots + |i_D| = i$, we have
\begin{equation}\label{smoothness} \Vert \eta(g_i^{\Psi(\mathbf{e}_1)^{i_1} \cdots \Psi(\mathbf{e}_D)^{i_D}}) \Vert_{\R/\Z} \ll_\delta N^{-i}\end{equation}
To obtain the desired contradiction with the $(A,N)$-irrationality hypothesis and thus complete the proof, it suffices (after taking $A$ sufficiently large depending on $\delta$) to establish that for at least one choice of $\mathbf{i}$, the map $\xi_{\mathbf{i}}: G_{(i)} \to \R$ defined by
\[ \xi_{\mathbf{i}}(g) := \eta(g^{\Psi(\mathbf{e}_1)^{i_1} \cdots \Psi(\mathbf{e}_D)^{i_D}})\] is a nontrivial horizontal $i$-character of complexity $O_\delta(1)$.

The complexity bound follows from the fact that the coefficients of the forms $\psi_i$ are integers of size $O(1)$ and the Baker-Campbell-Hausdorff formula (Appendix \ref{bch}). That at least one of these maps is nontrivial follows from that fact that $\eta$ is nontrivial on $G^{\Psi}_{(i)}$ and the fact that the vectors $\Psi(\mathbf{e}_1)^{i_1} \cdots \Psi(\mathbf{e}_D)^{i_D}$, $i_1 + \dots + i_D = i$, span $\Psi^{[i]}$ (a consequence of Lemma \ref{depol}).

Now $\xi_{\mathbf{i}}$ maps $\Gamma_{(i)}$ to $\Z$, since if $\gamma \in \Gamma_{(i)}$ then $\gamma^{ \Psi(\mathbf{e}_1)^{i_1} \cdots \Psi(\mathbf{e}_D)^{i_D} } \in \Gamma_{(i)}^{\Psi}$, and this group is mapped to $\Z$ by $\eta$.

Morover $\xi_{\mathbf{i}}$ annihilates $G_{(i+1)}$. To see this, note that $\Psi(\mathbf{e}_1)^{i_1} \cdots \Psi(\mathbf{e}_D)^{i_D} \in \Psi^{[i+1]}$ by Lemma \ref{depol} and the flag property, which guarantees\footnote{This is where the flag property is used crucially, and it is at this point that the error in our original paper occurred.} that $\Psi^{[i]} \leq \Psi^{[i+1]}$. Therefore if $g_{i+1} \in G_{(i+1)}$ then $g_{i+1}^{\Psi(\mathbf{e}_1)^{i_1} \cdots \Psi(\mathbf{e}_D)^{i_D}} \in G_{(i+1)}^{\Psi}$, and by the maximality assumption this latter group is annihilated by $\eta$.

To qualify $\xi_{\mathbf{i}}$ as an $i$-horizontal character we must also show that it vanishes on $[G_{(j)}, G_{(i-j)}]$ for each $0 \leq j \leq i$. To this end, note that we may factor
\[  \Psi(\mathbf{e}_1)^{i_1} \cdots \Psi(\mathbf{e}_D)^{i_D}= w w',\]
where $w \in \Psi^{[j]}$ and $w' \in \Psi^{[i-j]}$. Indeed, we may take
\[ w = \Psi(\mathbf{e}_1)^{j_1} \cdots \Psi(\mathbf{e}_D)^{j_D}, \qquad w' = \Psi(\mathbf{e}_1)^{i_1 - j_1} \cdots \Psi(\mathbf{e}_D)^{i_D -j_D}\] for any indices $j_1,\dots,j_D$ with $j_l \leq i_l$ and $j_1 + \dots + j_D = j$, whereupon the relevant containments follow from Lemma \ref{depol}.
Now if $g \in G_{(j)}$ and $g' \in G_{(i-j)}$ are arbitrary then we have
\[ [ g^w, g^{\prime w'}] \equiv [g,g']^{w w'} \md{G^{\Psi}_{(i+1)}}\] by the Baker-Campbell-Hausdorff formula \eqref{comm-form}. Applying $\eta$, which is trivial on $G^{\Psi}_{(i+1)}$ by assumption, we obtain
\[ \xi_{\mathbf{i}}([g,g']) = \eta ([g,g']^{w w'}) = \eta ([ g^w, g^{\prime w'} ]) = 0,\] the last step being a consequence of the fact that $\eta$ has abelian image and hence vanishes on $[G^{\Psi}, G^{\Psi}]$. 

We have shown that $\xi_{\mathbf{i}}$ is an $i$-horizontal character, and this concludes the proof of the counting lemma, Theorem \ref{count-lem}.

\section{Generalised von Neumann type theorems}\label{gvn-sec}

In this section we recall a number of results asserting the connection between Gowers norms and various types of linear configuration. These results are collectively known in the literature as ``generalised von Neumann theorems''. The connection between Gowers norms (not called by that name, of course) and linear configurations was first made in \cite{gowers-4aps}.  A fairly general result of this type, which appears in \cite{green-tao-linearprimes}, is the following.

\begin{theorem}[Generalised von Neumann Theorem]\label{cs-lemma}  Let $\Psi = (\psi_1,\ldots,\psi_t)$ be a collection of linear forms $\psi_1,\ldots,\psi_t: \Z^D \to \Z$ for some $t,D \geq 1$, any two of which are linearly independent.  Then there exists an integer $s = s(\Psi)$ with the property that one has the inequality
\begin{equation}\label{end}
| \E_{\mathbf{n} \in [N]^d} \prod_{i=1}^t f_i( \psi_i(\mathbf{n}) )| 
\ll_{t,D,\Psi} \inf_{1 \leq i \leq m} \|f_i\|_{U^{s+1}[N]}
\end{equation}
for all $N \geq 1$ and all $f_1,\ldots,f_m: [N] \to \C$ bounded in magnitude by $1$.  
\end{theorem}

\emph{Remarks.} A natural value of $s(\Psi)$ comes from the proof in \cite{green-tao-linearprimes}, which proceeds via $s$ applications of the Cauchy-Schwarz inequality. For this reason Gowers and Wolf \cite{gowers-wolf-1} call $s(\Psi)$ the \emph{Cauchy-Schwarz complexity} of the system $\Psi$. There is a linear-algebra recipe for computing $s(\Psi)$ which is not especially enlightening but sufficiently simple that we can give it here (see the introduction to \cite{green-tao-linearprimes} for more details). If $1 \leq i \leq t$ and $s \geq 0$ then we say that $\Psi$ has $i$-complexity at most $s$ if one can cover the $t-1$ forms $\{\psi_j : j \in [t] \setminus \{i\}\}$ by $s+1$ classes, such that $\psi_i$ does not lie in the linear span of the forms in any one of these classes. Then $s(\Psi)$ is the smallest $s$ for which the system has $i$-complexity at most $s$ for all $1 \leq i \leq t$. Note, then, that the Cauchy-Schwarz complexity of the system $\Psi = \{n_1, n_1 + n_2, \dots, n_1 + (k-1)n_2\}$ corresponding to a $k$-term arithmetic progression is $k-2$. As a final remark, let us note that Theorem \ref{cs-lemma}, as proved in \cite[Appendix C]{green-tao-linearprimes}, is regrettably somewhat difficult to understand as we had to establish a more general result in which the functions $f_i$ were bounded by an arbitrary pseudorandom measure, and this is notationally heavy. For a gentle explanation of the special case $\Psi = \{n_1, n_1 + n_2, n_1 + 2n_2, n_1 + 3n_2\}$ (where $s = 2$) the reader may consult \cite[Proposition 1.11]{green-montreal}.  A sketch of the proof of Theorem \ref{cs-lemma} is also given in \cite[\S 2]{gowers-wolf-1}.  See also \cite{leib-orbits} for a variant of these notions of complexity in the ergodic setting, and for polynomial forms instead of linear ones.

We will need a \emph{twisted} version of the Generalised von Neumann inequality, in which an additional nilsequence of lower degree is inserted. We shall not need it for general linear forms, so we formulate just the special case we need.

\begin{lemma}[Twisted generalised von Neumann theorem]\label{gvn-twist} Let $k \geq 3$, let $f_0,\ldots,f_{k-1}: [N] \to \C$ be bounded in magnitude by $1$, let $c_0,\dots, c_{k-1}$ be distinct integers, and let $F(g(n)\Gamma)$ be a degree $\leq (k-2)$ nilsequence of complexity at most $M$.  Then
$$ \big|\E_{n \in [N],d \in [-N,N]} F(g(d)\Gamma) \prod_{i=0}^{k-1} f_i(n + c_i d)\big| \ll_{k,M,c_0,\dots, c_{k-1}} \inf_{0 \leq i \leq k-1} \|f_i\|_{U^{k-1}[N]}.$$
\end{lemma}
\begin{proof}  We induct on $k$, starting with the case $k = 3$. The underlying nilmanifold $G/\Gamma$ is then a torus $(\R/\Z)^m$ with $m = O_M(1)$, and $g(n) = \theta n + \theta_0$ may be taken to be linear. By a standard Fourier decomposition we may assume that $F(x) = e(\xi \cdot x)$ for some $\xi \in \Z^m$ with $|\xi| = O_M(1)$, in which case we may rewrite the estimate to be proven as
\[ |\E_{n \in [N]} \E_{d \in [-N,N]} f_0(n + c_0 d) f'_1 (n + c_1 d) f'_2 (n + c_2 d)| \ll_{k,M} \inf_{i = 0,1,2} \| f_i \|_{U^2[N]},\]
where $f'_1(n) = f_1(n) e(-(c_2 - c_1)^{-1}\xi \cdot \theta n)$ and $f'_2(n) = f_2(n) e((c_2 - c_1)^{-1}\xi \cdot \theta n)$. However it is easy to establish the invariance properties $\Vert f_1 \Vert_{U^2} = \Vert f'_1\Vert_{U^2}$ and $\Vert f_2 \Vert_{U^2} = \Vert f'_2 \Vert_{U^2}$, and so the result follows immediately from Theorem \ref{cs-lemma}.

Now suppose that $k \geq 4$ and the claim has already been proven for smaller $k$.  By permuting indices and then translating $n$, it suffices to show that
\begin{equation}\label{endit}
 |\E_{n \in [N]; d \in [-N,N]} F(g(d)\Gamma) \prod_{i=0}^{k-1} f_i(n + c_i d)| \ll_{k,M,c_0,\ldots,c_{k-1}} \|f_{k-1}\|_{U^{k-1}[N]}
 \end{equation}
under the assumption that $c_0=0$.

Recall from \cite{green-tao-nilratner} that we define a \emph{vertical character} to be a continuous homomorphism $\xi: G_{(k-2)}/(G_{(k-2)} \cap \Gamma) \to \R/\Z$.  We say that $F$ has \emph{vertical frequency} $\xi$ if one has $F(g_{k-2} x) = e(\xi(g_{k-2})) F(x)$ for all $x \in G/\Gamma$ and $g_{k-2} \in G_{(k-2)}$.  By a standard Fourier decomposition in the vertical direction (e.g. by arguing exactly as in \cite[Lemma 3.7]{green-tao-nilratner}) we may assume without loss of generality that $F$ has a vertical frequency $\xi$.

Applying the Cauchy-Schwarz inequality, we can bound the left-hand side of \eqref{endit} by
$$
\ll |\E_{n \in [N]; h,d \in [-N,N]} F(g(d+h)\Gamma) \overline{F(g(d)\Gamma)}
\prod_{i=0}^{k-1} f_i(n + c_i d + c_i h) \overline{f_i(n + c_i d)}|^{1/2}.$$
Because $F$ has a vertical frequency, $F(g(d+h))\overline{F(g(d)\Gamma)}$ is a degree $\leq (k-3)$ nilsequence of complexity $O_{M,k}(1)$ (see \cite[Proposition 7.2]{green-tao-nilratner}).  Applying the induction hypothesis, we may thus bound the above expression by
$$ \ll_{M,k,c_0,\ldots,c_{k-1}} (\E_{h \in [-N,N]} \| \Delta_{c_i h} f_i \|_{U^{k-2}[N]}^2)^{1/2}$$
which by H\"older's inequality can be bounded by
$$ \ll_{M,k,c_0,\ldots,c_{k-1}} (\E_{h \in [-|c_i|N,|c_i|N]} \| \Delta_{h} f_i \|_{U^{k-2}[N]}^{2^{k-2}})^{1/2^{k-2}}$$
and the claim follows from the recursive definition of the Gowers norms.
\end{proof}

\emph{Remark.} The above argument is very similar to the short proof presented in \cite[Appendix G]{green-tao-ziegler-u4inverse} that $s$-step nilsequences obstruct uniformity in the $U^{s+1}$-norm (that is, the inverse conjecture $\GI(s)$ is an if-and-only if statement).

\section{On a conjecture of Bergelson, Host, and Kra}\label{bhk-sec}

We now apply the arithmetic regularity and counting lemmas to establish Theorem \ref{bhk-thm}, the proof of the conjecture of Bergelson, Host and Kra.  It will suffice to prove the following claim.

\begin{theorem}\label{bhk-thm2}  Let $k = 1, 2, 3$ or $4$, and suppose that $0 < \alpha < 1$ and $\eps > 0$.  Then for any $N \geq 1$ and any subset $A \subseteq [N]$ of density $|A| \geq \alpha N$, one can find a function $\mu: \Z \to \R^+$ such that
\begin{equation}\label{mudi}
\E_{d \in [-N,N]} \mu(d) = 1 + O(\eps)
\end{equation}
and
\begin{equation}\label{mudi-2}
 \sup_{d \in [-N,N]} \mu(d) \ll_{\alpha,\eps} 1
\end{equation}
such that
\begin{equation}\label{mudi-3}
 \E_{n \in [N]; d \in [-N,N]} 1_A(n) 1_A(n+d) \ldots 1_A(n+(k-1)d) \mu(d) \geq \alpha^k - O(\eps).
\end{equation}
\end{theorem}

Indeed, from \eqref{mudi}, \eqref{mudi-3}, we see that we have
$$
 \E_{n \in [N]} 1_A(n) 1_A(n+d) \ldots 1_A(n+(k-1)d) \geq \alpha^k - O(\eps)$$
for all $d$ in a subset $E$ of $[-N,N]$ with $\E_{d \in [-N,N]} 1_E(d) \mu(d) \gg_{\alpha,\eps} 1$.  From \eqref{mudi-2} we conclude that $|E| \gg_{\alpha,\eps} N$, and Theorem \ref{bhk-thm} follows (after shrinking $\eps$ by an absolute constant). Conversely, it is not difficult to deduce Theorem \ref{bhk-thm} from Theorem \ref{bhk-thm2}.

It remains to establish Theorem \ref{bhk-thm2}.  We may assume that $N$ is large depending on $\alpha,\eps$ as the claim is trivial otherwise (just take $\mu$ to be the Kronecker delta function at $0$).

For $k=1$ one can simply take $\mu \equiv 1$.  For $k=2$, we first observe that
$$ \E_{n \in [N]} \E_{h \in [-\eps N,\eps N]} 1_A(n+h) = \alpha + O(\eps);$$
applying Cauchy-Schwarz we conclude that
$$ \E_{h,h' \in [-\eps N,\eps N]} \E_{n \in [N]} 1_A(n+h) 1_A(n+h') \geq \alpha^2 - O(\eps).$$
The claim then follows, with $\mu$ being the probability density function of $h-h'$ as $h,h'$ range uniformly in $[-\eps N, \eps N]$.

Now we turn to the cases $k=3,4$.   Let $\Grow: \R^+ \to \R^+$ be a sufficiently rapidly growing function depending on $\alpha,\eps$ in a manner to be specified later.  We apply Theorem \ref{strong-reg} with $s:=k-2$ to obtain a quantity $M = O_{\eps,\Grow}(1)$ and a decomposition
\begin{equation}\label{decomposition} 1_A(n) = f_{\nil}(n) + f_{\sml}(n) + f_{\unf}(n)\end{equation}
such that
\begin{itemize}
\item[(i)] $f_\nil(n)$ is a $(\Grow(M),N)$-irrational degree $\leq k-2$ virtual nilsequence of complexity at most $M$ and scale $N$;
\item[(ii)] $f_\sml$ has an $L^2[N]$ norm of at most $\eps/100$;
\item[(iii)] $f_\unf$ has an $U^{k-1}[N]$ norm of at most $1/\Grow(M)$;
\item[(iv)] $f_\nil, f_\sml, f_\unf$ are all bounded in magnitude by $1$; and
\item[(v)] $f_\nil$ and $f_{\nil} + f_{\sml}$ are non-negative.
\end{itemize}

It is clear that $| \E_{n \in [N]} f_{\sml}(n)| = O(\eps)$, and furthermore, by Theorem \ref{cs-lemma} (setting all but one of the functions equal to $1$) we also have $|\E_{n \in [N]} f_{\unf}(n)| = O(\eps)$ if $\Grow$ grows rapidly enough.
Therefore 
\begin{equation}\label{nnil}
 \E_{n \in [N]} f_\nil(n) \geq \alpha - O(\eps).
\end{equation}

The heart of the matter is the following proposition.

\begin{proposition}[Bergelson-Host-Kra for $f_\nil$]\label{bhk-prop}  Let $k=3,4$.  Then there exists a non-negative $(k-2)$-step nilsequence $\mu: \Z \to \R^+$ of complexity $O_{\alpha,\eps,M}(1)$ obeying the normalisation
\begin{equation}\label{ddn}
\E_{d \in [N]} \mu(d) = 1 + O(\eps)
\end{equation}
and such that
\begin{equation}\label{fnil}
 \E_{n,d \in [N]}  f_\nil(n) f_\nil(n+d) \ldots f_\nil(n+(k-1)d)\mu(d)  \geq \alpha^k - O( \eps).
\end{equation}
\end{proposition}

\begin{proof}[Deduction of Theorem \ref{bhk-thm2} from Proposition \ref{bhk-prop}.]  Using \eqref{decomposition}, one can expand the left-hand side of \eqref{mudi-3} into $3^k$ terms, one of which is \eqref{fnil}.  As for the other terms, any term involving at least one copy of $f_\unf$ is of size $O_{\alpha,\eps,M}(1/\Grow(M))$ by Lemma \ref{gvn-twist} and the $U^{k-1}$ norm bound on $f_\unf$.  Finally, consider a term that involves at least one copy of $f_\sml$.  Suppose first that we have a term that involves $f_\sml(n)$.  Then after performing the average in $d$ using \eqref{ddn}, we see that this term is $O( \E_{n \in [N]} |f_\sml(n)|)$, which is $O(\eps)$ by the $L^2[N]$ bound on $f_\sml$ and the Cauchy-Schwarz inequality.  Similarly for any term that involves $f_\sml(n+id)$, after making a change of variables $(n',d) := (n+id,d)$.  Putting all this together we obtain the result. 
\end{proof}

It remains, of course, to establish Proposition \ref{bhk-prop}.  We may assume that $N$ is sufficiently large depending on $\alpha,\eps,M$, as the claim is trivial otherwise by taking $\mu$ to be a delta function.

We first establish the proposition in the easier of the two cases, namely the case $k=3$. This was previously considered in \cite{green-regularity}.  In this case it is actually easier to work with the (easier) weak regularity lemma, Proposition \ref{semi-reg}, in which the degree 1 polynomial sequence $g(n)$ is not required to be irrational. Note that we have not made any use of irrationality so far, though we shall do so later when discussing the case $k = 4$. We may identify $G/\Gamma$ with $(\R/\Z)^m$ for some $m = O_M(1)$ and, by modulating $F$ if necessary, we may suppose that $g(n) = \theta n$ is linear with no constant term, where $\theta \in \R^m$. Then $$ f_\nil(n) = F( n \theta ),$$ where $ F: (\R/\Z)^m \to \C$ has Lipschitz norm $O_M(1)$.  

Let $\eps' > 0$ be a small number depending on $\eps$ and $M$ to be chosen later, and let $B_1, B_2 \subseteq [-N,N]$ denote be the two Bohr sets 
$$ B_1 := \{ d \in [-\eps' N,\eps' N]:  \dist_{(\R/\Z)^{m}}(\theta d,0) \leq \eps'\}$$
and
$$ B_2 := \{ d \in [-\eps' N, \eps' N] : \dist_{(\R/\Z)^m}(\theta d ,0) \leq \eps'/2\}.$$
By the usual Dirichlet pigeonhole argument we see that $|B_2| \gg_{\eps',M} N$.  Also, from the Lipschitz nature of  $F$, we see that
$$ f_\nil(n+d) = f_\nil(n) + O_M(\eps')$$
whenever $d \in B_1$ and $n \in [-(1-\eps')N, (1-\eps')N]$.  As a consequence, it follows that 
$$ \E_{n \in [N]} f_\nil(n) f_\nil(n+d) f_\nil(n+2d) = \E_{n \in \N} f_\nil(n)^3 + O_M(\eps')$$
for such $d$.  However from \eqref{nnil} and H\"older's inequality one has
$$ \E_{n \in \N} f_\nil(n)^3 \geq \alpha^3 - O(\eps).$$
Proposition \ref{bhk-prop} (in the case $k = 3$) now follows by taking $\mu(d) = c \psi(\theta d)$, where $\psi : (\R/\Z)^m \rightarrow [0,1]$ is an $O_{M,\eps'}(1)$-Lipschitz function which is $1$ on $B_2$ and $0$ outside $B_1$, $c = O_{M,\eps'}(1)$ is a suitable normalisation constant, and by taking $\eps'$ to be suitably small.

We now turn to the $k=4$ case of Proposition \ref{bhk-prop}.  For simplicity let us first consider the model case when $f_\nil$ is a genuine nilsequence and not just a virtual nilsequence, that is to say
\begin{equation}\label{fnil-form}
f_\nil(n) = F(g(n) \Gamma)
\end{equation}
where $(G/\Gamma,G_\bullet)$ is a degree $\leq 2$ filtered nilmanifold of complexity $O_M(1)$, and $g \in \poly(\Z,G_\bullet)$ is $(\Grow(M),N)$-irrational.  By Taylor expansion (see Appendix \ref{appendix-a}), we have
$$ g(n) = g_0 g_1^n g_2^{\binom{n}{2}}$$
for some $g_0, g_1 \in G$ and $g_2 \in G_{(2)}$.  The $(\Grow(M),N)$-irrationality of $g$ ensures certain irrationality properties on $g_1$ and $g_2$, though we will not need these properties explicitly here, as we will only be using them through the counting lemma (Theorem \ref{count-lem}), which we shall be using as a black box.

Let $\pi: G \to T_1$ be the projection homomorphism to the torus\footnote{Note this is not quite the same thing as the \emph{horizontal torus}, which is so important in \cite{green-tao-nilratner}, which is $(G/\Gamma)_{\ab} := G/[G,G]\Gamma$.} $T:= G/(G_{(2)} \Gamma)$. Then
$$ \pi(g(n)) = \pi(g_0) \pi(g_1)^n.$$
Let $\eps' > 0$ be a small quantity depending on $\eps, M$ to be chosen later.  We set 
$$ \mu(d) := c 1_{[-\eps'N,\eps'N]}(d) \phi(\pi(g_1)^d),$$
where, much as in the analysis of the case $k = 3$, $\phi: T_1 \to \R^+$ is a smooth non-negative cutoff to the ball of radius $\eps'$ centered at the origin that is not identically zero, and $c$ is a normalisation constant to be chosen shortly.  From Theorem \ref{count-lem} one has
$$ \E_{d \in [-\eps'N,\eps'N]} \phi(\pi(g_1)^d) = \int_{T_1} \phi + o_{\Grow(M) \to \infty; \eps',M}(1)
+ o_{N \to \infty; \eps',M}(1).$$
Thus if we set
\begin{equation}\label{ctp}
 c := \frac{1}{\int_{T_1} \phi} = O_{\eps',M}(1)
 \end{equation}
then we have the normalisation \eqref{ddn}, if $\Grow$ is sufficiently rapid, depending on the way in which $\eps'$ depends on $\eps,M$, and $N$ is sufficiently large depending on $\eps,\eps',M$.
From the bound on $c$ we see that $\mu$ is a degree $\leq 1$ (and hence also degree $\leq 2$) nilsequence of complexity $O_{\eps',M}(1)$.

We now apply the counting lemma, Theorem \ref{count-lem}, to conclude that
\begin{align}\nonumber 
\E_{n,d \in [N]} f_\nil(n) f_\nil(n+d)&  f_\nil(n+2d) f_\nil(n+3d)\mu(d) \\ & = \int_{G^{\Psi}/\Gamma^{\Psi}} \tilde F + o_{\Grow(M) \to \infty; \eps',M}(1) + o_{N \to \infty; \eps',M}(1) \label{counting}
\end{align}
where $G^{\Psi} \subseteq G^4$ is the Leibman group associated to the (translation-invariant) collection $\Psi = (\psi_0,\psi_1,\psi_2,\psi_3): \Z^2 \to \Z^4$ of linear forms $\psi_i(\mathbf{n}) := n_1 + i n_2$, $i = 0,1,2,3$, that is to say the Hall-Petresco group $\HP^4(G)$, and $\tilde F: G^{\Psi} \to \C$ is the function
$$ \tilde F( x_0, x_1, x_2, x_3 ) := c \phi(\pi(x_1) \pi(x_0)^{-1}) F(x_0) F(x_1) F(x_2) F(x_3)$$
(here we use the identity $\pi(g(n+d))^{-1} \pi(g(n)) = \pi(g_1)^d$, immediately verified from the Taylor expansion).  

We now do some calculations in the Hall-Petrseco group very similar to those in \cite{bergelson-host-kra}. We saw in \S \ref{counting-sec} that
\[ G^{\Psi} = \{ (g_0, g_0 g_1 , g_0 g_1^2 g_2, g_0 g_1 g_2^3) : g_0 , g_1 \in G, g_2 \in G_{(2)}\} \] (note, of course, that $G_{(3)} = \id$ in the case we are considering).
For our calculations it is convenient to use the following obviously equivalent representation:
\begin{align*} G^{\Psi} = \{ (g_0 g_{2,0}, g_0 g_1 g_{2,1}, &  g_0 g_1^2 g_{2,2}, g_0 g_1^3 g_{2,3}): g_0,g_1 \in G; \\ &  g_{2,0},\ldots,g_{2,3} \in G_{(2)}; g_{2,0} g_{2,1}^{-3} g_{2,2}^3 g_{2,3}^{-1} = \id \}.\end{align*}
Here we have taken note of the fact that 
\[ \Psi^{[2]} = \{ (x_0, x_1, x_2, x_3) \in \R^4 : x_0 - 3x_1 + 3x_3 - x_3 = 0\}.\] This last equation is quite special in that it exhibits a certain ``positivity'', as we shall see later; this is key to our argument. 
The lattice $\Gamma^{\Psi}$ can be similarly described by requiring $g_0, g_1, g_{2,0},\ldots,g_{2,3}$ to also lie in $\Gamma$.  As a consequence of this, an arbitrary point of the nilmanifold $G^{\Psi}/\Gamma^{\Psi}$ can be parameterised uniquely as
\begin{equation}\label{point}
 (g_0 g_{2,0}, g_0 g_1 g_{2,1}, g_0 g_1^2 g_{2,2}, g_0 g_1^3 g_{2,3}) \Gamma^{\Psi}
\end{equation}
where $g_0, g_1$ lie in a fundamental domain $\Sigma_1 \subset G$ of the horizontal torus $T_1$ (i.e. a smooth manifold with boundary on which $\pi$ is a bijection from $\Sigma_1$ to $T_1$), and $g_{2,0},\ldots,g_{2,3}$ lie in a fundamental domain $\Sigma_2 \subset G_{(2)}$ of the \emph{vertical torus} $T_2 := G_{(2)} / \Gamma_{(2)}$ subject to the constraint $g_{2,0} g_{2,1}^{-3} g_{2,2}^3 g_{2,3}^{-1} \in \Gamma_{(2)}$.  For such a point \eqref{point}, the function $\tilde F$ takes the value
$$ c \phi( \pi(g_1) ) \prod_{j=0}^3 F( g_0 g_i^j g_{2,j} \Gamma ).$$
On the support of $\phi$, $g_1$ is a distance $O_M(\eps')$ from the identity (if the fundamental domain $\Sigma_1$ was chosen in a suitably smooth fashion), and so by the Lipschitz nature of $F$ and the boundedness of $g_0$ we have
$$ F( g_0 g_i^j g_{2,j} ) = F( g_0 g_{2,j} \Gamma ) + O_M(\eps').$$
As a consequence, the integral $\int_{G^{\Psi}/\Gamma^{\Psi}} \tilde F$ can be expressed as
\begin{equation}\label{lam}
c \int_{g_0 \in \Sigma_1} \int_{g_1 \in \Sigma_1} \phi(\pi(g_1)) \big( \int_{\substack{g_{2,0},\ldots,g_{2,3} \in T_2\\ g_{2,0} g_{2,1}^{-3} g_{2,2}^3 g_{2,3}^{-1}=\id}} \prod_{j=0}^3 F( g_0 g_{2,j} \Gamma ) + O_M(\eps')\big)
\end{equation}
where all integrals are with respect to Haar measure.

Let $\xi\in \hat T_2$ be a vertical character, i.e. a continuous homomorphism from $T_2$ to $\R/\Z$.  For any $x \in  G/\Gamma$, we can define the \emph{vertical Fourier transform} $\hat F(x,\xi)$ to be the quantity
$$ \hat F(x,\xi) := \int_{g_2 \in T_2} e(-\xi(g_2)) F(g_2 x).$$
From the Fourier inversion formula we have
$$ \int_{\substack{g_{2,0},\ldots,g_{2,3} \in T_2 \\ g_{2,0} g_{2,1}^{-3} g_{2,2}^3 g_{2,3}^{-1}=\id}} \prod_{j=0}^3 F( g_0 g_{2,j} \Gamma ) = \sum_{\xi \in \hat T_2} |\hat F(g_0,\xi)|^2 |\hat F(g_0,3\xi)|^2.$$
In particular, we have\footnote{This is the ``positivity'' alluded to earlier.  The argument is essentially that used in \cite{bergelson-host-kra}  and it is special to the $k=4$ case, which is of course consistent with the failure of Theorem \ref{bhk-thm2} to extend to $k \geq 5$.}
$$ \int_{\substack{g_{2,0},\ldots,g_{2,3} \in T_2 \\ g_{2,0} g_{2,1}^{-3} g_{2,2}^3 g_{2,3}^{-1}=\id}} \prod_{j=0}^3 F( g_0 g_{2,j} \Gamma ) \geq |\hat F(g_0,0)|^4.$$
Inserting this bound and \eqref{ctp} into \eqref{lam}, we conclude that
$$
\int_{G^{\Psi}/\Gamma^{\Psi}} \tilde F \geq
\int_{g_0 \in \Sigma_1} |\hat F(g_0 \Gamma,0)|^4 - O_M(\eps') - o_{\Grow(M) \to \infty; \eps',M}(1).$$
From Fubini's theorem we have
$$ \int_{g_0 \in \Sigma_1} \hat F(g_0 \Gamma,0) = \int_{G/\Gamma} F$$
and from Theorem \ref{count-lem}, \eqref{fnil-form} and \eqref{nnil} we have
$$ \int_{G/\Gamma} F = \alpha + O(\eps) + o_{\Grow(M) \to \infty; \eps',M}(1) + o_{N \to \infty; \eps',M}(1)
.$$
Applying H\"older's inequality, we conclude that
$$
\int_{G^{\Psi}/\Gamma^{\Psi}} \tilde F \geq
\alpha^4 - O(\eps) - O_M(\eps') - o_{\Grow(M) \to \infty; \eps',M}(1) - o_{N \to \infty; \eps',M}(1),$$
and so \eqref{fnil} follows from \eqref{counting}, if $\eps'$ is sufficiently small depending on $\eps,M$, $\Grow$ is sufficiently rapid depending on $\eps$, and $N$ is sufficiently large depending on $\eps',M$.

This concludes the proof of the $k=4$ case of Proposition \ref{bhk-prop} in the special case when $f_\nil(n) = F(g(n)\Gamma)$ with $g$ irrational. Unfortunately Theorem \ref{strong-reg} requires us to deal with the somewhat more general setting of virtual nilsequences, in which there is dependence on $n \mod q$ or $n/N$. The extra details required are fairly routine but notationally irritating. Let us now suppose, then, that
\begin{equation}\label{fnil-form-2}
f_\nil(n) = F(g(n) \Gamma, n \mod q, n/N).
\end{equation}
We let $\eps'$ be as before, but modify $\mu$ to now be given by
$$ \mu(d) := q 1_{q|d} c 1_{[-\eps'N,\eps'N]}(d) \phi(\pi(g_1)^d),$$
with $c$ still chosen by \eqref{ctp}.  As before, one can use Theorem \ref{count-lem} to establish \eqref{ddn}.

Now consider the left-hand side of the expression \eqref{fnil} we are to bound in Proposition \ref{bhk-prop}, that is to say  
\begin{equation}\label{to-bound-below} \E_{n,d \in [N]} f_\nil(n) f_\nil(n+d)f_\nil (n+2d) f_\nil (n + 3d) \mu(d). \end{equation}
Splitting into residue classes modulo $q$, we can express this as
\begin{align*}
c \E_{r \in [q]} \E_{n \in [N/q]} \E_{d \in [-\eps'N/q,\eps'N/q]} &\prod_{i=0}^3 F(g(qn+qid+r)\Gamma, r, \\ & q(n+ir)/N) \phi(\pi(g_1)^{qd}) + O_{N\to \infty;\eps',M}(1).
\end{align*}
We partition $[N/q]$ into intervals $P$ of length $\lfloor \eps' N\rfloor$ (plus a remainder of cardinality $O(\eps' N)$).  We can then rewrite the above expression as 
\begin{align*}
c \E_P \E_{r \in [q]} \E_{n \in P} & \E_{d \in [-\eps'N/q,\eps'N/q]} \prod_{i=0}^3 F(g(qn+qid+r)\Gamma, r,\\ & q(n+ir)/N) \phi(\pi(g_1)^{qd}) + O(\eps')+O_{N\to \infty;\eps',M}(1).
\end{align*}
For each such expression, we can use the Lipschitz nature of $F$ to replace $q(n+ir)/N$ by $qn_P/N$, where $n_P$ is an arbitrary element of $P$, losing only an error of $O_M(\eps')$.  The above expression thus becomes
\begin{align*} c \E_P \E_{r \in [q]} \E_{n \in P} \E_{d \in [-\eps'N/q,\eps'N/q]} & \prod_{i=0}^3 F(g(qn+qid+r)\Gamma, r, qn_P/N) \phi(\pi(g_1)^{qd}) \\ & + O_M(\eps') +O_{N\to \infty;\eps',M}(1).\end{align*}
Because the orbit $n \mapsto g(n) \Gamma$ is $(\Grow(M),N)$-irrational, we see from Lemma \ref{scaling-lemma} that shifted translate $n \mapsto g(q(n+n_P)+r) \Gamma$ is $(\gg_M \Grow(M),N)$-irrational.  We may then argue as in the previous case and bound the above average below by
\begin{align*} \geq \E_P \E_{r \in [q]} |\int_{G/\Gamma} F(\cdot,r,qn_P/N)|^4 & - O(\eps) - O_M(\eps')  \\ &- o_{\Grow(M) \to \infty; \eps',M}(1) - o_{N \to \infty;\eps',M}(1).\end{align*}
Using Theorem \ref{count-lem} again, we have
$$ \E_{n \in P} f_\nil(qn+r) = \int_{G/\Gamma} F(\cdot,r,qn_P/N) + o_{\Grow(M) \to \infty; \eps',M}(1) + o_{N \to \infty;\eps',M}(1)$$
and so \eqref{to-bound-below} is at least
\begin{align*} \geq \E_P \E_{r \in [q]} |\E_{n \in P} f_\nil(qn+r)|^4 & - O(\eps) - O_M(\eps') \\ &- o_{\Grow(M) \to \infty; \eps',M}(1) - o_{N \to \infty;\eps',M}(1).\end{align*}
Now from \eqref{nnil} and double-counting one has
$$ \E_P \E_{r \in [q]} \E_{n \in P} f_\nil(qn+r) = \alpha + O(\eps)$$ and so, from H\"older's inequality, we deduce that \eqref{to-bound-below} is 
\[ \geq \alpha^4 - O(\eps) - O_M(\eps') - o_{\Grow(M) \to \infty; \eps',M}(1) - o_{N \to \infty;\eps',M}(1).\]
Proposition \ref{bhk-prop} now follows by once again choosing $\eps'$ small enough depending on $\eps,M$, and choosing $\Grow$ rapid enough depending on $\eps$, and $N$ sufficiently large depending on $\eps,\eps',M$.

\newcommand\vol{\operatorname{vol}}
\section{Proof of Szemer\'edi's theorem}\label{szem-sec}

We turn now to the proof of Szemer\'edi's theorem. We deemed this result too famous to state in the introduction but, for the sake of fixing notation, we recall it here now. It is most natural to establish what might be called the ``functional'' form of the theorem which is \emph{a priori} a stronger statement (though quite easily shown to be equivalent to the standard formulation by an argument of Varnavides \cite{var}).

\begin{theorem}[Szemer\'edi's theorem]\label{szthm}
Let $0 < \alpha \leq 1$, let $k \geq 3$, and let $N \geq 1$.   If $f : [N] \rightarrow [0,1]$ is a function with $\E_{n \in [N]} f(n) \geq \alpha$ then 
\[ \Lambda_k(f,f,\dots,f) \gg_{k,\alpha} 1,\] 
where
\[ \Lambda_k(f_1,\dots,f_k) := \E_{n \in [N]; d \in [-N,N]} f_1(n) f_2(n+d) \dots f_k(n + (k-1)d)\] is the multilinear operator counting arithmetic progressions.
\end{theorem}

We now prove this theorem.  We fix $k, \alpha$, and allow implied constants to depend on these quantities.

As usual, we begin by applying the regularity lemma, Theorem \ref{strong-reg}.  In view of the generalised von Neumann theorem, Theorem \ref{cs-lemma}, it is natural to apply this theorem with $s=k-2$ (which, as remarked in \S \ref{gvn-sec}, is the Cauchy-Schwarz complexity $s = s(\Psi)$ of the system $\Psi$ of  linear forms $n_1,n_1+n_2,\ldots,n_1+(k-1)n_2$).  If we do so, with a small parameter $\eps > 0$ depending on $\alpha,k$ to be chosen later, and a growth function $\Grow$ depending on $\alpha,k,\eps$ to be specified later, we obtain a decomposition
\begin{equation}\label{f-expand}
f(n) = f_{\nil}(n) + f_{\sml}(n) + f_{\unf}(n)
\end{equation}
where
\begin{itemize}
\item[(i)] $f_\nil$ is a $(\Grow(M),N)$-irrational degree $\leq k-2$ virtual nilsequence of complexity $\leq M$ and scale $N$;
\item[(ii)] $f_\sml$ has an $L^2[N]$ norm of at most $\eps$;
\item[(iii)] $f_\unf$ has an $U^{k-1}[N]$ norm of at most $1/\Grow(M)$;
\item[(iv)] $f_\nil, f_\sml, f_\unf$ are all bounded in magnitude by $1$; and
\item[(v)] $f_\nil$ and $f_\nil+f_\sml$ are non-negative.
\end{itemize}

As we shall soon see, the contribution of $f_\unf$ can be quickly discarded using the generalised von Neumann theorem.  If one could also easily discard the contribution of the small term $f_\sml$, then matters would simply reduce to verifying that the contribution of $f_\nil$ is bounded away from zero, which would be an easy consequence of the counting lemma.  Unfortunately the small term $f_\sml$ is only moderately small (of size $O(\eps)$) rather than incredibly small (e.g. of size $O(1/\Grow(M))$), and so one has to take a certain amount of care in dealing with this term, which makes the analysis significantly more delicate\footnote{In the language of ergodic theory, the problem here is that the characteristic factor is not necessarily a nilsystem, but may merely be a pro-nilsystem - an inverse limit of nilsystems.}.

We turn to the details.  
Much as the key to proving Theorem \ref{bhk-thm} was to establish Proposition \ref{bhk-prop}, the key to establishing Szemer\'edi's theorem is the following proposition.

\begin{proposition}[Szemer\'edi for $f_\nil$]\label{szem-prop}  Let $f_\nil$ be as above, and let $\eps > 0$.  Then there exists a function $\mu: \Z \times \Z  \to \R^+$ supported on the set
\begin{equation}\label{nand}
 \{ (n,d) \in \Z \times \Z: d \in [-\eps N,\eps N]; n+id \in [N] \hbox{ for all } i=0,\ldots,k-1 \}
 \end{equation}
with
\begin{equation}\label{mu-norm}
\E_{n \in [N]; d \in [-\eps N,\eps N]} \mu(n,d) = 1 + O(\eps)
\end{equation}
and with $\mu$ bounded in magnitude by $O_{M,\eps}(1)$, such that
\begin{equation}\label{fnil-near}
f_\nil(n+id) = f_\nil(n) + O(\eps)
\end{equation}
whenever $0 \leq i \leq k-1$ and $\mu(n,d) \neq 0$, and such that one has the equidistribution property
\begin{equation}\label{note}
 \E_{n \in [N]} |\E_{d \in [-\eps N,\eps N]} \mu(n-id,d)|^2 = 1 + O(\eps)
\end{equation}
for all $0 \leq i \leq k-1$.
\end{proposition}

The crucial feature of Proposition \ref{szem-prop} for us is that, with the exception of the uniform bound on $\mu$, the error terms here decay as $\eps \to 0$, even if the complexity bound $M$ on $f_\nil$ is extremely large compared to $1/\eps$.

The reader may benefit from a few words about the role of the function $\mu$. Supposing that $f_{\nil}(n) = F(g(n)\Gamma)$ is a genuine nilsequence, this function acts like a kind of ``weight'' on progressions $(n, n+d,\dots, n + (k-1)d)$ which are ``almost diagonal'' in the sense that $g(n)\Gamma \approx \dots \approx g(n+(k-1)d)\Gamma$ in $G/\Gamma$. The condition \eqref{note} reflects the fact that the weighted number of almost diagonal progressions whose $i$th point is $n$ is roughly independent of $n$. This ``non-concentration'' of almost diagonal progressions ultimately means that the error $f_\sml$ cannot destroy too many of these progressions, a fact that is crucial for our argument.

Let us assume Proposition \ref{szem-prop} for now and see how it implies Theorem \ref{szthm}.  We use \eqref{f-expand} to expand out the form $\Lambda_k(f,\ldots,f)$ into $3^k$ terms.  By Theorem \ref{cs-lemma}, any term that involves $f_\unf$ will be of size $O( 1 / \Grow(M) )$, thus
\begin{equation}\label{lack}
 \Lambda_k(f,\ldots,f) = \Lambda_k(f_\nil+f_\sml,\ldots,f_\nil+f_\sml) + O(1/\Grow(M)).
\end{equation}
Next, we use the weight $\mu$ arising from Proposition \ref{szem-prop} and the non-negativ-ity of $f_\nil+f_\sml$ guaranteed by Theorem \ref{strong-reg} to write
\begin{align*}
 & \Lambda_k  (f_\nil+f_\sml,\ldots,f_\nil+f_\sml)  \\ & \gg_{M,\eps} 
 \E_{n \in [N]; d \in [-\eps N, \eps N]} (f_\nil+f_\sml)(n) \ldots
(f_\nil+f_\sml)(n+(k-1)d) \mu(n,d).\end{align*}
We then expand this latter average into the sum of $2^k$ terms.  The main term is
\begin{equation}\label{mainterm}
 \E_{n \in [N]; d \in [-\eps N, \eps N]} f_\nil(n) \ldots
f_\nil (n+(k-1)d) \mu(n,d),
\end{equation}
and the other terms are error terms, involving at least one factor of $f_\sml$.

Consider one of the error terms, involving the factor $f_\sml(n+id)$ (say) for some $0 \leq i \leq k-1$.  We can bound the contribution of this term by
$$ \E_{n \in [N]; d \in [-\eps N, \eps N]} |f_\sml(n+id)| \mu(n,d),$$
which by a change of variables $n \mapsto n-id$ we can write as
$$ \E_{n \in [N]} |f_\sml(n)| \E_{d \in [-\eps N,\eps N]} \mu(n-id,d).$$
By Cauchy-Schwarz, \eqref{note}, and the $L^2[N]$ bound on $f_\sml$, this is $O(\eps)$.  

Finally, we look at the main term \eqref{mainterm}.  Using \eqref{fnil-near} we can approximate
$$ f_\nil(n) \ldots f_\nil (n+(k-1)d) = f_\nil(n)^k + O(\eps)$$
and so (using \eqref{mu-norm}) we can write \eqref{mainterm} as
$$ \E_{n \in [N]} f_\nil(n)^k \E_{d \in [-\eps N,\eps N]} \mu(n,d) + O(\eps).$$
Now, from \eqref{mu-norm} one has
$$ \E_{n \in [N]} \E_{d \in [-\eps N,\eps N]} \mu(n,d) = 1 + O(\eps)$$
and hence by \eqref{note}
$$ \E_{n \in [N]} |\E_{d \in [-\eps N,\eps N]} \mu(n,d) - 1|^2 = O(\eps).$$
In particular, by Chebyshev's inequality, we have
$$ \E_{d \in [-\eps N,\eps N]} \mu(n,d) = 1 + O(\eps^{1/3})$$
for all $n \in E$, where $E \subseteq [N]$ has cardinality $|E| \geq (1-O(\eps^{1/3})) N$.  Thus, for $\eps$ small enough, we can bound \eqref{mainterm} from below by
$$ \gg \E_{n\in [N]} 1_E(n) f_\nil(n)^k - O(\eps^{1/3}).$$
Now from hypothesis we have $\E_{n \in [N]} f(n) \gg 1$.  From Cauchy-Schwarz we have \[ \E_{n \in [N]} f_\sml(n) = O(\eps),\] and from Theorem \ref{cs-lemma} we also have \[ \E_{n \in [N]} f_\unf(n) = O(\eps)\] if $\Grow$ is rapid enough.  Thus if $\eps$ is small enough we have $\E_{n \in [N]} f_\nil(n) \gg 1$,
which implies that
$ \E_{n \in [N]} 1_E(n) f_\nil(n) \gg 1$, and hence by H\"older's inequality that
$ \E_{n \in [N]} 1_E(n) f_\nil^k(n) \gg 1$.
Putting all this together, we conclude that \eqref{mainterm} is $\gg 1$ if $\eps$ is small enough, and thus
$$  \Lambda_k(f_\nil+f_\sml,\ldots,f_\nil+f_\sml) \gg_{M,\eps} 1.$$
Inserting this bound into \eqref{lack} we obtain the claim, completing the proof of Szemer\'edi's theorem,  if $\Grow$ is chosen sufficiently rapid.\vspace{11pt}

\emph{Proof of Proposition \ref{szem-prop}.}  Let us first establish this in the easy case $k=3$.  In this case, $f_\nil$ is essentially quasiperiodic, which will allow us to take $\mu(n,d)$ to be of the form 
$$ \mu(n,d) = 1_{[2\eps N, (1-2\eps)N]}(n) \mu(d)$$
with $\mu(d)$ normalised by requiring
$$ \E_{d \in [-\eps N,\eps N]} \mu(d) = 1 + O(\eps).$$
It is then easy to verify that both \eqref{mu-norm} and \eqref{note} follow from this normalisation.  To establish the remaining claims in Proposition \ref{szem-prop}, we use the degree $\leq 1$ nature of the orbit $n \mapsto g(n) \Gamma$ as in Section \ref{bhk-sec} to write $f_\nil$ as
$$ f_\nil(n) = F( n \theta )$$
for some $\theta \in (\R/\Z)^D$ with $D = O_M(1)$ and some $F: (\R/\Z)^D \to \C$ of Lipschitz constant $O_M(1)$.  If one then sets $\mu$ to equal
$$ \mu(d) := \frac{|[-\eps N,\eps N]|}{|B|} 1_B(d)$$
where $B$ is the Bohr set
$$ \{ d \in [-\eps N,\eps N]: d_{(\R/\Z)^D}(d\theta, 0) \leq \delta \}$$
and $\delta > 0$ is sufficiently small depending on $\eps, M$, one easily verifies all the required claims.

We now turn to the case $k>3$, which is harder because $f_\nil$ is no longer quasiperiodic, and so $\mu(n,d)$ will have to depend more heavily on $n$ and not just on $d$.
By arguing as in the previous section we can normalise $g(0)$ to equal $\id$.  We may also assume $N$ is sufficiently large depending on $\eps, M$, since otherwise we may simply take $\mu(n,d) = 1_{[N]}(n) \delta_0(d)$ where $\delta_0$ is the Kronecker delta function at $0$.  We may of course also assume that $\eps < 1$.

We take an $O_M(1)$-rational Mal'cev basis $X_1,\dots, X_{\dim(G)}$ for the Lie algebra ${\mathfrak g} = \log G$ adapted to the filtration $G_{\bullet}$ as described in \cite[Appendix A]{green-tao-nilratner}. For any radius $r>0$, we define the ``ball'' $B_r$ in $G$ to be the set of all group elements of the form
\begin{equation}\label{etx}
 \exp( \sum_{j=1}^{\dim(G)} t_j X_j )
\end{equation}
where the $t_j$ are real numbers with $t_j \leq r^{s+1-i}$
whenever $1 \leq i \leq s$ and $j \leq \dim(G)-\dim(G_{(i)})$.   Thus, when $r$ is small, $B_r$ is quite ``narrow'' (of diameter comparable to $r^s$) when projected down to $G/G_{(2)}$, but is relatively large when restricted to the top order component $G_{(s)}$ (of diameter comparable to $r$).  This type of eccentricity is necessary in order to make $B_r$ approximately ``normal'' with respect to conjugations.  Indeed, we have

\begin{lemma}[Approximate normality]\label{normal}  Let $A, \delta > 0$, and let $g \in G$ be such that $d_G(g,\id) \leq A$.  Then we have the containments
\begin{equation}\label{bepsr}
 B_{(1-\delta) r} \subseteq g B_r g^{-1} \subseteq B_{(1+\delta) r}.
\end{equation}
whenever $r > 0$ is sufficiently small depending on $A, \delta, M$.
\end{lemma}

\begin{proof}  We prove the second inclusion only, as the first is similar (and can also be deduced from the second).  The conjugation action $h \mapsto ghg^{-1}$ on $G$ induces a Lie algebra automorphism $\exp(\operatorname{ad}(\log g)): {\mathfrak g} \to {\mathfrak g}$.  If we conjugate the group element \eqref{etx} by $g$, we thus obtain
$$
 \exp( \sum_{j=1}^{\dim(G)} t_j \exp(\operatorname{ad}(\log g))(X_j) ).$$
 But if $1 \leq i \leq s$ and $j \leq \dim(G)-\dim(G_{(i)})$, we see from the Baker-Campbell-Hausdorff formula \eqref{comm-form} that
$$ \exp(\operatorname{ad}(\log g))(X_j) = X_j + \sum_{j'=\dim(G)-\dim(G_{(i)})+1}^{\dim(G)} c_{j,j'} X_{j'}$$
for some coefficients $c_{j,j'}$ of size $O_{A,M}( r^{s+1-i} )$.  Collecting all the coefficients together, we obtain the claim for $r$ small enough.
\end{proof}

Let $0 < \delta < 1/10$ be a small quantity (depending on $\eps, M$), let $R$ be a large quantity depending on the same parameters, and let $r_0 > 0$ be an even smaller\footnote{Readers may find it helpful to keep the hierarchy of scales 
$$ 1 \sim 1/k, \alpha \gg \eps \gg 1/M \gg \delta \gg 1/R \gg r_0 \gg r \gg 1/\Grow(M) \gg 1/N > 0$$
in mind.
}  quantity than $\delta$ (depending on $\eps, M, \delta,R$) to be chosen later.  For each $r$ with $0 < r < r_0$ take a Lipschitz function $\phi_r: G \to \R^+$ of Lipschitz norm $O_{M,r,\delta}(1)$ which is supported on $B_{r}$ and equals one on $B_{(1-\delta) r}$, and choose these functions so that $\phi_{r} \leq \phi'_r$ pointwise whenever $ 0 < r < r' < r_0$. For each such $r$, let $\Phi_r: G/\Gamma \times G/\Gamma \to \R^+$ be the induced function
$$ \Phi_r(x, x') := \sum_{g \in G: gx = x'} \phi_r(g).$$
This function $\Phi_r$ is supported near the diagonal of $G/\Gamma \times G/\Gamma$; indeed, $\Phi_r(x,x')$ is only non-zero when $x' \in B_r x$, and furthermore if $x' \in B_{(1-\delta)r} x$ then $\Phi_r(x,x') = 1$.  If $r_0$ is chosen sufficiently small depending on $M,\delta$, we conclude from Lemma \ref{normal} that we have the approximate shift-invariance
\begin{equation}\label{phoi}
\Phi_{(1-3\delta)r}(x,x') \leq \Phi_r( gx, gx' ) \leq \Phi_{(1+3\delta) r}(x, x')
\end{equation}
whenever $x,x' \in G/\Gamma$ and $g \in G$ is such that $d_G(g,\id) \leq R$ (say).

We now define our cutoff function $\mu = \mu_r$ by
\begin{equation}\label{monkey}
\mu_r(n,d) := c_r 1_{q|d} 1_{[k\eps N, (1-k\eps) N]}(n) 1_{[-\delta N,\delta N]}(d) \prod_{i=1}^{k-1} \Phi_r( g(n) \Gamma, g(n+id)\Gamma ),
\end{equation}
where $c_r>0$ is a normalisation constant to be chosen later. This function, as discussed immediately following the statement of Proposition \ref{szem-prop}, is a smooth cutoff to the set of ``almost-diagonal'' progressions in $G/\Gamma$. Specifically, $\mu_r$ is supported in \eqref{nand}, and also in the region where $g(n+id) \Gamma \in B_{r} g(n) \Gamma$, $|d| \leq \delta N$, and $q|d$ for $i=0,\ldots,k-1$.  From the Lipschitz nature of $F$ we thus have
\begin{align*} F( g(n+id) \Gamma,  (n+id) \md{q}, & (n+id)/N ) \\ & = F(g(n) \Gamma, n \md{q}, n/N ) + O_{M}(r_0)\end{align*}
for $(n,d)$ in the support of $\mu_r$, which gives \eqref{fnil-near} for $\mu_r$ if $r_0$ is sufficiently small depending on $\eps, M$.

Next, we compute the expectation of $\mu_r(n,d)$, in order to work out what the normalisation constant $c_r$ should be.  Observe that
\begin{align}\nonumber
&  \E_{n \in [N], d \in [-\eps N, \eps N]}  \mu_r(n,d) \\  \label{slash} & = \frac{\delta}{q\eps} (1+O(\eps)) c_r \times \\ \times &\E_{n \in [k\eps N, (1-k\eps) N]; d \in [-\delta N, \delta N]; q|d} \tilde \Phi_r( g(n) \Gamma, \ldots, g(n+(k-1)d) \Gamma ) ,\nonumber
\end{align}
where $\tilde \Phi_r: (G/\Gamma)^k \to \R^+$ is the function
\begin{equation}\label{far}
 \tilde \Phi_r( x_0,\ldots,x_{d-1} ) := \prod_{i=1}^{k-1} \Phi_r( x_0, x_i ).
\end{equation}
Observe that $\tilde \Phi$ has a Lipschitz norm of $O_{M,r,\delta}(1)$.  Applying Theorem \ref{count-lem}, we can express \eqref{slash} as
$$ \frac{\delta}{q\eps} (1+O(\eps)) c_r ( \int_{G^{\Psi}/\Gamma^{\Psi}} \tilde \Phi_r + o_{\Grow(M) \to \infty; M,r,\delta}(1) + 
o_{N \to \infty; M,r,\delta}(1) ),$$
where $G^{\Psi} \subseteq G^k$ is the $k^{\operatorname{th}}$ Hall-Petresco group, that is to say the Leibman group associated to the (translation-invariant) collection $\Psi = (\psi_0,\ldots,\psi_{k-1})$ of linear forms $\Psi^{(i)} := (n,d) \mapsto n+id$ for $i=0,\ldots,k-1$.  

The group $G^{\Psi}$ is a $O_M(1)$-rational subgroup of $G^k$, which itself has complexity $O_M(1)$.  Meanwhile, the function $\tilde \Phi_r$ equals $1$ on a ball of radius $r^{O_M(1)}$ centred at the identity, and is bounded by $1$ throughout.  We conclude that the quantity
$$ v_r := \int_{G^{\Psi}/\Gamma^{\Psi}} \tilde \Phi_r $$
obeys the bounds
$$ r^{O_M(1)} \ll_M v_r \leq 1.$$
Furthermore, from the properties of the functions $\phi_r$, we have the monotonicity property
$$ v_{(1-\delta)r} \leq v_r$$
for any $0 < r < r_0$.  Applying the pigeonhole principle (using the fact that polynomial growth is always slower than exponential growth), and choosing $\delta \gg_{\eps,M} 1$ sufficiently small depending on $\eps, M$, one can thus find a radius 
$$r_0 > r \gg_{r_0,\eps,\delta,M} 1$$
such that we have the regularity property
\begin{equation}\label{vr-stable}
 (1-O(\eps)) v_r \leq v_{(1-3\delta)r} \leq v_{(1+3\delta)r} \leq (1+O(\eps)) v_r.
\end{equation}
Note that this idea of picking a ``regular'' radius originates, in additive combinatorics, in Bourgain's paper \cite{bourgain-triples}. Fix from now on a value of $r$ with this property.  If we then set
\begin{equation}\label{camera-2}
c_r := \frac{q\eps}{\delta v_r}
\end{equation}
we conclude that
\begin{equation}\label{camera}
 c_r \ll_{M,r_0,\eps} 1
\end{equation}
and
$$ \E_{n \in [N], d \in [-\eps N, \eps N]} \mu_r(n,d) = 1 + O(\eps) + 
o_{\Grow(M) \to \infty; M,\eps,r_0}(1) + o_{N \to \infty; M,\eps,r_0}(1).$$
This will give \eqref{mu-norm} provided that $r_0$ is chosen to depend on $M, \eps, \delta$, that $\Grow$ is sufficiently rapid depending on $\eps$, and $N$ is sufficiently large depending on $M, \eps$.

Our remaining task, and the most difficult one, is to study the expression in \eqref{note}. That is to say, we fix $0 \leq i \leq k-1$ and consider 
\begin{equation}\label{note-2}
\E_{n \in [N]} |\E_{d \in [-\eps N,\eps N]} \mu_r(n-id,d)|^2.
\end{equation}
Using \eqref{monkey}, we can write this expression as
\begin{align*}
(1 +O(\eps)) (\frac{\eps}{q\delta} c_r)^2 & \E_{n \in [k\eps N, (1-k\eps)N]}  \E_{d,d' \in [-\delta N, \delta N]; q|d,d'}\\
&\tilde \Phi^{\otimes 2}_r( g(n-id)\Gamma,\ldots,g(n+(k-1-i)d)\Gamma, \\ & \qquad \qquad  g(n-id')\Gamma,\ldots,g(n+(k-1-i)d')\Gamma )
\end{align*}
where $\tilde \Phi^{\otimes 2}_r: (G/\Gamma)^k \times (G/\Gamma)^k \to \R^+$ is the tensor square
$$ \tilde \Phi^{\otimes 2}_r( x, x' ) := \tilde \Phi_r(x) \tilde \Phi_r(x').$$
Applying Theorem \ref{count-lem}, we can thus express \eqref{note-2} as
\begin{equation}\label{note-3}
(1+O(\eps)) (\frac{\eps}{q\delta} c_r)^2\big ( \int_{G^{\Psi^{(i)}}/\Gamma^{\Psi^{(i)}}} \tilde \Phi_r^{\otimes 2} + o_{\Grow(M) \to \infty; \eps,M,r_0}(1) + o_{N \to \infty; \eps,M,r_0}(1) \big)
\end{equation}
where $G^{\Psi^{(i)}} \subset G^{2k}$ is the Leibman group associated to the (translation-invariant) collection
$$ \Psi^{(i)} := ( \psi_{0,i},\ldots,\psi_{k-1,i}, \psi'_{0,i}, \ldots, \psi'_{k-1,i} )$$
of linear forms 
$$ \psi_{j,i}: (n,d,d') \mapsto n+(j-i)d$$
and
$$ \psi'_{j,i}: (n,d,d') \mapsto n+(j-i)d'$$ 
for $j=0,\ldots,k-1$.

We will be establishing the following claim.

\begin{claim}[Approximate factorisation]\label{approx-factor}  We have
\begin{equation}\label{facta}
 \int_{G^{\Psi^{(i)}}/\Gamma^{\Psi^{(i)}}} \tilde \Phi_r^{\otimes 2} = (1+O(\eps)) v_r^2.
 \end{equation}
\end{claim}

\emph{Proof of Proposition \ref{szem-prop} assuming Claim \ref{approx-factor}.} Substitute back into \eqref{note-3} and use \eqref{camera-2}, \eqref{camera} to conclude that
$$ \eqref{note-2} = 1 + O(\eps) + o_{\Grow(M) \to \infty; \eps,M,r_0}(1) + o_{N \to \infty; \eps,M,r_0}(1). $$
This gives the result upon choosing $r_0$ sufficiently small depending on $\eps, M, \delta$, $\Grow$ sufficiently rapid depending on $\eps$, and $N$ sufficiently large depending on $\eps, M$.

It remains to establish Claim \ref{approx-factor}. For notational simplicity we establish only the claim $i = 0$ (the others being very similar). The intuition behind this claim (and behind the key assertion that the number of almost-diagonal progressions whose $i^{\operatorname{th}}$ term is $n$ does not depend on $n$) is that the linear forms $(\psi_{0,0},\ldots,\psi_{k-1,0})$ and $(\psi'_{0,0},\ldots,\psi'_{k-1,0})$ are almost independent of each other, except for the fact that they are coupled via the obvious identity $\psi_{0,0} = \psi'_{0,0}$. 

One way to encode this formally is to note that the Leibman group $G^{\Psi^{(0)}}$ is given by
\[ H := \{ (x,x') \in G^{\Psi} \times G^{\Psi} : x_0 = x'_0\},\] a product of two copies of the Hall-Petresco group $G^{\Psi} = \HP^k(G)$ fibred over the zeroth coordinate. To prove this, one may note that the containment $G^{\Psi^{(0)}} \subseteq H$ is obvious. On the other hand, one may compute directly using the dimension formula \eqref{go} that
\[ \dim (G^{\Psi}) = \dim(G) + \sum_{i=1}^{k-2} \dim(G^{(i)}) \]
and
\[ \dim (G^{\Psi^{(0)}}) = \dim(G) + 2 \sum_{i=1}^{k-2} \dim(G^{(i)}) \]
and thus
\[ \dim(G^{\Psi^{(0)}}) = 2 \dim (G^{\Psi}) - \dim (G) = \dim (H),\] 
and so since both sides are connected, simply-connected nilpotent Lie groups (and so both are homeomorphic to their Lie algebras) we have $G^{\Psi^{(0)}} = H$.

Write $J_r$ for the integral appearing in \eqref{facta}, that is to say
\[ J_r := \int_{(x, x') \in G^{\Psi}/\Gamma^{\Psi} \times G^{\Psi}/\Gamma^{\Psi} : x_0 = x'_0} \tilde \Phi_r^{\otimes 2}(x,x').\]
Let $R$ be some quantity, and suppose that $\dist_G(g, \id) \leq R$. Then by the almost-invariance property \eqref{phoi} we have
\[ \int_{(x,x') \in G^{\Psi}/\Gamma^{\Psi} \times G^{\Psi}/\Gamma^{\Psi} : x_0 = gx'_0} \tilde \Phi_{r(1 + 3\delta)}^{\otimes 2}(x,x') \geq J_{r}.\]
Integrate this over the ball $B_R := \{g \in G: \dist_G(g,\id) \leq R\}$. Then we obtain
\[ \int_{(x,x') \in (G^{\Psi}/\Gamma^{\Psi} )^2} \lambda(x,x') \tilde \Phi_{r(1 + 3\delta)}^{\otimes 2}(x,x') \geq \vol(B_R) J_{r},\]
where $\lambda(x,x')$ is the number of $g \in B_R$ for which $x_0 = g x'_0 \md{\Gamma}$, or equivalently
\[ \lambda(x,x') := | \Gamma \cap x_0^{-1} B_R x'_0|.\] Choose representatives $x_0, x'_0$ in some fundamental domain with $x_0, x'_0 = O_M(1)$. 
By a volume-packing argument and simple geometry we then have
\[ \lambda(x,x') = \vol(B_R)(1 + o_{R \rightarrow \infty;M}(1)).\]
Comparing with the above we have
\[ v_{r(1 - 3\delta)}^2 = \int_{(x,x') \in (G^{\Psi}/\Gamma^{\Psi} )^2}  \tilde \Phi_{r(1 + 3\delta)}^{\otimes 2} \geq J_{r} (1 + o_{R \rightarrow \infty;M}(1)),\] and so by \eqref{vr-stable} we have
\[ J_r \leq (1 + O(\eps) + o_{R \rightarrow \infty;M}(1)) v_r^2.\]

This gives the upper bound for Claim \ref{approx-factor}.  The lower bound is proven similarly.  This concludes the proof of Proposition \ref{szem-prop} and thus Theorem \ref{szthm}.

\section{On a theorem of Gowers and Wolf}\label{gowers-wolf-sec}

Our aim in this section is to prove Theorem \ref{gwolf}, whose statement we recall now. 

\begin{theorem}[Theorem \ref{gwolf}]  Let $\Psi = (\psi_1,\ldots,\psi_t)$ be a system of linear forms $\psi_1,\ldots,\psi_t: \Z^D \to \Z$ satisfying the flag condition, and let $s \geq 1$ be an integer such that the polynomials $\psi_1^{s+1},\ldots,\psi_t^{s+1}$ are linearly independent.  Then for any function $f: [N] \to \C$ bounded in magnitude by $1$ \textup{(}and defined to be zero outside of $[N]$\textup{)} obeying the bound $\|f\|_{U^{s+1}[N]} \leq \delta$ for some $\delta > 0$, one has
$$ \E_{\mathbf{n}  \in [N]^D} \prod_{i=1}^t f( \psi_i(\mathbf{n}) ) = o_{\delta \to 0; s, D,t,\Psi}(1).$$
\end{theorem}

Henceforth we allow all implied constants to depend on $d,t,s,\Psi$ without indicating this explicitly. Let $s' = s'(\Psi)$ be the Cauchy-Schwarz complexity of the linear forms $\Psi$, as defined in Theorem \ref{cs-lemma}. We may of course assume that $s'>s$, as Theorem \ref{gwolf} is immediate otherwise. We may also assume that $N$ is large depending on $\delta$, since otherwise the claim is trivial from a compactness argument.

Let $\eps > 0$ be a small number depending on $\delta$ to be chosen later, and let $\Grow$ be a growth function depending on $\eps$ to be chosen later.  Applying  Theorem \ref{strong-reg} at degree $s'$ (after first decomposing $f$ as a linear combination of $O(1)$ functions taking values in $[0,1]$), we can find a positive quantity $M = O_{\eps,\Grow}(1)$ and a decomposition
\begin{equation}\label{fdecomp}
f = f_\nil + f_\sml + f_\unf
\end{equation}
where:

$f_\nil$ is a $(\Grow(M),N)$-irrational virtual nilsequence of degree $\leq s'$, complexity $\leq M$, and scale $N$;

$f_\sml$ has $L^2[N]$ norm at most $\eps$;

$f_\unf$ has $U^{s'+1}[N]$ at most $1/\Grow(M)$;

All functions $f_\nil, f_\sml, f_\unf$ are bounded in magnitude by $O(1)$.

We apply this decomposition to split the expression
\begin{equation}\label{gw-form}
 \E_{\mathbf{n} \in [N]^D} \prod_{i=1}^t f( \psi_i(\mathbf{n}) ) 
\end{equation}
as the sum of $3^t$ terms, in which each copy of $f$ has been replaced with either $f_\nil$, $f_\sml$, or $f_\unf$.

Any term involving at least one factor of $f_\sml$ can be easily seen to be of size $O(\eps)$ by crudely estimating all other factors by $1$.  By \eqref{end}, any term involving at least one factor of $f_\unf$ is of size $O(1/\Grow(M))$, which is also of size $O(\eps)$ if $\Grow$ is chosen to be sufficiently rapidly growing depending on $\eps$.  We can therefore express \eqref{gw-form} as
$$
 \E_{\mathbf{n}\in [N]^D} \prod_{i=1}^t f_\nil(\psi_i(\mathbf{n}) ) + O(\eps).
$$
By hypothesis, we can write
$$ f_\nil(n) = F( g(n) \Gamma, n \md{q}, n/N )$$
for some $q$ with $1 \leq q \leq M$, some degree $\leq s$, $(\Grow(M),N)$-irrational, orbit $n \mapsto g(n) \Gamma$ of complexity $\leq M$ and some Lipschitz function $F: G/\Gamma \times \Z/q\Z \times \R$ of norm at most $M$. The mod $q$ and Archimedean behaviour in $f_\nil$ are nothing more than technical annoyances, and we set about eliminating them now. We encourage the reader to work through the heart of the argument, starting at \eqref{yoga} below, in the model case $f_{\nil} = F(g(n)\Gamma)$.  Let $\eps'$ be a small quantity depending on $\eps, M$ to be chosen later\footnote{Readers may find it helpful to keep the hierarchy of scales 
$$ 1 \gg \eps \gg 1/M, 1/q \gg \eps' \gg 1/\Grow(M) \gg \delta \gg 1/N > 0$$
in mind.}.  We partition $[N]$ into progressions $P$ of spacing $q$ and length $\eps' N$, plus a remainder set of size at most $O_M(1)$.  We can then rewrite the above expression as
$$
 \E_{P_1,\ldots,P_D} \E_{\mathbf{n} \in P_1 \times \dots \times P_D}  \prod_{i=1}^t f_\nil( \psi_i(\mathbf{n}) ) + O(\eps).
$$
We abbreviate $P_1 \times \ldots \times P_D$ as ${\mathbf{P}}$.
For a given ${\mathbf{P}}$, observe that as $\mathbf{n}$ ranges in ${\mathbf{P}}$, the residue class of $\psi_i(\mathbf{n})$ modulo $q$ is equal to a fixed class $a_{{\mathbf{P}},i}$, and the value of $\psi_i({\mathbf{P}})/N$ differs by at most $O_M(\eps')$ from a fixed number $x_{{\mathbf{P}},i}$.  We may assume that $x_{{\mathbf{P}},i} \in [0,1]$ for each $i$, otherwise the inner expectation is zero (except for a few ``boundary'' values of ${\mathbf{P}}$ which give a net contribution of $O_M(\eps')$).

If $\eps'$ is small enough depending on $\eps,M$, the $O_M(\eps')$ error in the above discussion can be absorbed in the $O(\eps)$ error, and so we have
$$\E_{\mathbf{n} \in [N]^D} \prod_{i=1}^t f( \psi_i(\mathbf{n}) ) 
 = \E_{{\mathbf{P}}} \E_{\mathbf{n} \in {\mathbf{P}}} \prod_{i=1}^t F( g(\psi_i(\mathbf{n})) \Gamma, a_{{\mathbf{P}},i}, x_{{\mathbf{P}},i} ) + O(\eps).$$
We now apply Theorem \ref{count-lem} , which tells us the the right-hand side here is
\begin{equation}\label{yoga}
\E_{{\mathbf{P}}} \int_{G^{\Psi}/\Gamma^{\Psi} }\tilde F_{{\mathbf{P}}} + O(\eps) + o_{\Grow(M) \to\infty;M,\eps,\eps'}(1)
+ o_{N \to\infty;M,\eps,\eps'}(1),
\end{equation}
where as usual $G^{\Psi} \leq G^t$ is the Leibman group associated to the system of forms $\Psi = \{\psi_1,\dots,\psi_t\}$, and here $\tilde F_{{\mathbf{P}}}: G^{\Psi}/\Gamma^{\Psi} \to \C$ is the function
$$ \tilde F_{{\mathbf{P}}}((g_1,\ldots,g_t)\Gamma^{\Psi}) := \prod_{i=1}^t F(g_i\Gamma, a_{{\mathbf{P}},i}, x_{{\mathbf{P}},i} ).$$
The heart of the matter is to obtain an upper bound on the quantity $\E_{{\mathbf{P}}} \int_{G^{\Psi}/\Gamma^{\Psi} }\tilde F_{{\mathbf{P}}}$ appearing in \eqref{yoga}. To do this, of course, we need to make use the assumption on the forms $\psi_1,\dots,\psi_t$, as well as the fact that $\Vert f \Vert_{U^{s+1}} \leq \delta$.

The aforementioned assumption, namely that $\psi_1^{s+1},\ldots,\psi_t^{s+1}$ are linearly independent, implies that $\Psi^{[s+1]}$ is the whole of $\R^t$ which, in view of the definition of the Leibman group $G^{\Psi}$, implies that $G_{(s+1)}^t \leq G^{\Psi}$.  By Fubini's theorem, we thus have
$$ \int_{G^{\Psi}/\Gamma^{\Psi}} \tilde F_{{\mathbf{P}}} = \int_{G^{\Psi}/\Gamma^{\Psi}} \tilde F_{{\mathbf{P}},\leq s}$$
where
\begin{equation}\label{ftilp}
 \tilde F_{{\mathbf{P}},\leq s}((g_1,\ldots,g_t)\Gamma^{\Psi}) := \prod_{i=1}^t F_{\leq s}(g_i\Gamma, a_{{\mathbf{P}},i}, x_{{\mathbf{P}},i} )
\end{equation}
and $F_{\leq s}$ is defined by averaging over cosets of $G_{(s+1)}$, specifically
$$ F_{\leq s}(g\Gamma,a,x) := \int_{G_{(s+1)}/\Gamma_{(s+1)}} F(g g_{s+1} \Gamma,a,x)\ dg_{s+1}.$$
Since $F$ was Lipschitz with norm $O_M(1)$, we see that $F_{\leq s}$ is Lipschitz with norm $O_M(1)$ also.  Also, since $F$ is bounded in magnitude by $O(1)$, so is $F_{\leq s}$.

As the forms $\psi_1^{s+1},\ldots,\psi_t^{s+1}$ are independent, we see in particular that $\psi_1$ is non-zero.  This implies that the projection of $G^{\Psi}$ to the first coordinate $G$ is surjective.  Meanwhile, from \eqref{ftilp} and the boundedness of $F_{\leq s}$ we have the crude upper bound
$$
|\tilde F_{{\mathbf{P}},\leq s}((g_1,\ldots,g_t)\Gamma)| \ll |F_{\leq s}(g_1\Gamma, a_{{\mathbf{P}},1}, x_{{\mathbf{P}},1} )|.$$
From Fubini's theorem, we obtain the bound
\begin{equation}\label{tfap}
 |\int_{G^{\Psi}/\Gamma^{\Psi}} \tilde F_{{\mathbf{P}}}| \ll \int_{G/\Gamma} |F_{\leq s}(\cdot, a_{{\mathbf{P}},1}, x_{{\mathbf{P}},1} )|.
\end{equation}

To proceed further, we need a crucial smallness estimate on $F_{\leq s}$:

\begin{proposition}[$F_{\leq s}$ small in $L^2$]  For any $a \in \Z/q\Z$ and $x \in [0,1]$, one has
\begin{align*} \int_{G/\Gamma} |F_{\leq s}(\cdot,a,x)|^2 &\ll O(\eps) + O_M(\eps') +\\ & o_{\delta \to \infty; M,\eps,\eps'}(1) + o_{\Grow(M) \to \infty; M,\eps,\eps'}(1) + o_{N \to\infty;M,\eps,\eps'}(1).\end{align*}
\end{proposition}
\begin{proof} By reflection symmetry we may assume that $x \leq 1/2$.  We may also round $x$ so that $x=qn_0/N$ for some $n_0 \in [N/2q]$, as the error in doing so can be easily absorbed by the Lipschitz properties of $F_{\leq s}$.

By construction, $F_{\leq s}$ is invariant on $G_{(s+1)}$-cosets, while $F-F_{\leq s}$ integrates to zero on any such coset.  In particular, $F_{\leq s}(\cdot,a,x)$ and $F-F_{\leq s}(\cdot,a,x)$ are orthogonal, and thus
$$ \int_{G/\Gamma} |F_{\leq s}(\cdot,a,x)|^2 = \int_{G/\Gamma} F \overline{F_{\leq s}}(\cdot,a,x).$$
Applying Theorem \ref{count-lem} (really just the special case of this result asserting that $(g(n)\Gamma)$ is equidistributed, cf. Lemma \ref{irrat-equi}) and the Lipschitz nature of $F \overline{F_{\leq s}}$, the right-hand side can be written as
$$ \E_{n \in [\eps' N]} F \overline{F_{\leq s}}(g(qn+qn_0+a)\Gamma,a,x) + o_{\Grow(M) \to \infty; M,\eps,\eps'}(1) +  o_{N \to\infty;M,\eps,\eps'}(1).$$
Let $P$ be the progression $\{ qn+qn_0+a: n \in [\eps' N]\}$.  Then by a further use of the Lipschitz properties of $F$, we can rewrite the above expression as
\begin{align}\nonumber\E_{n \in P} F(g(n)\Gamma,  n \mod q, n/N&) \psi(n) + O_M(\eps') \\ &+ o_{\Grow(M) \to \infty; M,\eps,\eps'}(1) + o_{N \to\infty;M,\eps,\eps'}(1)\label{to-be-bounded}\end{align}
where
$$ \psi(n) := \overline{F_{\leq s}}(g(n)\Gamma,a,x).$$
Note that, as a consequence of the $G_{(s+1)}$-invariance of $F_{\leq s}$, $\psi(n)$ is a degree $\leq s$ nilsequence of complexity $O_M(1)$.
Now by \eqref{fdecomp} we have
$$F(g(n)\Gamma, n \mod q, n/N) = f(n) - f_\unf(n) - f_\sml(n).$$
The contribution of $f_\sml(n)$ to \eqref{to-be-bounded} is $O(\eps)$ by the Cauchy-Schwarz inequality.  Now consider the contribution of $f$.  Observe that because $F_{\leq s}$ is $G_{(s+1)}$-invariant, $\psi$ is a degree $\leq s$ nilsequence of complexity $O_M(1)$.   Meanwhile, $\|f\|_{U^{s+1}[N]} \leq \delta$ by hypothesis. Applying the converse to the inverse conjecture for the Gowers norms (first established in \cite{green-tao-u3inverse}, though for a simple proof see \cite[Appendix G]{green-tao-ziegler-u4inverse}), we see that
$$ \E_{n \in P} f(n) \psi(n) = o_{\delta \to 0; M,\eps,\eps'}(1).$$
Similarly, since $\|f_\unf\|_{U^{s'+1}[N]} \leq 1/\Grow(M)$ and $s' \geq s$, we have
$$ \E_{n \in P} f(n) \psi(n) = o_{\Grow(M) \to 0; M,\eps,\eps'}(1).$$
Putting all of these estimates together, we obtain the claim.
\end{proof}

Applying this bound and \eqref{tfap}, we can thus bound \eqref{yoga} in magnitude by
$$ 
O(\eps) + O_M(\eps') + o_{\delta \to \infty; M,\eps,\eps'}(1) + o_{\Grow(M) \to \infty; M,\eps,\eps'}(1) + o_{N \to\infty;M,\eps,\eps'}(1).$$
Choosing $\eps'$ sufficiently small depending on $M$ and $\eps$, and choosing $\Grow$ sufficiently rapidly growing depending on $\eps$, and then using the bound $M = O_{\eps,\Grow}(1)$ (and recalling that $N$ can be chosen large depending on $\delta$), we conclude that
$$ | \E_{\mathbf{n} \in [N]^D} \prod_{i=1}^t f( \psi_i(\mathbf{n}) ) | \ll \eps$$
whenever $\delta$ is sufficiently small depending on $\eps$.  Theorem \ref{gwolf} follows.\vspace{11pt}

\emph{Remark.}  It seems certain that one can extend this result to the case when one has $t$ distinct functions $f_1,\ldots,f_t: [N] \to \C$ rather than a single function $f: [N] \to \C$.  The main change in the argument would be to use a version of the regularity lemma (Theorem \ref{strong-reg}) valid for several functions simultaneously, in which one regularises the $f_1,\ldots,f_t$ using the same data $M$, $q$, $(G/\Gamma,G_\bullet)$, $g()$ (but allows each function $f_i$ to be given a separate Lipschitz function $F_i: G/\Gamma \times \Z/q\Z \times \R \to \C$).  Such a result could be obtained by straightforward modifications to the proof of Theorem \ref{strong-reg}, but we do not pursue this matter here.

\appendix

\section{Properties of polynomial sequences}\label{appendix-a}

In this appendix we collect a variety of facts and definitions concerning polynomial sequences in nilpotent groups, all of which were required at some point in the paper proper. We take for granted the definition of filtration $G_{\bullet}$ and of the group $\poly(\Z^d, G_{\bullet})$ of polynomial sequences $g : \Z^d \rightarrow G$ adapted to $G_{\bullet}$; these notions were recalled in the introduction.\vspace{11pt}

\textsc{Taylor expansions.} Polynomial sequences may be described in terms of so-called Taylor expansions. In the lemma that follows we make use of the generalised binomial coefficients 
$\binom{\mathbf{n}}{\mathbf{i}}$ are the generalised binomial coefficients
$$ \binom{(n_1,\ldots,n_D)}{(i_1,\ldots,i_D)} := \binom{n_1}{i_1} \ldots \binom{n_D}{i_D}$$
where
$$ \binom{n}{i} := \frac{n(n-1)\ldots(n-i+1)}{i!}.$$ If $\mathbf{i} = (i_1,\dots,i_d) \in \N^D$ is a $D$-tuple of non-negative integers we define the degree $|\mathbf{i}|  := i_1+\ldots+i_D$. Choose an arbitrary ordering on $\N^D$ with the property that $|\mathbf{i}| \geq |\mathbf{j}|$ whenever $\mathbf{i} \geq \mathbf{j}$.

\begin{lemma}[Taylor expansions]\label{taylor-lem}
Suppose that $g \in \poly(\Z^D, G_{\bullet})$. Then there are unique \emph{Taylor coefficients} $g_{\mathbf{i}} \in G_{|\mathbf{i}|}$ with the property that 
\[ g(\mathbf{n}) = \prod_{\mathbf{i} \in \N^d} g_{\mathbf{i}}^{\binom{\mathbf{n}}{\mathbf{i}}}\]
for all $\mathbf{n} \in \Z^D$. 
Conversely, every Taylor expansion of this type gives rise to a polynomial sequence $g \in \poly(\Z^D, G_{\bullet})$.
\end{lemma}

\emph{Remarks.} This is proven in \cite[Lemma 6.7]{green-tao-nilratner}. Note that, since $G$ is nilpotent, this is a finite expansion. In the case $D = 1$ (which will feature most prominently in the paper) the it takes the form
\[ g(n) = g_0 g_1^{\binom{n}{1}} \dots g_s^{\binom{n}{s}}.\] Note how, from the presentation of polynomial sequences as Taylor expansions, it is by no means clear (and somewhat remarkable) that they form a group under pointwise multiplication (Theorem \ref{ll-thm}).

Polynomial sequences that vary slowly, in a certain sense, are called \emph{smooth}. We employ the following definition, which is the same as the one given in the introduction to \cite{green-tao-nilratner}. 

\begin{definition}[Smooth sequences]\label{smooth-def}
Let  $A$ be a positive parameter and let $N \geq 1$ be an integer. Let $\beta \in \poly(\Z,G_\bullet)$. We say that $\beta$ is \emph{$(A,N)$-smooth} if we have $d_G(\beta(n),\id) \leq A$ and $d_G(\beta(n),\beta(n+1)) \leq A/N$ for all $n \in [N]$.
\end{definition}

Here $d_G$ is a metric on the group $G$ constructed using the Mal'cev basis, see \cite[Definition 2.2]{green-tao-nilratner}.  The precise definition of this metric is not terribly important for our analysis.  

In counterpoint\footnote{One could take an ``adelic'' perspective here and view smooth sequences as those that are local to the Archimedean place $\infty$, while rational sequences are those that are local to finite places $p$.} to the notion of a smooth sequence is that of a \emph{rational} sequence. 

\begin{definition}[Rational sequences]\label{rational-def}
Let $A \geq 1$ be an integer, and let $(G/\Gamma,G_\bullet)$ be a filtered nilmanifold. Then an element $g \in G$ is \emph{$A$-rational} if there is some $q$, $1 \leq q \leq A$, such that $g^q \in \Gamma$. If $\gamma \in \poly(\Z,G_\bullet)$ is a polynomial sequence then we say that it is \emph{$A$-rational} if $\gamma(n)$ is $A$-rational for every integer $n$.
\end{definition}

We have the following basic facts about smooth and rational sequences:

\begin{lemma}[Basic facts]\label{prod-smooth}  Let $(G/\Gamma,G_\bullet)$ be a filtered nilmanifold of complexity $\leq M_0$. By a ``sequence'', we mean an element of $\poly(\Z,G_{\bullet})$. Then:

\textup{(i)} The product of two $(A,N)$-smooth sequences is $O_{M_0,A}(1)$-smooth;

\textup{(ii)} The product of two $A$-rational sequences is $O_{M_0,A}(1)$-rational;
 
\textup{(iii)} Any $A$-rational sequence is periodic with period $O_{M_0,A}(1)$. 
\end{lemma}

\begin{proof} For (i), see \cite[Lemma 10.1]{green-tao-nilratner}; for (ii), see \cite[Lemma A.11 (v)]{green-tao-nilratner}; and for (iii), see \cite[Lemma A.12 (ii)]{green-tao-nilratner}.  In fact these results hold in the multiparameter setting, with polynomially effective bounds, but we will not need these facts here.
\end{proof}

We turn now to an important new definition for this paper, that of an \emph{irrational} polynomial sequence. In \cite{green-tao-nilratner}, much emphasis was placed on the notion of an \emph{equidistributed} polynomial sequence $g : \Z \rightarrow G$: one for which the orbit $(g(n)\Gamma)_{n \in [N]}$ is close to equidistributed on $G/\Gamma$. The notion of an irrational sequence implies equidistribution (see Lemma \ref{irrat-equi}, which is also a special case of Theorem \ref{count-lem}), but also encodes an assertion that the filtration $G_\bullet$ is in some sense ``minimal'' for the sequence. To illustrate the difference, let us think about a simple abelian case in which $G/\Gamma$ is just the unit circle $\R/\Z$ (written additively), and $g: \Z \to \R$ is a polynomial 
\begin{equation}\label{abelian-eq} g(n) = \alpha_0 + \alpha_1 \binom{n}{1} + \ldots + \alpha_s \binom{n}{s}.
\end{equation}
This sequence is adapted to the filtration in which $G_{(i)}=\R$ for $i \leq s$ and $G_{(i)}=\{0\}$ for $i>s$.
Qualitatively speaking, $g$ is equidistributed if at least one of $\alpha_1,\dots,\alpha_s$ is irrational; in contrast, $g$ is irrational with respect to this filtration if it is $\alpha_s$ which is irrational.  Note that if $s > 1$ and $\alpha_s$ is rational, then (after removing the periodic component $\alpha_s n^s$ from $g$) $g$ is now adapted to the filtration $G'_\bullet$ in which $G'_{(i)}=\R$ for $i \leq s-1$ and $G'_{(i)}=\{0\}$ for $i > s-1$, which has a strictly smaller total dimension.  This basic example is the model for the more sophisticated result in Lemma \ref{init}.

Let us turn now to the precise definition in the more general setting of Lie group-valued polynomial sequences, in which the role of the $\alpha_i$ is played by the Taylor coefficients of $g$. We need a preliminary definition.

\begin{definition}[$i$-horizontal characters]\label{def-a4}
Let $(G/\Gamma, G_{\bullet})$ be a filtered nilmanifold of degree $\leq s$ with filtration $G_{\bullet} = (G_{(i)})_{i = 0}^{\infty}$. Then by an \emph{$i$-horizontal character} we mean a continuous homomorphism from $\xi_i : G_{(i)} \rightarrow \R$ which vanishes on $G_{(i+1)}$, and on $[G_{(j)},G_{(i-j)}]$ for any $0 \leq j \leq i$, and which maps $\Gamma_{(i)}$ to $\Z$. We say that such a character is \emph{non-trivial} if it is not constant. We can assign a notion of complexity by taking a Mal'cev basis adapted to $G_{\bullet}$, whereupon one has a natural isomorphism $G_{(i)}/G_{(i+1)} \cong \R^k$. Writing $\psi(g_i)$ for the coordinates of $g_i\md{G_{(i+1)}}$, any $i$-horizontal character has the form $\xi_i(g_i) = \vec{m}. \psi(g_i)$, for some vector $\vec{m} = (m_1,\dots,m_k)$ of integers. We may then define the \emph{complexity} of $\xi_i$ to be $|m_1| + \dots + |m_k|$.
\end{definition}

The list of subgroups on which $\xi_i$ is required to vanish looks rather restrictive and slightly unnatural at first sight.  Roughly speaking, this list is intended to isolate that behaviour which genuinely ``belongs'' to the degree $i$ portion of the filtered nilmanifold, as opposed to arising from those terms of higher or lower degree, or which disappear after quotienting out by the lattice $\Gamma$.

\begin{definition}[Irrationality]\label{irrat-def}  Let $(G/\Gamma,G_\bullet)$ be a filtered nilmanifold of degree $\leq s$ with filtration $G_\bullet = (G_{(i)})_{i=0}^\infty$. Let $g_i \in G_{(i)}$. Let $A, N > 0$.  Then we say that $g_i$ is \emph{$(A,N)$-irrational} in $G_{(i)}$ if for every non-trivial $i$-horizontal character $\xi_i: G_{(i)} \to \R$ of complexity $\leq A$ one has $\| \xi_i(g_i)\|_{\R/\Z} \geq A / N^i$.
We say that the sequence $g(n)$ is \emph{$(A,N)$-irrational} if its $i^{\operatorname{th}}$ Taylor coefficient $g_i$ is $(A,N)$-irrational in $G_{(i)}$ for each $i$, $1 \leq i \leq s$.
 \end{definition}

To understand this definition, it is helpful to consider examples. We leave it as an exercise to check that in the abelian case \eqref{abelian-eq} this amounts to stipulating that the top coefficient of $g$ is poorly approximated by rationals, thus $\Vert q\alpha_s \Vert_{\R/\Z} \geq A'/N^s$ whenever $1 \leq q \leq A'$. 

A second interesting case to examine is that in which $g(n) = g^n$ is a linear polynomial sequence adapted to the lower central series filtration $(G_i)_{i = 0}^{\infty}$. For the lower central series filtration there are no nontrivial $i$-horizontal characters when $i \geq 2$, and $1$-horizontal characters are the same thing as horizontal characters in the sense of \cite[Definition 1.5]{green-tao-nilratner}. It follows from this and \cite[Theorem 1.16]{green-tao-nilratner} that $g(n)$ is irrational if and only if $(g(n)\Gamma)_{n \in [N]}$ is equidistributed. Now polynomial sequences that are not linear do not arise naturally in ergodic-theoretic settings such as those considered in \cite{bergelson-host-kra,leibman-orb-diag}, and thus the equivalence of the notions of ``irrational'' and ``equidistributed'' in this setting explains why the former concept has not appeared in the literature before.  The need for it is a new feature of the quantitative world, as is the need for polynomial nilsequences themselves, for reasons explained on \cite[\S 1]{green-tao-nilratner}.  

The following third example is also edifying. Take $g(n)$ to be any polynomial sequence on the Heisenberg group, for example $g(n) = \left( \begin{smallmatrix} 1 & \alpha n & \gamma n^2 \\ 0 & 1 & \beta n \\ 0 & 0 & 1 \end{smallmatrix}\right)$. This sequence is a polynomial sequence adapted to the lower central series filtration $G_0 = G_1 = G$, $G_2 = [G,G]$, $G_3 = \{\id\}$, and it will be equidistributed in that setting for generic $\alpha,\beta,\gamma$. However $g$ is also a polynomial sequence with respect to some much flabbier filtrations, for example the one in which $G_{(0)} = G_{(1)} = G_{(2)} = \dots = G_{(10)} = G$, $G_{(11)} = \dots = G_{(100)} = [G,G]$ and $G_{(i)} = \{ \id\}$ for $i \geq 101$. It is easy to check that $g$ is \emph{not} irrational in this setting, and indeed irrationality is somehow detecting the fact that a given filtration $G_{\bullet}$ is minimal for $g$. This point is quite clear in the proof of Lemma \ref{init} (which itself depends on Lemma \ref{short-factorisation} below), where the failure of a sequence to be irrational is used to create a coarser filtration for a polynomial sequence related to $g$.

\begin{lemma}\label{short-factorisation}
Suppose that $(G/\Gamma, G_{\bullet})$ is a filtered nilmanifold of degree $\leq s$ with filtration $G_{\bullet} = (G_{(i)})_{i=0}^\infty$. Suppose that $g$ is not $(A,N)$-irrational. Then there is an index $i$, $1 \leq i \leq s$, such that the $i^{\operatorname{th}}$ Taylor coefficient $g_i$ factors as $\beta_i g'_i \gamma_i$, where $\beta_i, g'_i, \gamma_i \in G_{(i)}$, $g'_i$ lies in the kernel of some $i$-horizontal character $\xi_i : G_{(i)} \rightarrow \R$ of complexity at most $A$, $d_G(\beta_i, \id) = O_{A,M}(N^{-i})$ and $\gamma_i$ is $O_{A,M}(1)$-rational.
\end{lemma}

\begin{proof} The proof is (unsurprisingly) extremely similar to that of \cite[Lemma 7.9]{green-tao-nilratner}. Reversing the definition of irrational polynomial sequence, we see that there is an index $i$ together with an $i$-horizontal character $\xi_i : G_{(i)} \rightarrow \R$ such that $\Vert \xi_i (g_i) \Vert_{\R/\Z} \leq A/N^i$. It is convenient at this point to work in a Mal'cev coordinate system adapted to $G_{\bullet}$, whereby $G_{(i)}/G_{(i+1)}$ may be identified with $\R^k$ and $\Gamma_{(i)}/G_{(i+1)}$ with $\Z^k$. If $g_i \in G_{(i)}$ then, as above, we write $\psi(g) \in \R^k$ for the corresponding coordinates. Then $\xi_i$ has the form $\xi_i (g_i) = \vec{m} . \psi(g)$ for some vector $\vec{m} = (m_1,\dots,m_k)$ of integers with $|m_1| + \dots + |m_k| \leq A$. Now by assumption we have $\Vert \vec{m} . \psi(g_i) \Vert_{\R/\Z} \leq A/N^i$, and therefore $\vec{m} . \psi(g_i) = r + O(A/N^i)$ for some integer $r$. It follows from simple linear algebra that we may write $\psi(g_i) = \vec{t} + \vec{u} + \vec{v}$, where $\vec{m} . \vec{u} = 0$, the coordinates of $\vec{v}$ lie in $\frac{1}{Q}\Z$ for some $Q = O_{A}(1)$ and each coordinate of $\vec{t}$ is $O_{A}(1/N^i)$. Now choose $\beta_i \in G_{(i)}$ in such a way that $\psi(\beta_i) = \vec{t}$ and $d_G(\beta_i, \id) = O_{A,M}(1/N_i)$, choose an $O_{A,M}(1)$-rational element $\gamma_i \in G_{(i)}$ with $\psi(\gamma_i) = \vec{v}$, and finally choose $g'_i$ so that $g_i = \beta_i g'_i \gamma_i$. Then one automatically has $\psi(g'_i) = \vec{u}$, which means that $g'_i$ lies in the kernel of the $i$-homomorphism $\xi_i$.\end{proof}

Finally, we record a convenient scaling lemma.

\begin{lemma}[Scaling lemma]\label{scaling-lemma}  Let $(G/\Gamma,G_\bullet)$ be a filtered nilmanifold of complexity $\leq M$.
If $g \in \poly(\Z,G_\bullet)$ is $(A,N)$-irrational, $r \in [-N,N]$, and $1 \leq q \leq M$, then the sequence $n \mapsto g(nq+r)$ is $(\gg_{M,\eps} A, \eps N)$-irrational for any $\eps > 0$.
\end{lemma}

\begin{proof}  We need to show that the $i^{\operatorname{th}}$ Taylor coefficient of $n \mapsto g(nq+r)$ is $(\gg_{M,\eps} A, \eps N)$-irrational for each $i \geq 0$.  Note that we may assume $i \leq M$ since the filtered manifold has degree $\leq M$.

Fix $i$.  We may quotient out the nilmanifold by the normal subgroups $G_{(i+1)}$ and $[G_{(j)},G_{(i-j)}]$ for $0 \leq j \leq i$, since these do not affect the irrationality of the $i^{\operatorname{th}}$ coefficient.  We may then expand $g$ as a Taylor series
$$ g(n) = \prod_{j=0}^i g_j^{\binom{n}{j}},$$
and thus
$$ g(qn+r) = \prod_{j=0}^i g_j^{\binom{qn+r}{j}}.$$
Expanding out the binomial coefficient and using many applications of the Baker-Campbell-Hausdorff formula, we obtain
$$ g(qn+r) = (\prod_{j=0}^{i-1} (g'_j)^{\binom{n}{j}}) g_i^{q^i \binom{n}{i}}$$
for some $g'_j \in G_{(j)}$; the point being that the Baker-Campbell-Hausdorff term cannot generate any terms involving polynomials in $n$ of degree $i$ or higher due to the fact that the groups $G_{(i+1)}$ and $[G_{(j)},G_{(i-j)}]$ have been quotiented out.  As a consequence, we see that the $i^{\operatorname{th}}$ Taylor coefficient of $n \mapsto g(qn+r)$ is $q^i g_i$, and the claim is easily verified.
\end{proof}

\section{A multiparameter equidistribution result}\label{appendix-b}

The purpose of this appendix is to prove Theorem \ref{mpet}, which we recall here again.

\begin{mpet-rpt}
Suppose that $(G/\Gamma, G_{\bullet})$ is a filtered nilmanifold of complexity $\leq M$ and that $g \in \poly(\Z^D, G_{\bullet})$ is a polynomial sequence for some $D \leq M$. Suppose that $\Lambda \subseteq \Z^D$ is a lattice of index $\leq M$, that $\mathbf{n_0} \in \Z^D$ has magnitude $\leq M$, and that $P \subseteq [-N,N]^D$ is a convex body. Suppose that $\delta > 0$, and that 
$$ \big| \sum_{{\mathbf n} \in (\mathbf{n_0}+\Lambda) \cap P} F(g({\mathbf{n}})\Gamma) - \frac{\operatorname{vol}(P)}{[\Z^D:\Lambda]} \int_{G/\Gamma} F \big| > \delta N^D \|F\|_{\operatorname{Lip}}$$
for some Lipschitz function $F: G/\Gamma \to \C$.
Then there is a nontrivial homomorphism $\eta : G \rightarrow \R$ which vanishes on $\Gamma$, has complexity $O_M(1)$ and such that 
\[ \Vert \eta \circ g \Vert_{C^{\infty}([N]^D)} = O_{\delta,M}(1).\]
\end{mpet-rpt}

Recall from \cite[Definition 8.2]{green-tao-nilratner} that the norm $\Vert g\Vert_{C^\infty([N]^D)}$ of a polynomial sequence $g: [N]^D \to \R$ is given by the formula
$$ \Vert g\Vert_{C^\infty([N]^D)} = \sup_{{\mathbf i}\in \N^D} N^{-|{\mathbf i}|} \|g_{\mathbf i}\|_{\R/\Z}$$
where $g_{\mathbf i}$ are the Taylor coefficients of $g$, thus
$$ g({\mathbf n}) = \sum_{{\mathbf i}\in \N^D} \binom{\mathbf{n}}{\mathbf{i}} g_{\mathbf i}.$$

We now prove the theorem, allowing all implied constants to depend on $\delta$ and $M$.  We may assume that $N$ is sufficiently large depending on $\delta,M$, since the claim is trivial otherwise.  A simple volume packing argument (using \cite[Corollary A.2]{green-tao-linearprimes}, for example, to control the boundary terms) shows that
$$ |(\mathbf{n_0}+\Lambda) \cap P| = \frac{\operatorname{vol}(P)}{[\Z^D:\Lambda]} + o_{N \to \infty}(N^D).$$
As a consequence, for $N$ large enough we may subtract off the mean of $F$ and normalise $F$ to have Lipschitz norm $1$ and mean zero, thus
$$ \big| \sum_{{\mathbf n} \in (\mathbf{n_0}+\Lambda) \cap P} F(g({\mathbf{n}})\Gamma) \big| \gg N^D.$$
As $\Lambda$ has index $\leq M$ in $\Z^D$, it contains the sublattice $q\Z^D$ for some positive integer $q = O(1)$.  By the pigeonhole principle, we may thus find $\mathbf{n_1} \in \Z^D$ of magnitude $O(1)$ such that
$$ \big| \sum_{{\mathbf n} \in (\mathbf{n_1}+q\Z^D) \cap P} F(g({\mathbf{n}})\Gamma) \big| \gg N^D,$$
and thus
$$ \big| \sum_{{\mathbf n} \in \Z^D \cap P'} F(g(q{\mathbf{n}}+{\mathbf n_1})\Gamma) \big| \gg N^D.$$
for some convex body $P'$ contains in a ball of radius $O(N)$ centered at the origin.

By subdividing $P'$ into cubes of sidelength $\eps N$ for some sufficiently small $\eps > 0$ (and again using \cite[Corollary A.2]{green-tao-linearprimes} to control the boundary terms), and then applying the pigeonhole principle, we see that
$$ \big| \sum_{{\mathbf n} \in \Z^D \cap {\mathbf n_2} + [\eps N]^D} F(g(q{\mathbf{n}}+{\mathbf n_1})\Gamma) \big| \gg N^D$$
for some $\eps \gg 1$ and ${\mathbf n_2} = O(N)$.  We can rearrange this as
$$ \big| \sum_{{\mathbf n} \in \Z^D \cap [\eps N]^D} F(g(q{\mathbf{n}}+{\mathbf n_3})\Gamma) \big| \gg N^D$$
for some ${\mathbf n_3}=O(N)$.

We may now invoke \cite[Theorem 8.6]{green-tao-nilratner} to conclude that there exists a nontrivial homomorphism $\eta : G \rightarrow \R$ which vanishes on $\Gamma$, has complexity $O(1)$ and such that 
\[ \Vert \eta \circ g(q \cdot + {\mathbf n_3}) \Vert_{C^{\infty}([N]^D)} \ll 1.\]
Applying \cite[Lemma 8.4]{green-tao-nilratner} we conclude that
\[ \Vert Q \eta \circ g(\cdot + {\mathbf n_3}) \Vert_{C^{\infty}([N]^D)} \ll 1\]
for some non-negative integer $Q = O(1)$.  Shifting the Taylor expansion by ${\mathbf n_3}$, we conclude that
\[ \Vert Q \eta \circ g \Vert_{C^{\infty}([N]^D)} \ll 1.\]
The claim follows (with $\eta$ replaced by $Q\eta$).

\section{The Baker-Campbell-Hausdorff formula}\label{bch}

Let $G$ be a connected, simply connected nilpotent Lie group, and let $\exp: {\mathfrak g} \to G$ and $\log: G \to {\mathfrak g}$ be the associated exponential and logarithm maps between $G$ and its Lie algebra ${\mathfrak g}$.   
The \emph{Baker-Campbell-Hausdorff formula} asserts that
$$
\exp( X_1 ) \exp( X_2 ) = \exp( X_1 + X_2 + \frac{1}{2} [X_1,X_2] + \prod_\alpha c_\alpha X_\alpha )$$
for any $X_1,X_2$, where $\alpha$ is a finite set of labels, $c_\alpha$ are real constants, and $X_\alpha$ are an iterated Lie bracket of $k_1 = k_{1,\alpha}$ copies of $X_1$ and $k_2 = k_{2,\alpha}$ copies of $X_2$ where $k_1,k_2 \geq 1$ and $k_1+k_2 \geq 2$.

Using this formula, it is a routine matter to see that for any $g_1,g_2 \in G$ and $x \in \R$, we have 
\begin{equation}\label{bch-2}
(g_1 g_2)^x = g_1^x g_2^x \prod_\alpha g_\alpha^{Q_\alpha(x)}
\end{equation}
where $\alpha$ is a finite set of labels, each $g_\alpha$ is an iterated of $k_1 = k_{1,\alpha}$ copies of $g_1$ and $k_2 = k_{2,\alpha}$ copies of $g_2$ where $k_1,k_2 \geq 1$ and $k_1+k_2 \geq 2$, and the $Q_\alpha: \R \to \R$ are polynomials of degree at most $k_1+k_2$ with no constant term.

In a similar vein, for any $g_1, g_2 \in G$ and $x_1,x_2 \in \R$, we have the formula
\begin{equation}\label{comm-form}
[g_1^{x_1}, g_2^{x_2}] = [g_1,g_2]^{x_1 x_2} \prod_\alpha g_\alpha^{P_\alpha(x_1,x_2)}
\end{equation}
where $\alpha$ is a finite set of labels, each $g_\alpha$ is an iterated commutator of $k_1 = k_{1,\alpha}$ copies of $g_1$ and $k_2 = k_{2,\alpha}$ copies of $g_2$ where $k_1,k_2 \geq 1$ and $k_1+k_2 \geq 3$, and the $P_\alpha: \R \times \R \to \R$ are polynomials of degree at most $k_1$ in $x_1$ and at most $k_2$ in $x_2$ which vanish when $x_1=0$ or $x_2=0$.

\section{True complexity and Cauchy-Schwarz complexity}\label{complexity-app}

This appendix (which does not appear in the published version of the paper, and which for some reason seems to be absent from the literature) relates the Gowers-Wolf and Cauchy-Schwarz notions of complexity. 

Let $\Psi = (\psi_1,\dots, \psi_t)$ be a collection of linear forms $\psi_i : \Z^D \rightarrow \Z$.  Recall (see Section \ref{gvn-sec}) that the \emph{Cauchy-Schwarz complexity} $s(\Psi)$ is the smallest $s$ such that the following is true. For every $i \in [t]$, one may partition the $t-1$ forms $\{\psi_j : j \in [t] \setminus \{i\}\}$, into $s+1$ classes such that $\psi_i$ does not lie in the linear span (over $\Q$) of the forms in any one class. 

Let us also define the \emph{Gowers-Wolf complexity} $\gw(\Psi)$ to be the smallest $s$ such that the forms $\psi_1^{s+1},\dots, \psi_t^{s+1}$ are linearly independent (cf. the statement of Theorem \ref{gwolf}). 

\begin{proposition}\label{gw-cs}
We have $\gw(\Psi) \leq s(\Psi)$.
\end{proposition}
\begin{proof}
Let $m \geq 1$. We need to show that if $\psi_1^{m},\dots, \psi_t^{m}$ are linearly dependent then, for some $i \in [t]$, the following is true. No matter how we partition the $t-1$ forms $\{\psi_j : j \in [t] \setminus \{i\}\}$ into $m$ classes, $\psi_i$ lies in the linear span of the forms in one of the classes. Let the linear relation be
\[ \lambda_1 \psi_1^{m} + \dots + \lambda_t \psi_t^{m} = 0,\] and suppose without loss of generality that $\lambda_1 \neq 0$. Suppose the forms $\{ \psi_2,\dots, \psi_t\}$ are partitioned into $m$ classes $S_1,\dots, S_m$; we will show that $\psi_1$ lies in the linear span of the $\{ \psi_j : j \in S_k\}$ for some $k \in [m]$.

For any linear form $\psi : \Q^D \rightarrow \Q$ we have the depolarisation identity \[ \psi(\mathbf{x}_1) \ldots \psi(\mathbf{x}_m) = \frac{(-1)^m}{m!} \sum_{\omega \in \{0,1\}^m} (-1)^{|\omega|} \psi(\omega_1 \mathbf{x}_1 + \ldots + \omega_m \mathbf{x}_m)^m\]
where $\omega = (\omega_1,\ldots,\omega_m)$ and $|\omega| := \omega_1+\ldots+\omega_m$. It follows that 
\begin{equation}\label{depolarised-new}  \lambda_1 (\psi_1(\mathbf{x}_1) \cdots \psi_1(\mathbf{x}_m)) + \dots + \lambda_t (\psi_1(\mathbf{x}_1) \cdots \psi_t(\mathbf{x}_m)) = 0,\end{equation} for any $\mathbf{x}_1,\dots, \mathbf{x}_m \in \Q^D$.
Let $k \in [m]$, and suppose that $\psi_1$ is not in the linear span of the $\{ \psi_j : j \in S_k\}$. Then we may choose $\mathbf{x}_k \in \Q^D$ such that $\psi_j(\mathbf{x}_k) = 0$ for all $j \in S_k$, but $\psi_1(\mathbf{x}_k) = 1$. Substituting into \eqref{depolarised-new} gives $\lambda_1 = 0$, contrary to assumption.
\end{proof}

\providecommand{\bysame}{\leavevmode\hbox to3em{\hrulefill}\thinspace}


\begin{thebibliography}{99}

\bibitem{afks}
N.~Alon, E.~Fischer, M.~Krivelevich, B.~Szegedy, \emph{Efficient
testing of large graphs}, Proc. of 40th FOCS, New York, NY, IEEE (1999), 656--666. Also: Combinatorica \textbf{20} (2000), 451--476.

\bibitem{austin-multdimsz}
T.~Austin, \emph{Deducing the multidimensional Szemer\'edi Theorem from an infinitary removal lemma}, preprint.

\bibitem{austin-dhj}
\bysame, \emph{Deducing the Density Hales-Jewett Theorem from an infinitary removal lemma}, preprint.

\bibitem{bergelson-host-kra}
V.~Bergelson, B.~Host and B.~Kra, \emph{Multiple recurrence and nilsequences}, with an appendix by Imre Ruzsa, Invent. Math. \textbf{160} (2005), no. 2, 261--303.

\bibitem{leib-orbits}
V.~Bergelson, 	A.~Leibman and E.~Lesigne, \emph{Weyl complexity of a system of polynomials and constructions in combinatorial number theory}, J. D'Analyse Math\'ematique 103 (2007), 47--92. 

\bibitem{bourgain}
J.~Bourgain, \emph{A Szemer\'edi type theorem for sets of positive density in $\R^k$}, {Israel J. Math.} \textbf{54} (1986), no. 3, 307--316.

\bibitem{bourgain-triples}
J.~Bourgain, \emph{On triples in arithmetic progression}, {GAFA } \textbf{9} (1999), 968--984.

\bibitem{chung}
F.~Chung, \emph{Regularity lemmas for hypergraphs and quasi-randomness}, Random Struct. Alg. \textbf{2} (1991), 241--252.

\bibitem{furstenberg-book}
H.~Furstenberg, {\it Recurrence in Ergodic theory and Combinatorial Number Theory}, Princeton University Press, Princeton NJ 1981.

\bibitem{furst} \bysame,
\emph{Ergodic behavior of diagonal measures and a theorem of Szemer\'edi on arithmetic progressions},
J. Analyse Math. \textbf{31} (1977), 204--256.

\bibitem{fk}
H.~Furstenberg and Y.~Katznelson, \emph{A density version of the Hales-Jewett theorem}, J. d'Analyse Math. \textbf{57} (1991), 64--119.

\bibitem{fko} H.~Furstenberg, Y.~Katznelson and D.~Ornstein,
\emph{The ergodic theoretical proof of Szemer\'edi's theorem},
Bull. Amer. Math. Soc. (N.S.) \textbf{7} (1982), no. 3, 527--552.

\bibitem{frieze}
A.~Frieze and R.~Kannan, \emph{Quick approximation to matrices
and applications}, Combinatorica \textbf{19} (1999), no. 2,
175--220.

\bibitem{gowers-lower}
W.~T.~Gowers, \emph{Lower bounds of tower type for Szemer\'edi's uniformity lemma}, GAFA \textbf{7} (1997), 322--337.

\bibitem{gowers-4aps} 
\bysame, \emph{A new proof of Szemer\'edi's theorem for progressions of length four,} GAFA \textbf{8} (1998), 
no. 3, 529--551.

\bibitem{gowers-longaps} 
\bysame, \emph{A new proof of Szemer\'edi's theorem,} GAFA \textbf{11} (2001), 465--588. 

\bibitem{gowers-hypergraph}
\bysame, \emph{Hypergraph regularity and the multidimensional Szemer\'edi theorem}, Ann. of Math. (2) \textbf{166} (2007), no. 3, 897--946.

\bibitem{gowers-hyper-4}
\bysame, \emph{Quasirandomness, counting, and regularity for $3$-uniform hypergraphs}, Combin. Probab. Comput. \textbf{15} (2006), no. 1-2, 143--184.

\bibitem{gowers-regularity}
\bysame, \emph{Decompositions, approximate structure, transference, and the Hahn-Banach theorem}, preprint.

\bibitem{gowers-wolf-1}
W.~T.~Gowers, J. Wolf, \emph{The true complexity of a system of linear equations}, preprint.

\bibitem{gowers-wolf-2}
\bysame, \emph{Linear forms and uniformity for functions on $\Z/N\Z$}, preprint.

\bibitem{gowers-wolf-3}
\bysame, \emph{Linear forms and uniformity for functions on $\F_p^n$}, preprint.

\bibitem{green-regularity}
B.~J.~Green, \emph{A Szemer\'edi-type regularity lemma in abelian groups}, GAFA \textbf{15} (2005), no. 2, 340--376.

\bibitem{green-montreal}
\bysame, \emph{Montr\'eal lecture notes on quadratic Fourier analysis}, Additive Combinatorics (Montr\'eal 2006, ed. Granville et al.), CRM Proceedings vol. 43, 69--102, AMS 2007.

\bibitem{green-tao-longprimeaps} B.~J.~Green and T.~C.~Tao, \emph{The primes contain arbitrarily long arithmetic progressions}, to appear in Annals of Math.

\bibitem{green-tao-u3inverse}
\bysame, \emph{An inverse theorem for the Gowers $U^3(G)$-norm}, Proc. Edin. Math. Soc. \textbf{51} (2008), 73--153.

\bibitem{green-tao-r4fn}
\bysame, \emph{New bounds for Szemer\'edi's Theorem, I: Progressions of length $4$ in finite field geometries}, Proc. Lond. Math. Soc. \textbf{98} (2009), 365--392.

\bibitem{green-tao-nilratner} \bysame, \emph{The quantitative behaviour of polynomial orbits on nilmanifolds,} preprint available at arxiv:0709.3562.

\bibitem{green-tao-linearprimes} 
\bysame, \emph{Linear equations in primes,} to appear in Ann. Math.

\bibitem{green-tao-ukmobius} \bysame, \emph{The M\"obius function is strongly orthogonal to nilsequences,} preprint available at arxiv:0807.1736.


\bibitem{green-tao-ziegler-u4inverse}
B.~J.~Green, T.~C.~Tao and T. Ziegler, \emph{An inverse theorem for the Gowers $U^4$-norm}, submitted.

\bibitem{uk-inverse}
\bysame, \emph{An inverse theorem for the Gowers $U^k$-norm}, submitted.

\bibitem{host-kra} B.~Host and B.~Kra, \emph{Nonconventional ergodic averages and nilmanifolds,} Ann. of Math. (2) \textbf{161}
(2005), no. 1, 397--488.

\bibitem{komlos}
J. Koml\'os and M.~Simonovits, 
\emph{Szemer\'edi's regularity lemma and its applications in graph theory},
Combinatorics, Paul Erd\H{o}s is eighty, Vol. 2 (Keszthely, 1993), 295--352, 
Bolyai Soc. Math. Stud., 2, J\'anos Bolyai Math. Soc., Budapest, 1996.

\bibitem{leibman-group-1}
A.~Leibman, \emph{Polynomial sequences in groups}, Journal of Algebra 201 (1998), 189--206.

\bibitem{leibman-group-2} \bysame, \emph{Polynomial mappings of groups,}
Israel J. Math. \textbf{129} (2002), 29--60.

\bibitem{leib}
\bysame, \emph{Pointwise convergence of ergodic averages of polynomial sequences of translations on a nilmanifold}, Ergodic Theory and Dynamical Systems \textbf{25} (2005), no. 1, 201--213.

\bibitem{leibman-multi-poly} \bysame, \emph{Pointwise convergence of ergodic averages for polynomial actions of $\Z^d$ by translations on a nilmanifold,} Ergodic Theory and Dynamical Systems \textbf{25} (2005), no. 1, 215--225.

\bibitem{leibman-orb-diag} \bysame, \emph{Orbit of the diagonal of the power of a nilmanifold,} preprint.

\bibitem{lovasz}
L. Lov\'asz and B. Szegedy, \emph{Szemer\'edi's Lemma for the analyst}, Geom. Func. Anal. \textbf{17} (2007), 252--270. 

\bibitem{nagle-rodl-schacht1} 
B.~Nagle, V.~R\"odl and M.~Schacht, \emph{A short proof of the 3-graph counting lemma,} preprint.

\bibitem{nagle-rodl-schacht} 
\bysame, \emph{The counting lemma for regular $k$-uniform hypergraphs}, Random
Structures and Algorithms, to appear.

\bibitem{polymath}
D.H.J.~Polymath, \emph{A new proof of the density Hales-Jewett theorem}, preprint.

\bibitem{rttv}
O.~Reingold, L.~Trevisan, M.~Tulsiani and S.~Vadhan, \emph{New Proofs of the Green-Tao-Ziegler Dense Model Theorem: An Exposition}, preprint.

\bibitem{szemeredi-4}
E.~Szemer\'edi, \emph{On sets of integers containing no four elements in arithmetic progression},
Acta Math. Acad. Sci. Hungar. \textbf{20} (1969), 89--104.

\bibitem{szemeredi-aps}
\bysame, \emph{On sets of integers containing no $k$ elements in arithmetic progression},
Acta Arith. \textbf{27} (1975), 299--345.

\bibitem{szemeredi-reg}
\bysame, \emph{Regular partitions of graphs}, in ``Probl\`emes Combinatoires et Th\'eorie des Graphes, Proc. Colloque Inter. CNRS,'' (Bermond, Fournier, Las Vergnas, Sotteau, eds.), CNRS Paris, 1978, 399--401.

\bibitem{tao-quant-ergodic}
T.~C.~Tao, \emph{A quantitative ergodic theory proof of Szemer\'edi's theorem}, Electron. J. Combin. \textbf{13} (2006). 1 No. 99, 1--49.

\bibitem{tao-icm}
\bysame, \emph{The dichotomy between structure and randomness, arithmetic progressions, and the primes}, 2006 ICM proceedings, Vol. I., 581--608.

\bibitem{tao-hypergraph}
\bysame, \emph{A variant of the hypergraph removal lemma}, J. Combin. Thy. A \textbf{113} (2006), 1257--1280.

\bibitem{tao-revisit}
\bysame, \emph{Szemer\'edi's regularity lemma revisited}, Contrib. Discrete Math. \textbf{1} (2006), 8--28.

\bibitem{tao-focs}
\bysame, \emph{Structure and randomness in combinatorics}, Proceedings of the 48th annual symposium on Foundations of Computer Science (FOCS) 2007, 3--18.

\bibitem{tao-vu}
T.~C.~Tao and V.~H.~Vu, Additive Combinatorics, Cambridge University Press, 2006.

\bibitem{var}
P.~Varnavides, \emph{On certain sets of positive density}, {J. London Math. Soc.}  \textbf{39} (1959), 358--360.

\bibitem{ziegler} T.~Ziegler, \emph{Universal Characteristic Factors and Furstenberg Averages,} J. Amer. Math. Soc. \textbf{20} (2007), 53--97.


\end{thebibliography}
\end{document}